\documentclass[12pt]{amsart}
\usepackage{amsmath}
\usepackage{amsmath,mathtools}
\usepackage{imakeidx}
\makeindex

\usepackage{hyperref}
\usepackage{graphicx, wrapfig}
\usepackage{txfonts}
\usepackage{ulem, cancel}
\usepackage[usenames]{color}

\counterwithin{figure}{section}
\title[Distances between Space-Times]{Introducing Various Notions of Distances between Space-Times}

\author[Sakovich]{A Sakovich}
\author[Sormani]{C Sormani}

\thanks{The authors began this research while in residence at the Mathematical Sciences Research Institute (MSRI). 
 Sakovich was partly funded by the Swedish Research Council's grants dnr. 2016-04511 and 2024-04845. The research was also funded in part by Sormani's PSC-CUNY and NSF DMS-1612409 grants. 
}

\theoremstyle{plain}

\newtheorem{thm}{Theorem}[section]
\newcommand{\bt}{\begin{thm}}
\newcommand{\et}{\end{thm}}

\newtheorem{remark}[thm]{Remark}

\newtheorem{ex}[thm]{Example}

\newtheorem{cor}[thm]{Corollary}   

\newtheorem{quest}[thm]{Question}

\newcommand{\bc}{\begin{cor}}
\newcommand{\ec}{\end{cor}}

\newtheorem{lem}[thm]{Lemma}   

\newcommand{\bl}{\begin{lem}}
\newcommand{\el}{\end{lem}}

\newtheorem{prop}[thm]{Proposition}
\newcommand{\bp}{\begin{prop}}
\newcommand{\ep}{\end{prop}}

\newtheorem{defn}[thm]{Definition}
\newtheorem{conj}[thm]{Conjecture}

\newcommand{\ben}{\begin{itemize}}
\newcommand{\een}{\end{itemize}}

\newcommand{\bd}{\begin{defn}}    
\newcommand{\ed}{\end{defn}}

\newtheorem{rmrk}[thm]{Remark}   

\newcommand{\br}{\begin{rmrk}}
\newcommand{\er}{\end{rmrk}}

\newcommand{\ptGHto}{\stackrel { \textrm{pt-GH}}{\longrightarrow} }

\newcommand{\Sdistto}{\stackrel {{S-dist}}{\longrightarrow} }
\newcommand{\distto}{\stackrel {{dist}}{\longrightarrow} }

\newcommand{\Ln}{\operatorname{log}}

\newcommand{\be}{\begin{equation}}

\newcommand{\ee}{\end{equation}}

\newcommand{\diam}{\operatorname{diam}}

\newcommand{\set}{\rm{set}}

\newcommand{\mass}{{\mathbf M}}

\newcommand{\area}{\operatorname{Area}}

\newcommand{\spt}{\operatorname{spt}}

\def\iff{\Longleftrightarrow}
\def\implies{\Longrightarrow}

\DeclareMathOperator{\tr}{tr}

\begin{document}

\begin{abstract}  
We introduce the notion of causally-null-compactifiable space-times which can be canonically converted into  compact timed-metric-spaces using the cosmological time 
of Andersson-Galloway-Howard and the null distance of Sormani-Vega.  We produce a large class of such space-times including future developments of compact initial data sets and regions which exhaust asymptotically flat space-times.   We then present various notions of intrinsic distances between these space-times (introducing the timed-Hausdorff distance)
and prove some of these notions of distance are 
definite in the sense that they equal zero iff there is a time-oriented Lorentzian isometry between the space-times.   These definite distances enable us to define various notions of convergence of space-times to limit space-times which are not necessarily smooth.  Many open questions and conjectures are included throughout.
\end{abstract}

\maketitle

\vspace{-.2cm}
\begin{center}
{\bf Dedicated to Yvonne Choquet-Bruhat.}
\\
Thank you for your inspiring and beautiful work.
\end{center}
\newpage

\tableofcontents

\section{{\bf{Introduction}}}\label{sec:intro}

When physicists study isolated regions within the universe evolving from an initial data set or the evolution of the entire universe from a big bang, they often study a simplified smoother space-time without black holes and then assume that the actual universe exterior to the black holes is close in some sense to this simplified space-time.  It is natural to ask 
\begin{itemize}
\item
In what sense is an almost isotropic  universe close to an isotropic homogeneous FLRW space-time?
\item
In what sense is an isolated gravitational system with almost no mass close to Minkowski space?
\item
In what sense does an isolated gravitational system evolve towards a Kerr space-time?
\end{itemize}
In this paper we introduce various notions for an intrinsic distance between a pair of space-times that can be applied to study convergence of space-times and stability of solutions to the Einstein equations.
 We need to define their closeness even when the space-times are not diffeomorphic if we wish to allow for the possible existence of black holes.
 We define these distances in an  intrinsic way because the space-times do not lie within a common extrinsic space-time. 

There are a variety of notions of intrinsic distances between distinct compact Riemannian manifolds which do not require that they be diffeomorphic.   Gromov introduced the idea of converting a pair of compact Riemannian manifolds, $(M_j,h_j)$, into compact metric-spaces, $(M_j, d_{h_j})$,  and then taking the Gromov-Hausdorff ($GH$) distance between them.  This notion of $GH$ distance between Riemannian manifolds is definite in the sense that the $GH$ distance is only $0$ when there is a distance preserving bijection, $F:M_1\to M_2$, which is thus also a Riemannian isometry, $F^*g_2=g_1$.   

Other definite intrinsic distances between pairs of compact Riemannian manifolds have been defined using a
similar conversion process but also keeping track of volume measures by Sturm and Lott-Villani using earlier work of Fukaya and Cheeger-Colding 
\cite{ChCo-PartI}
\cite{Fukaya-collapsing}\cite{Lott-Villani-09}
\cite{Sturm-2006-I}. 
Sormani-Wenger 
defined the intrinsic flat ($\mathcal F$) distance between compact oriented Riemannian manifolds by converting them into metric-spaces with integral current structures and proved this notion is definite using earlier work of Ambrosio-Kirchheim \cite{SW-JDG}
\cite{AK}. 
We review some of these notions of intrinsic distances between compact Riemannian manifolds and metric-spaces with various structures in Subsections~\ref{sect:back-GH}-\ref{sect:back-Kuratowski}.

Given a pair of space-times, $(N_j,g_j)$, we will convert them into compact metric-spaces and then apply the various notions of intrinsic distances between metric-spaces mentioned above.  This takes a few steps.   Although the Lorentzian distance is not defined between all pairs of points, it can be used to uniquely define cosmological time functions, 
\be
\tau_j=\tau_{g_j}: N_j \to [0,\infty]
\ee
as seen in work of Andersson-Galloway-Howard \cite{AGH} (see also work of Wald-Yip \cite{WaldYip}). When the cosmological time function is regular, the null distance, 
$
\hat{d}_{g_j}=\hat{d}_{\tau_{g_j}},
$
defined by Sormani-Vega in \cite{SV-Null}, can be applied to canonically convert the space-times  into complete metric-spaces, $\left(\bar{N}_j, \hat{d}_{g_j} \right)$. 
See Theorem~\ref{thm:null-canonical} within the review of these concepts in Subsections~\ref{sect:back-L}-\ref{sect:back-Null} for details. 

In this paper we will apply the null distance to define various notions of definite intrinsic distances between space-times which we believe will prove useful especially when combined with these alternate approaches.

\begin{defn}\label{defn:ncst}
We will say that a space-time, $(N,g)$, is a {\bf null compactifiable space-time} if its cosmological time function,
$\tau_g:N\to [0,\tau_{max}]$ is regular and bounded, and its {\bf associated metric-space},
$\left(\bar{N}, \hat{d}_{g} \right)$, is compact.  The associated compact timed-metric-space, $\left(\bar{N}, \hat{d}_{g}, \tau_g \right)$, may also be endowed with a measure or an integral current structure.
\end{defn}

Allen, Burtscher, Graf, Garci\'a-Heveling, Kunzinger,
Sakovich, Sormani, Steinbauer,
and Vega
have studied sequences of such
space-times, $(\bar{N}_j,g_j,\tau_j)$, by studying the $GH$ and $\mathcal{F}$ convergence of their associated compact metric-spaces, $(\bar{N}_j,\hat{d}_{g_j})$ defined using the null distance as above in
\cite{Allen-Null}
\cite{Allen-Burtscher-22} 
\cite{Burtscher-Garcia-Heveling-Global} 
\cite{Graf-Sormani}
\cite{Kunzinger-Steinbauer-22}
\cite{SakSor-SIF}
\cite{Sor-Ober-18}
\cite{SV-Null}
\cite{SV-BigBang}.  
Here we introduce stronger notions of convergence
which also keep track of time and thus the causal structure of the space-times.

When a space-time, $(N,g)$, has a regular cosmological time function, $\tau_g$, which is also proper then its cosmic strips, 
\be
N_{s,t}=\tau_g^{-1}(s,t),
\ee
are null compactifiable space-times. In \cite{SakSor-Null} Sakovich and Sormani proved that, 
in this setting, the null distance of the cosmological time function encodes causality:
\be\label{eq:causal}
p\in J^+(q) \iff \tau_{g}(p)-\tau_g(q)=\hat{d}_{g}(p,q).
\ee
Burtscher, Garc\'ia-Heveling, and Galloway extended this result to larger classes of space-times
\cite{Burtscher-Garcia-Heveling-Global}\cite{Galloway-causal}.  See the review in Subsection~\ref{sect:back-C}.   In this paper we will consider the full class of possible space-times and regions within space-times that have this property by introducing the following notion:

\begin{defn}\label{defn:cncst}
A space-time, $(N,g)$, is a {\bf causally-null-compactifiable space-time} if it has a bounded regular cosmological time function whose null distance encodes causality, so that it has an {\bf associated compact timed-metric-space}, 
$(\bar{N},\hat{d}_{g},\tau_{g})$, which is a compact metric-space,
$(\bar{N},\hat{d}_{g})$, 
endowed with Lipschitz one time function, 
$\tau_{g}:\bar{N} \to [\tau_{min}, \tau_{max}]$, that recovers the original causal structure as in (\ref{eq:causal}).
The associated compact timed-metric-space may also be endowed with a measure or an integral current structure.
\end{defn}

Sakovich-Sormani proved that a bijection, $F:N_1\to N_2$, between two 
causally-null-compactifiable space-times
preserves distance 
\be
\hat{d}_{g_2}(F(p),F(q))=
\hat{d}_{g_1}(p,q) \quad \forall \, p,q\in N_1, 
\ee
and preserves time
\be
\tau_2(F(p))=\tau_1(p)\quad \forall \, p \in N_1,
\ee
if and only if this bijection is also a Lorentzian isometry, $F^*g_2=g_1$, that is time-oriented (Thm 1.3 in \cite{SakSor-Null}).  Thus, it is natural for us to consider two compact timed-metric-spaces to be the same space-time iff there is a space and time preserving bijection.  We can endow these compact timed-metric-spaces with a causal structure using (\ref{eq:causal}) and view them as a class of non-smooth space-times.  

In Section~\ref{sect:space-times}, we provide many classes of causally-null-compactifiable space-times including big bang space-times and future developments from compact initial data sets, as well as regions which exhaust asymptotically flat space-times.  We begin with
Section~\ref{sect:space-times-induced} where we present the explicit construction of the associated metric space using the null distance (see Theorem~\ref{thm:induced}).  In
Section~\ref{sect:space-times-proper} we apply this theorem combined with Sakovich-Sormani's encoding causality theorem of \cite{SakSor-Null} to prove Theorem~\ref{thm:proper} which states that if $(N,g)$ has a proper regular cosmological time function, then the cosmic strip is a causally-null-compactifiable space-time.   In Section~\ref{sect:space-times-big-bang} we briefly review motivating unpublished work on big bang space-times by Sormani-Vega
\cite{SV-BigBang}.  In Section~\ref{sect:space-times-initial-data} we prove certain classes of future developments from compact initial data sets are causally-null-compactifiable space-times (see Theorem~\ref{thm:FD}).

Naturally one may wish to study unbounded space-times.  In Section~\ref{sect:space-times-strip} we consider space-times which do not have finite cosmological time functions because the past extends indefinitely.  We introduce {\em local cosmological time functions} (see Definition~\ref{defn:local-cosmo}) and prove 
that, when proper and regular, their cosmic strips are causally-null-compactifiable (see Definition~\ref{defn:local-cosmo} and Theorem~\ref{thm:local-cosmo}).
In Sections~\ref{sect:space-times-sub}-~\ref{sect:space-times-ex}, we discuss various ways one might canonically exhaust an asymptotically flat space-time by causally null compactifiable sub-space-times.

In Sections~\ref{sect:space-times-measure} and~\ref{sect:space-times-current} we introduce the idea of a {\em timed metric measure space} 
(which will be explored in future work by Mondino-Perales) and a {\em timed integral current space} (which will be explored in future work by Sakovich-Sormani).
In Section~\ref{sect:space-times-other} we survey other weak notions of space-times developed by
Alexander-Bishop
\cite{Alexander-Bishop-Lorentz},
Braun-McCann\cite{Braun-McCann},
Braun-Ohta
\cite{Braun-Ohta}, 
Burtscher-Garc\'ia-Heveling
\cite{Burtscher-Garcia-Heveling-Time},
Busemann
\cite{Busemann-Timelike}, %
Ebrahimi-Vatandoost-Pourkhandani
\cite{Neda-Mehdi-Rahimeh},
Harris \cite{Harris-triangle},
Kunzinger-Saemann
\cite{Kunzinger-Saemann},
Kunzinger-Steinbauer
\cite{Kunzinger-Steinbauer-22},
McCann
\cite{McCann-Synthetic},
Minguzzi-Suhr
\cite{Minguzzi-Suhr-24},
Mondino-Suhr
\cite{Mondino-Suhr},
M\"uller
\cite{Mueller-GH}, 
and others mentioned within.
Throughout Section~\ref{sect:space-times} we include open problems in addition to our new definitions and theorems.

In Section~\ref{sect:defns} we define a collection of intrinsic distances between pairs of causally-null-compactifiable space-times as follows:

\begin{defn}\label{defn:sd}
 We will say that an {\bf intrinsic distance (defined via null compactification with respect to cosmological time) between a pair of causally-null-compactifiable space-times}, 
$(N_j,g_j)$, is defined by taking an intrinsic distance, $d_{dist}$,
\index{d@$d_{dist}$}
\index{da@$d_{S-dist}$}
between their associated compact timed-metric-spaces, 
$(\bar{N}_j,\hat{d}_{g_j},\tau_{g_j})$,
 as follows:
\be
d_{S-dist}\Big((N_1,g_1),(N_2,g_2)\Big)
=d_{dist}\left(\big(\bar{N}_1,\hat{d}_{g_1},\tau_{g_1}\big),
\big(\bar{N}_2,\hat{d}_{g_2},\tau_{g_2}\big)\right).
\ee
We say this intrinsic distance is {\bf definite} when
\be\label{eq:definite}
d_{S-dist}\Big((N_1,g_1),(N_2,g_2)\Big)=0
\ee
iff there is a bijection, $F: N_1\to N_2$ that is
distance and time preserving
\be
\hat{d}_{g_1}\big(p,q\big)=\hat{d}_{g_2}\big(F(p),F(q)\big) \textrm{ and } \tau_{g_1}(p)=\tau_{g_2}\big(F(p)\big) \quad \forall \, p,q \in N_1,
\ee
and thus is also a Lorentzian isometry satisfying $g_1=F^*g_2$ everywhere.
\end{defn}

In Section~\ref{sect:defns}, we introduce a variety of intrinsic distances  between compact timed-metric-spaces
using a variety of choices for the $d_{dist}$.  
We begin in Section~\ref{sect:defns-tl}
with the timeless Gromov-Hausdorff distance,
\be
d_{S-GH}^{tls}\Big((N_1,g_1),(N_2,g_2)\Big)
=d_{GH}\left(\big(\bar{N}_1,\hat{d}_{g_1}\big),
\big(\bar{N}_2,\hat{d}_{g_2}\big)\right),
\ee
the timeless metric measure distance
and the timeless intrinsic flat distances between space-times
 that have been studied in the past by Sormani, Vega, Kunzinger, Saemann, Allen, Burtscher, 
 and Garc\'ia-Heveling in   
\cite{SV-BigBang}
\cite{Sor-Ober-18}
\cite{Kunzinger-Steinbauer-22}
\cite{Allen-Burtscher-22} 
\cite{Burtscher-Garcia-Heveling-Global} 
\cite{Allen-Null}. 
This notion $d_{S-GH}^{tls}$ is timeless (in the sense that it does not involve the time function) and, as a consequence, it is not definite (see Example~\ref{ex:t-to-x}).

We introduce new notions of intrinsic distances which do keep track of time including distances defined by matching level sets in Section~\ref{sect:defns-S-L},
distances between pairs of causally null compactifiable big-bang space-times  
including $d_{S-BB-GH}$ in  Section~\ref{sect:defns-BB}, distances between pairs of future developments from compact initial data sets
including $d_{S-FD-HH}$
in Section~\ref{sect:defns-FD},
and distances between pairs of cosmic strips including $d_{S-GH}^{strip-sup}$
in Section~\ref{sect:defns-strip}.
Each of our notions is explicitly defined using the $GH$ or Hausdorff distances but we also comment on notions defined using metric measure (mm) or intrinsic flat ($\mathcal{F}$) distances between timed-metric-spaces endowed with measures or current structures respectively.  Some of these notions are familiar from our talks but are defined for the first time within this paper.

Our most powerful new notion is the {\bf intrinsic timed-Hausdorff distance},
$d_{\tau-H}$, introduced in Section~\ref{sect:defns-tau-K},
which is defined on any pair of compact timed-metric-spaces. Thus its corresponding, $d_{S-\tau-H}$, is defined on any pair of causally-null-compactifiable space-times.  It can be used to compare cosmic strips in big bang space-times to cosmic strips in space-times without a big bang.  It can be used to study causally null compactifiable sub-space-times in asympototically flat spaces without a finite cosmological time function.   

In Section~\ref{sect:definite}, we show that some of our intrinsic distances defined via null compactification with respect to cosmological time are definite in the sense defined in Definition~\ref{defn:sd} above
and show others are not definite and leave a few as conjectures and open questions.    Most importantly, we prove
the intrinsic timed-Hausdorff distance, $d_{S-\tau-H}$, is definite in Theorem~\ref{thm:definite-tau-K}.
We also prove in Theorem~\ref{thm:BB-GH-definite} that the $d_{S-BB-GH}$ between causally null compactifiable big bang space-times is definite and in Theorem~\ref{thm:FD-HH} that $d_{S-FD-HH}$ between cosmic strips in future developments of compact initial data sets is definite.
 
In Section~\ref{sect:conv}, 
we state a number of open problems related to convergence of space-times. 
Here we say a sequence of causally-null-compactifiable space-times, $(N_j,g_j)$ {\bf converges} to $(N_\infty,g_\infty)$ in the $S-dist$ sense
if
\be
d_{S-dist}((N_j,g_j),(N_\infty,g_\infty))\to 0.
\ee
We can also consider the convergence to a non-smooth timed-metric-space, $(\bar{N}_\infty,d_\infty,\tau_\infty)$, if
\be
d_{dist}\bigg((\bar{N}_j, \hat{d}_{g_j},\tau_{g_j}),
(\bar{N}_\infty, \hat{d}_{g_\infty},\tau_{g_\infty})\bigg)\to 0.
\ee
In Section~\ref{sect:relations} we state conjectures accessible to doctoral students as to the strength of our various notions of convergence.
In Section~\ref{sect:prior} we describe various existing notions of  convergence for space-times by various teams of authors including Noldus \cite{Noldus},
Cavalletti-Manini-Mondino
\cite{Cavalletti-Manini-Mondino-Optimal}, 
Minguzzi-Suhr \cite{Minguzzi-Suhr-24},
and M\"uller
 \cite{Mueller-GH}. 
Ideally our notions can be related to theirs and used together. 

In Section~\ref{sect:stability-compact}
and Section~\ref{sect:FLRW} we restate classical questions regarding the stability of the Einstein vacuum equations with Compact Initial Data Sets and of Friedmann-Lema\^itre-Robertson-Walker (FLRW) space-times in the language of our intrinsic distances and review related literature.  The key point is that our notions allow space-times with many black-holes (like the actual universe) to be compared to smooth homogeneous and isotropic space-times (like the FLRW space-times).

In Section~\ref{sect:conv-ex} we suggest
defining the convergence of sequences of asymptotically flat space-times using two possible canonical exhaustions by causally-null-compactifiable sub-space-times.
In Sections~\ref{sect:stability}
and~\ref{sect:Penrose}
we apply these ideas to restate classical conjectures regarding the stability of Minkowski, Schwarzschild, and Kerr space-times and to restate the Final State Conjecture and ideas of Penrose in a way which allows for lack of diffeomorphisms.   
Throughout the paper we include remarks, open questions, and conjectures, many of which we hope will be accessible to doctoral students and postdocs.   We also announce future work already in progress by various authors.  In particular, work on space-time intrinsic flat convergence is in progress by Sakovich-Sormani and will appear in \cite{SakSor-SIF}.  We welcome anyone interested in working on the open problems to join one of our teams.   It is especially of interest to us to work together with those developing other notions of convergence and other notions of non-smooth space-times to better understand the relationships between our diverse notions.

Readers should be able to follow most of this paper with only a background in basic Riemannian and Lorentzian Geometry as in Minguzzi's Living Reviews \cite{Minguzzi-living} and Metric Geometry as in Burago-Burago-Ivanov's textbook \cite{BBI}, although we try to review all concepts here before applying them.

\vspace{.3cm}
\noindent{\bf Acknowledgements:} Shing-Tung Yau first suggested the idea of defining a space-time convergence to Christina Sormani and then Lars Andersson suggested using the cosmological time function.   We would also like to thank  Greg Galloway and Carlos Vega for discussions at MSRI in 2013 
and the following two years.  We would like to thank the Simons Center for Geometry and Physics and Oberwolfach for hosting workshops on Mathematical General Relativity in 2018 where we had the opportunity to develop many of these ideas.  We would like to thank the Fields Institute for hosting the Thematic Program on Nonsmooth Riemannian and Lorentzian Geometry
in 2022.  We thank
Raquel Perales and Andrea Mondino
for their comments on an early draft of this paper in 2024.  Finally, we thank
Raquel Perales,
Mauricio Che, Andr\'es Ahumada G\'omez, and
Jaime Santos Rodríguez
for careful feedback on the first arxiv post of this paper.   We thank the Simons Center for Geometry and Physics for funding the upcoming Fall 2025 long program on {\em Geometry and Convergence in Mathematical General Relativity}.

\section{{\bf{Background}}}
\label{sect:back}
In this section we review key prior results that we will be using within this paper.

\subsection{{\bf The Gromov-Hausdorff Distance between Riemannian Manifolds and Metric Spaces}}
\label{sect:back-GH}

Gromov first introduced the Gromov-Hausdorff distance between Riemmanian manifolds in \cite{Gromov-1981}\cite{Gromov-metric} using a canonical conversion of the Riemannian manifolds into metric spaces.  The
text by Burago-Burago-Ivanov is a good source \cite{BBI}.
Gromov's notion is definite because the Riemannian manifolds can be recovered from the metric spaces as we review here.

To define an intrinsic distance between a pair of compact Riemannian manifolds with different topologies, 
$(M_j,h_j)$, Gromov first converted them canonically into metric-spaces, $(M_j, d_{h_j})$, using the Riemannian distance function:  \index{aad@$d_{h}$ Riemannian dist}
\be\label{eq:Riem-dist}
d_{h_j}(p,q)=\inf_{C:[0,1]\to M} \, L_{h_j}(C)
\textrm{ where } L_{h_j}(C)=\int_0^1 h_j(C'(s),C'(s))^{1/2}\, ds
\ee 
and the infimum is over all curves between the pair of points $C(0)=p$ and $C(1)=q$.   Gromov then
used what he called the Intrinsic Hausdorff distance between the two resulting metric-spaces, $(X_j,d_j)$: 

\begin{defn} \label{defn:GH-defn} \index{aaea@$d_{GH}$ Gromov Hausdorff dist}\index{aaeb@$d^Z_H=d_H$ Hausdorff dist}
The Gromov-Hausdorff (GH) distance between two compact metric-spaces, $(X_i,d_i)$, is 
\be \label{defn:GH}
d_{GH}\bigg( (X_1,d_1), (X_2,d_2) \bigg)= 
\inf 
d^Z_{H}
\bigg(
\varphi_1(X_1),\varphi_2(X_2)
\bigg)
\ee
where the infimum is taken over all complete metric-spaces, $Z$,
and over all distance preserving maps, $\varphi_i: X_i \to Z$,
which means
\be\label{eq:dist-pres}
d_Z(\varphi_i(x),\varphi_i(y))=d_i(x,y) \qquad \forall \, x,y\in X_i.
\ee
Here the Hausdorff distance between subsets in $Z$ is
\be\label{eq:Hausdorff}
d^Z_H(A,B)=\inf 
\left\{ r\ \middle\vert \begin{array}{l}
 \forall \, a\in A \, \exists \, b_a\in B
\textrm{ s.t. } d_Z(a,b_a)<r \\
    \forall \, b\in B \,\exists \, a_b\in A
\textrm{ s.t. } d_Z(a_b,b)<r
  \end{array}\right\}.
\ee
\end{defn}

Gromov used compact metric-spaces, $Z$, in his infimum, but the definition is easily seen to be the same using the larger class of complete spaces, since he has shown one really only needs to consider the compact images of the $X_i$ lying inside the $Z$
with the restricted distance from $Z$.   We use complete metric-spaces here for reasons that will be clear later.

\begin{thm} [Gromov]
The Gromov-Hausdorff distance between compact Riemannian manifolds is definite in the sense that 
\be
d_{GH}((M_1,d_{h_1}),(M_2,d_{h_2}))=0 
\ee
implies that there is a distance preserving bijection, 
\be\label{eq:dist-pres-M}
F: M_1 \to M_2 \textrm{ such that }
d_{h_1}(p,q)=d_{h_2}(F(p),F(q)) \qquad \forall \, p,q\in M_1,
\ee
which is thus is also a Riemannian isometry,
\be\label{eq:Riem-isom}
F: M_1\to M_2 \textrm{ such that }F^*h_2=h_1.
\ee
\end{thm}

In fact, Gromov proved that if 
\be
d_{GH}((X_1,d_1)),(X_2,d_2)))<\epsilon
\ee
then there is a $2\epsilon$ almost isometry
$F:X_1\to X_2$ which is a map that
is $2\epsilon$ almost distance preserving
\be
|d_1(p,q)-d_2(F(p),F(q))|<2\epsilon \qquad \forall \, p,q\in X_1,
\ee
and is $2\epsilon$ almost onto
\be
\forall \, q\in X_2\quad
\exists \, p\in X_1 \textrm{ such that }
d_2(q,F(p))<2\epsilon.
\ee
So we have a distance preserving bijection
when taking $\epsilon \to 0$.    Once one has a
distance preserving isometry as in (\ref{eq:dist-pres-M}), one shows the rectifiable lengths of curves in the Riemannian manifolds are preserved, $L_{h_1}(C)=L_{h_2}(F\circ C)$,
and so we recover a Riemannian isometry as in (\ref{eq:Riem-isom}).
 
Such a canonical conversion and recovery process was not understood between space-times with Lorentzian metric tensors and their corresponding metric spaces endowed with the null distance.  This has caused significant delays in proving intrinsic space-time distances are definite.  However, we now have the necessary tools, as we shall see within this paper.

\subsection{{\bf Review of Metric Completions}}
\label{sect:back-completion}

The Gromov-Hausdorff distance is only proven to be definite for compact metric spaces.
For example, if one were to take the open Euclidean disk and the closed Euclidean disk, the $GH$ distance between these two spaces is $0$.   When a pair of Riemannian manifolds or metric spaces is only precompact (not compact), then one takes their {\bf metric completions} to have a pair of compact metric spaces before taking their Gromov-Hausdorff distance.   As we need this concept, we add more details here:

\begin{defn}\label{defn:metric completion}
The {\bf metric completion}, $(\bar{X},d)$, of a metric space, $(X,d)$, is defined as the set of equivalence classes of Cauchy sequences in $(X,d)$, where a
sequence, $\{x_i:\, i\in {\mathbb N}\}$, is Cauchy iff
\be
\forall \epsilon>0 \,\exists N_\epsilon\in {\mathbb N} \textrm{ s.t. } 
\forall i,j \ge N_\epsilon
\,\,d(x_i,x_j)<\epsilon,
\ee
and $\{x_i:\, i\in {\mathbb N}\}$ is equivalent to
$\{y_i:\, i\in {\mathbb N}\}$ iff
\be
\forall \epsilon>0 \,\exists N'_\epsilon\in {\mathbb N} \textrm{ s.t. } 
\forall i \ge N'_\epsilon
\,\,d(x_i,y_i)<\epsilon.
\ee
\end{defn}

Every metric space has a complete metric completion. A metric space is said to be precompact if its metric completion is compact.

Given a $\lambda$ Lipschitz map, $f: X\to Y$,
\be \label{eq:Lip-ext}
d_Y(f(p),f(q))\le \lambda d_X(p,q)
\quad \forall p,q\in X,
\ee
it has a unique
$\lambda$ Lipschitz extension to the
metric completion, 
$f: \bar{X}\to \bar{Y}$.   In particular,
distance preserving isometries extend to distance preserving maps.

\subsection{{\bf Pointed Gromov-Hausdorff Convergence}}
\label{sect:pted-GH}

When studying sequences of complete Riemannian manifolds, one needs to use an exhaustion by compact sets in order to find a limit using Gromov-Hausdorff convergence.  
Closed balls, $\bar{B}(p,R)$, within complete Riemannian manifolds are always compact, so one can use them to define pointed Gromov-Hausdorff convergence of Riemannian manifolds using the pointed Gromov-Hausdorff convergence of their corresponding metric spaces:

\begin{defn}\label{defn:pted-GH}
A sequence of pointed metric spaces, $(X_j,d_j, p_j)$ converges in the pointed
Gromov-Hausdorff sense to 
$(X_\infty,d_\infty,p_\infty)$
iff
\be
\forall R>0 \, \exists R_j \to R
\textrm{ s.t. }
d_{GH}(\bar{B}(p_j,R_j), \bar{B}(p_\infty,R))\to 0.
\ee
\end{defn}

Within this paper we will also suggest natural exhaustions of space-times by causally null compactifiable sub-space-times in settings where the space-times themselves are not converted into compact metric spaces using the null distance.

\subsection{{\bf Review of Measures and Metric Measure Distances}}
\label{sect:back-mm}
\footnote{This section may be skipped by those only interested in $GH$ or $\mathcal{F}$ notions of distance.}

Various notions of intrinsic metric measure ($mm$) distances
between compact Riemannian manifolds and more general metric measure spaces have been developed by Lott-Villani \cite{Lott-Villani-09} and by Sturm \cite{Sturm-2006-I} building upon work of Cheeger-Colding and Fukaya \cite{ChCo-PartI}\cite{Fukaya-collapsing}. 

First they converted the $(M_j,h_j)$
into compact metric-spaces, $(X_j,d_j)=(M_j, d_{h_j})$
exactly as done by Gromov, but then they also 
kept track of the renormalized volume measures, 
$\mu_j$, defined so that $\mu_j(M_j)=1$
to create a pair of metric measure spaces, $(X_j,d_j,\mu_j)$, endowed with probability measures, $\mu_j$, such that $\spt(\mu_j)=X_j$.  One can then define the metric measure ($mm$) distance between these spaces using the Wasserstein distance 
as follows:\index{aaw@$d_{mm}$}
\be \label{defn:mm}
d_{mm}((X_1,d_1,\mu_1),(X_2,d_2,\mu_2))
=\inf d_W^Z(\varphi_{1*}\mu_1,\varphi_{2*}\mu_2)
\ee
where the infimum is taken over all complete metric-spaces, $Z$, and over all distance preserving maps, $\varphi_i:X_i\to Z$.   \index{aawa@$d_W^Z$}
Here 
$d_W^Z$ is the Wasserstein distance between the push forwards, $\varphi_{j*}\mu_j$, of the probability measures, $\mu_j$, defined using mass transportation.  Note that it is also possible to define a Wasserstein distance without any reference to $Z$, but for our purposes we define it in this way.   See, for example,  Sturm's D-distance in \cite{Sturm-2006-I} and Villani's textbook on Optimal Transport \cite{Villani-text}.

In this paper we do not need to know the precise definition of the push forwards and the Wasserstein distance,  we just need to understand that metric measure distances can be defined using distance preserving maps and that they are definite in the following sense:  
\be
d_{{mm}}((X_1,d_1,\mu_1),(X_2,d_2,\mu_2))=0 
\ee
implies there is a distance and measure preserving bijection $F:X_1\to X_2$:
\be
d_2(F(x),F(y))=d_1(x,y) \quad \forall \, x,y\in X_1
\textrm{ and } F_{*}\mu_1=\mu_2
\ee
which is also a Riemannian isometry when $(X_i,d_i,\mu_i)=(M_i,d_{g_i},\mu_i)$.  
In this paper we will discuss potential applications of this notion to define various notions of intrinsic space-time distances, but we leave the details to experts. 

\subsection{{\bf Integral Currents and the Intrinsic Flat Distance}}
\label{sect:back-F}
\footnote{This section may be skipped by those only interested in $GH$ and $mm$ notions of distance.}

Sormani-Wenger defined the intrinsic flat ($\mathcal{F}$) distance between compact oriented Riemannian manifolds, $(M_j,g_j)$, and more general spaces called integral current spaces in \cite{SW-JDG}.   As seen in work of Lee-Sormani \cite{LeeSormani1} and Lakzian-Sormani \cite{Lakzian-Sormani}, an almost flat Riemannian manifold with a deep well of small volume or a black hole is close in this $\mathcal{F}$ sense to flat space without a well or hole.  This notion has been applied to measure the geometric stability or almost rigidity of the Riemannian Positive Mass Theorem 
(see Remark~\ref{rmrk:PMT-Stab} within).   

To define the intrinsic flat distance between oriented compact Riemannian manifolds, Sormani-Wenger first convert the $(M_j,g_j)$
into metric-spaces, $(X_j,d_j)=(M_j, d_{g_j})$
exactly as done by Gromov, but then they also 
kept track of the integral current structure, 
$T_j=[[M_j]]$ to create a pair of integral current spaces, $(X_j,d_j,T_j)$, where
\be\label{eq:int-str}
[[M_j]](\omega)=\int_{M_j}\omega,
\ee 
for any top dimensional differential form, $\omega$ \cite{SW-JDG}.  The intrinsic flat ($\mathcal{F}$) distance between these spaces is then defined using Ambrosio-Kirchheim Theory from  \cite{AK}  \index{aawba@$d_{\mathcal{F}}$ intrinsic flat dist}
\index{aawbb@$d^Z_{F}=d_F$ flat distance }
as follows:
\be \label{defn:F} 
d_{\mathcal{F}}((X_1,d_1,T_1),(X_2,d_2,T_2))
=\inf d_F^Z(\varphi_{1\#}T_1,\varphi_{2\#}T_2)
\ee
where the infimum is taken over all complete separable metric-spaces, $Z$, and over all distance preserving maps, $\varphi_i:X_i\to Z$.   Here 
$d_F^Z$ is the flat distance between the push forwards, $\varphi_{j\#}T_j$, of the integral
current structures.   

In this paper we do not need to know the precise definition of the push forwards and the flat distance, $d_F^Z$,  we just need to understand that it involves distance preserving maps and that it is definite in the following sense:  
\be
d_{\mathcal{F}}((X_1,d_1,T_1),(X_2,d_2,T_2))=0 
\ee
implies that there is a distance and current preserving bijection $F:X_1\to X_2$:
\index{aawbaa@$d_{\mathcal{VF}}$}
\be
d_2(F(x),F(y))=d_1(x,y) \quad \forall \, x,y\in X_1
\textrm{ and } F_{\#}T_1=T_2
\ee
which is also an orientation preserving Riemannian isometry when $(X_i,d_i,T_i)=(M_i,d_{g_i},[[M_i]])$.  

There is also the volume preserving intrinsic flat distance denoted,
\be
d_{\mathcal{VF}}((X_1,d_1,T_1),(X_2,d_2,T_2))
= d_{\mathcal{F}}((X_1,d_1,T_1),(X_2,d_2,T_2))
+|\mass_{d_1}(T_1)-\mass_{d_2}(T_2)|,
\ee
which was
first studied by Portegies \cite{Portegies-evalues} 
and later by many others as
surveyed in \cite{Sormani-conjectures}.   

In this paper we will discuss potential applications of this notion to define various types of intrinsic flat space-time distances, but the advanced geometric measure theory needed to develop these notions in full will appear in future work by 
Sakovich-Sormani in \cite{SakSor-SIF}.   We wish to keep this paper accessible to a wider audience.

\subsection{{\bf Review of Fr\'echet Maps}}
\label{sect:back-Kuratowski}
\footnote{This section may be skipped if uninterested in the timed Hausdorff distance.}

A key ingredient in our paper is an adaption of 
the Fr\'echet map first defined in \cite{Frechet1910}.\footnote{In our v1 post on the arxiv we accidentally attributed these maps to Kuratowski.}
So we review this map here.
First recall that a separable metric-space, $(X,d)$,  is a metric-space containing a countable dense collection of points, $\{x_1,x_2,...\}\subset X$.
Compact metric-spaces are separable and so is the Banach space, 
\be\label{eq:defn-ell-infty}
\ell^\infty=\{(s_1,s_2,...)\,: s_i \in {\mathbb{R}}, \, d_{\ell^\infty}((s_1,s_2,...),(0,0,...))<\infty\}
\ee
where   
\be\label{eq:d-ell-infty}
d_{\ell^\infty}((s_1,s_2,...),(r_1,r_2,...))
=\sup\{|s_i-r_i|\,:\, i\in {\mathbb N}\}.
\ee

\begin{defn}\label{defn:Kuratowski}
Given a separable metric-space, $(X,d)$, 
one can define a Fr\'echet map
\be
\kappa_X=\kappa_{X,\{x_1,x_2,...\}}: (X,d)\to (\ell^\infty, d_{\ell^\infty})
\ee
by
\be\label{eq:Kuratowski}
\kappa_X(x)=(d(x_1,x),d(x_2,x),...)\subset \ell^\infty.
\ee
\end{defn}

For completeness of exposition we state and prove Fr\'echet's Theorem \cite{Frechet1910}: 

\begin{thm}[Fr\'echet]
\label{thm:K-dist-pres}
Fr\'echet maps as in Definition~\ref{defn:Kuratowski}
are distance preserving:
\be\label{eq:K-dist-pres}
d_{\ell^\infty}(\kappa_X(x),\kappa_X(y))=
d(x,y) \quad \forall \, x,y\in X.
\ee
\end{thm}

\begin{proof}
We first apply the reverse 
triangle inequality:
\be
|d(x,x_i)-d(y,x_i)|\le d(x,y).
\ee
Next we choose a sequence $x_{i_j}\to y$ from the countable
dense set, and observe that
\be
|d(x,x_{i_j})-d(y,x_{i_j})|\to 
|d(x,y)-d(y,y)|=d(x,y).
\ee
Thus
\be
\sup_{i\in {\mathbb N}}|d(x,x_i)-d(y,x_i)|=d(x,y)
\ee
which gives (\ref{eq:K-dist-pres}).
\end{proof}

Note that the fact that the Fr\'echet map is distance preserving immediately implies that $\ell^\infty$ is a candidate for the metric-spaces, $Z$, used in the definitions of $GH$, $mm$, and $\mathcal{F}$ distances, and that Fr\'echet maps are candidates for the distance preserving maps in (\ref{defn:GH}),
(\ref{defn:mm}), and (\ref{defn:F}).    Some papers 
will use this infima as the definition of $d_{GH}$ but they are not always exactly equal, so we will define the infima achieved using Fr\'echet maps in 
(\ref{defn:GH}), (\ref{defn:mm}), and (\ref{defn:F})
as $d_{\kappa-GH}$, $d_{\kappa-mm}$, and
$d_{\kappa-\mathcal{F}}$ respectively.

To be very clear, we will use the notation,
\be\label{defn:kappa-GH}
d_{\kappa-GH}((X,d_X),(Y,d_Y))
=\inf d_H^{\ell^\infty}(\kappa_X(X),\kappa_Y(Y))
\ee
where the infimum is over all pairs of Fr\'echet maps
\be
\kappa_X:X \to \ell^\infty
\textrm{ and }
\kappa_Y:Y \to \ell^\infty
\ee
which is an infimum over
all selections of countably dense points and reorderings of these selected points in
$X$ and in $Y$
and the Hausdorff distance is defined as before in (\ref{eq:Hausdorff}), so that one immediately has
\be\label{eq:kappa-ge}
d_{\kappa-GH}((X,d_X),(Y,d_Y))\ge
d_{GH}((X,d_X),(Y,d_Y)).
\ee
For completeness of exposition we include the following well known
proposition which demonstrates that
\be \label{eq:kappa-GH-comp}
d_{GH}((X,d_X),(Y,d_Y))
\le d_{\kappa-GH}((X,d_X),(Y,d_Y))\le
2 d_{GH}((X,d_X),(Y,d_Y))
\ee
for any pair of compact metric-spaces.  Similar inequalities with different constants can be proven to relate $d_{mm}$ with $d_{\kappa-mm}$ and $d_{\mathcal{F}}$ with $d_{\kappa-{\mathcal{F}}}$.

\begin{prop}\label{prop:2GH}
Given two compact metric-spaces, $(X,d_X)$
and $(Y,d_Y)$,
if 
\be
d_{GH}(X,Y)<R
\ee
then there exists countable dense sets
\be
\{x_1,x_2,...\}\subset X
\textrm{ and }
\{y_1,y_2,....\}\subset Y
\ee
such that the Fr\'echet maps, 
\be
\kappa_X:X\to \ell^\infty
\textrm{ and }
\kappa_Y:Y\to \ell^\infty
\ee
defined using these countable dense sets
as in Definition \ref{defn:Kuratowski} are distance
preserving maps such that
\be\label{K-Haus}
d_H^{\ell^\infty}(\kappa_X(X),\kappa_Y(Y))
< 2R.
\ee
\end{prop}

\begin{proof}
Since $d_{GH}(X,Y)<R$
then there exists distance preserving maps,
\be
\varphi_X:(X,d_X)\to (Z, d_Z)
\textrm{ and }
\varphi_Y:(Y,d_Y)\to (Z, d_Z),
\ee
into a separable metric-space, $(Z,d_Z)$, such that
\be
d^Z_H(\varphi_X(X),\varphi_Y(Y))<R.
\ee
So for all $x\in X$ there exists $y_x\in Y$
such that
\be\label{eq:y_x-Haus}
d_Z(\varphi_X(x),\varphi_Y(y_x))<R,
\ee
and for all $y\in Y$ there exists $x_y\in X$
such that
\be
d_Z(\varphi_X(x_y),\varphi_Y(y))<R.
\ee
Let $\{x'_1,x'_2,...\}$ be any countably dense collection of points in $X$
and let 
$\{y'_1,y'_2,...\}$ be any countable dense collection of points in $Y$.   We create a larger countably dense collection of points,
\be
\{x_1,x_2,x_3...\}=\{x_1',x_{y_1'},x_2',x_{y_2'},
x_3',x_{y_3'},...\}\subset X
\ee
and
\be
\{y_1,y_2,y_3,....\}=
\{y_{x_1'},y_1',y_{x_2'}, y_2',
y_{x_3'},y_3',...\}\subset Y
\ee
and define the Fr\'echet maps,
$\kappa_X$ and $\kappa_Y$ as in the
statement of this proposition using this collection of points.

We need only show that:
\be
d_{\ell^\infty}(\kappa_X(x),\kappa_Y(y_x))
<2R
\textrm{ and }
d_{\ell^\infty}(\kappa_X(x_y),\kappa_Y(y))
<2R
\ee
to see that (\ref{K-Haus}) holds.  
We will show the first of these and the
other follows similarly.
Note that
\begin{eqnarray}
\kappa_X(x)&=&(d_X(x_1',x),d_X(x_{y_1'},x),d_X(x_2',x),
d_X(x_{y_2'},x),...)\\
\kappa_Y(y_x)&=&
(d_Y(y_{x_1'},y_x),d_Y(y_1',y_x),
d_Y(y_{x_2'},y_x), d_Y(y_2',y_x),...)
\end{eqnarray}
So we need only show that 
\be
|d_X(x_i',x)-d_Y(y_{x_i'},y_x)|<2R
\textrm{ and }
|d_X(x_{y_i'},x)-d_Y(y_i',y_x)|<2R.
\ee
We show the first of these as the second
follows similarly.
By the fact that $\varphi_X$ and 
$\varphi_Y$ are distance preserving,
and that the triangle inequality
\footnote{$
|d(a,b)-d(c,d)|\le 
|d(a,d)-d(c,d)|+|d(a,b)-d(a,d)|
\le d(a,c)+d(b,d).
$
}
gives,
\be
|d(a,b)-d(c,d)|
\le d(a,c)+d(b,d),
\ee
we can apply (\ref{eq:y_x-Haus})
to see that
\begin{eqnarray*}
|d_X(x_i',x)-d_Y(y_{x_i'},y_x)|
&=&
|d_Z(\varphi_X(x_i'),\varphi_X(x))-
d_Z(\varphi_Y(y_{x_i'}),\varphi_Y(y_x))|\\
&\le &
d_Z(\varphi_X(x_i'),\varphi_Y(y_{x_i'}))
+ d_Z(\varphi_X(x),\varphi_Y(y_x))
< 2R.
\end{eqnarray*}
\end{proof}

\subsection{{\bf Review of Lorentzian Manifolds and Cosmological Time}}
\label{sect:back-L}


A space-time, $(N,g)$, is a
smooth $n$ dimensional manifold with a
Lorentzian metric tensor, $g$, of signature, $- + \cdots +$, endowed with a time orientation. 
A curve, $C:[a,b]\to N$, is a causal curve if 
\be\label{eq:causal-curve}
C:[a,b]\to N \textrm{ such that } g(C'(s),C'(s))\le 0 \qquad \forall \, s\in [a,b],
\ee
and it is a future directed causal curve
if $C'(s)$ points in the future direction.

The space-time, $N$, has 
a {\bf causal structure}, defined as follows: 
$p\in N$ is in the {\bf causal future} of $q\in N$,
written $p\in J^+(q)$, if there exists a future directed causal curve $C:[0,1]\to N$ such that $C(0)=q$ and $C(1)=p$. Equivalently, we can say that $q$ is in the {\bf causal past} of $p$, 
written $q\in J^-(p)$.

One cannot define a distance as in (\ref{eq:Riem-dist}) because the infimum would be zero as seen by taking piecewise null curves,
\be \label{eq:null-curve}
C:[a,b]\to N_j \textrm{ such that } g(C'(s),C'(s))= 0 \qquad \forall \, s\in [a,b].
\ee
However, one can canonically define the {\bf Lorentzian distance} between points 
such that $p\in J^+(q)$ or $q\in J^+(p)$ by taking the supremum
of the {\bf Lorentzian lengths} over all future directed causal curves between them:
\be \index{c@$d_{g}$ Lorentzian dist}
d_{g}(p,q)=\sup_{C \textrm{ causal}} \, L_{g}(C)
\textrm{ where } L_{g}(C)=\int_0^1 |g(C'(s),C'(s))|^{1/2}\, ds.
\ee

We can now introduce the cosmological time function as in Andersson-Galloway-Howard \cite{AGH} (see also Wald-Yip \cite{WaldYip}):

\begin{defn} \label{defn:cosmotime} \index{ca@$\tau_{g}$ cosmo time}
Given a spacetime $(N,g)$, its {\bf cosmological time function} $\tau$ is defined by
\be
\tau_{g}(p)=\sup\{d_{g}(p,q'): \,\, q'\in J^-(p)\}\in [0,\infty].
\ee   
\end{defn}

\begin{ex}\label{ex:Mink-tau}
On Minkowski space-time, 
\be
{\mathbb{R}}^{1,3}=
\{(t,x)\,:\, t\in {\mathbb R}, \, x\in {\mathbb R}^3\}
\textrm{ with } g_{Mink}= -dt^2+dx_1^2+dx_2^2+dx_3^2,
\ee
the cosmological time function is infinite valued everywhere, $\tau(t,x)=\infty$.
\end{ex}

\begin{ex} \label{ex:future-Mink-tau}
On Future Minkowski space-time, 
\be
{\mathbb{R}}^{1,3}_+=
\{(t,x)\,:\, t>0, \, x\in {\mathbb R}^3\}
\textrm{ with } g_{Mink}= -dt^2+dx_1^2+dx_2^2+dx_3^2,
\ee
the cosmological time function is found to be, $\tau(t,x)=t$, which can intuitively be thought of as the time since the initial data slice at $t=0$.
\end{ex}

Our next example includes the 
Friedmann–Lema\^itre–Robertson–Walker (FLRW) big bang space-times:

\begin{ex}\label{ex:BB-tau}
On any warped product big bang space-time, 
\be 
N_{BB,f}=(0,t_{max})\times M
\textrm{ with }
g=-dt^2+f^2(t)\,h_M 
\ee
where $\lim_{t\to 0} f(t)=0$ and $h_M$ is the Riemannian metric on a compact Riemannian manifold, $M$, we easily see that the cosmological time function, $\tau(t,x)=t$, which can be thought of intuitively as the time since the ``big bang"
at $t=0$. Note that this includes as a special case the FLRW big bang space-times which are spatially homogeneous and isotropic so that $(M,h)$ has constant sectional curvature.
\end{ex}

The following theorem clarifies what it means to say that the cosmological time function is canonical.   It summarizes key properties of time-oriented Lorentzian isometries and can easily be proven by the reader.

\begin{thm} \label{thm:time-canonical}
Suppose there is a time-oriented Lorentzian isometry, which is a diffeomorphism,
\be
F:N_1\to N_2
\textrm{ such that } F^*g_2=g_1,
\ee
which matches future causal cones to future causal cones.
Then $F$
preserves the Lorentzian lengths of curves,
\be
L_{g_1}(C)=L_{g_2}(F\circ C) \textrm{ for any timelike curve } C:[a,b]\to N_1,
\ee
$F$ preserves the time-orientation,
\be
\textrm{$C$ is future timelike causal iff $F\circ C$ is future timelike causal,}
\ee
$F$ preserves the causal structure,
\be
F(p)\in J_{+}(F(q)) \quad \iff \quad p\in J^+(q),
\ee
$F$ preserves the Lorentzian distances,
\be
d_{g_2}(F(p),F(q))=d_{g_1}(p,q) \quad \forall \, p \in J^+(q) \text{ and } q\in N_1,
\ee
and $F$ preserves the cosmological time functions:
\be
\tau_{g_2}(F(p))=\tau_{g_1}(p) \qquad \forall \, p\in N_1.
\ee
\end{thm}

Andersson-Galloway-Howard made the following definition in \cite{AGH}
and proved the following theorem using the
notion of a future causal curve, $C:(a,b)\to N$,
which is {\bf past inextendible}:
\be\label{eq:past-inext}
\nexists q\in N \textrm{ such that }
\lim_{s\to a}C(s)=q. 
\ee

\begin{defn} \label{defn:regular}\cite{AGH}
A Lorentzian space-time $(N,g)$ has a {\bf regular cosmological time function} if its cosmological time function, $\tau_g$, has finite values at every $p$ and for every future causal curve, $C:(a,b)\to N$, which is past inextendible 
we have
$\lim_{s\to a}\tau_{g}(C(s))=0$.   
\end{defn}

\begin{thm}\cite{AGH} \label{thm:AGH-generator}
If a Lorentzian space-time has a regular cosmological time function, $\tau_g$, then at every point, $p\in N$, there is a {\bf generator}, $\gamma_p:(0,\tau(p)]\to N$, which is a past inextendible timelike unit speed geodesic such that
\be \label{eq:generator}
\tau_g(\gamma_p(t))=t \textrm{ and } \gamma_p(\tau_g(p))=p.
\ee
This generator is unique for any $p\in N$ where $\tau_g$ is differentiable (at almost every $p\in N$). 
\end{thm}

Further results on cosmological time 
appear in work of Bahrampour, Cui, Ebrahimi, Jin, Koohestani, Treude, and Vatandoost   
in \cite{Cui-Jin} \cite{E-cosmo}\cite{KEVB-cosmo}
\cite{Treude-Diploma}.

\begin{rmrk}\label{rmrk:Howard}
When Shing-Tung Yau first proposed to Christina Sormani that she define an intrinsic flat convergence of space-times, Lars Andersson suggested that a natural way to do this would be to use the cosmological time function to canonically convert the space-times into integral current spaces.
When a space-time, $(N,g)$, has a smooth cosmological time function, $\tau_g$, Andersson pointed out that one may use the Wick transform to canonically convert $(N,g)$ into a Riemannian manifold $(N, 2d\tau^2+g)$ of the same dimension. Since regular cosmological time functions are differentiable almost everywhere, Andersson suggested that one could convert $N$ into an integral current space
or at least a metric-space using $2d\tau^2+g$ almost everywhere.   There was some unpublished work in this direction by Howard and a student of his in 2013.  However, this work was never completed. \end{rmrk}
 
\begin{rmrk}\label{rmrk:Riem-levels}
Another natural approach would be to consider the spacelike level sets of the cosmological time functions, $M_{t}=\tau_{g}^{-1}(t)\subset N$, as metric-spaces with a Riemannian metric tensor, $h_t$, induced by the Lorentzian metric restricted to the level sets of $\tau_g$ (defined almost everywhere
on almost every level set).  One could then define a distance between two space-times using the distances between their level sets.  However,
a pair of space-times can have isometric level sets without having a Lorentzian isometry between them.  See Example~\ref{ex:levels-match} within.  
\end{rmrk}

As neither approaches described in Remarks~\ref{rmrk:Howard} and~\ref{rmrk:Riem-levels} were leading towards a definite notion of intrinsic flat distance between space-times, Sormani-Vega \cite{SV-Null} introduced the null distance that we describe in the next section.

\subsection{{\bf Review of the Null Distance}}
\label{sect:back-Null}

Given a Lorentzian manifold endowed with a time function, $(N_j, g_j, \tau_j)$, Sormani-Vega defined the null distance in \cite{SV-Null}
as follows:

\begin{defn}\label{defn:null-distance}
The null distance between arbitrary pairs of points in $N$: \index{cb@$\hat{d}_{g}=\hat{d}_{g,\tau}$ null dist}
\be\label{eq:null}
\hat{d}_{g,\tau}(p,q)=\inf \sum_{i=1}^N |\tau(p_i)-\tau(p_{i-1})| 
\ee
where the infimum is over collections of points, $\{p=p_0, p_1,p_2,...,p_{2N}=q\}$,
that are causally related alternating past
with future:
\be\label{eq:null-points}
p_{2i-1}\in J^-(p_{2i-2})\cap J^-(p_{2i}) 
\textrm{ for } i=1,...,N.
\ee
That is, there is a piecewise causal curve, alternating future causal and past causal, from
$p$ to $q$ with corners at the $p_i$.
\end{defn}  

In this paper, we will consider only null distances defined using 
cosmological time functions, $\tau_g$, 
and write 
\be 
\hat{d}_g=\hat{d}_{g, \tau_g},
\ee
to emphasize that this null distance canonically depends only on the Lorentzian metric tensor, $g$.  In \cite{SV-Null} Sormani-Vega proved this null distance, $\hat{d}_g$, is definite and
induces the same topology as the original topological structure on 
$N$ 
when $\tau_g$ is regular.  So $(N, g)$ can be  canonically converted into a metric-space,
$(N, \hat{d}_{g})$.
The following theorem (which easily follows from Theorem~\ref{thm:time-canonical}
and Definition~\ref{defn:null-distance}), clarifies what it means to say that this null distance is {\bf canonical}:

\begin{thm} \label{thm:null-canonical}
Suppose there is a time-oriented Lorentzian isometry, which is a diffeomorphism,
\be
F:N_1\to N_2
\textrm{ such that } F^*g_2=g_1,
\ee
which matches future causal cones to future causal cones.

Then 
$F$  
is a distance preserving bijection
\be
 \hat{d}_{g_2}(F(p),F(q))= \hat{d}_{g_1}(p,q)
 \quad \forall \, p,q\in N_1.
 \ee
\end{thm}

In \cite{SV-Null}, Sormani-Vega
proved that the cosmological time function is Lipschitz with respect to its null distance:
\be \label{eq:Lip}
|\tau_g(p)-\tau_g(q)| \le \hat{d}_{g}(p,q) \quad \forall \, p,q \in N
\ee
with equality (without the absolute values) when $p\in J^+(q)$.

Allen-Burtscher proved the following lemma
in \cite{Allen-Burtscher-22} 
building upon a result of Sormani-Vega concerning Minkowski space in \cite{SV-Null}:

\begin{lem} \label{lem:Mink-null} 
If $N=[0,a]\times M=\{(t,x):\, t\in [0,a],\, x\in M\}$ and $g=-dt^2+h$ where
$(M,h)$ is a Riemannian manifold, then
the cosmological time function is $\tau_{g}(t,x)=t$, and it is regular, and the null distance
is
\be
\hat{d}_{g}\bigg( (t_1,x_1), (t_2,x_2) \bigg)=\max\{|t_1-t_2|, d_h(x_1,x_2)\},
\ee
where $d_h$ is the Riemannian distance on $M$.
\end{lem}


\subsection{{\bf Review of Causality and Lorentzian Isometries}}
\label{sect:back-C}

Sormani-Vega conjectured that {\bf the null distance encodes causality},
\be\label{eq:encode-causality}
p\in J^+(q) \iff \tau_{g}(p)-\tau_g(q)=\hat{d}_{g}(p,q),
\ee
so that one can recover the causal structure without knowing the metric tensor $g$ in \cite{SV-Null}.
This was proven in the warped product case by
Allen-Burtscher in \cite{Allen-Burtscher-22}.  

Sakovich and Sormani then proved that the null distance encodes causality as in (\ref{eq:encode-causality}) when the space-time, $(N,g)$, has a regular cosmological time function, $\tau_g: N \to [0,\tau_{max}]$ which is proper \cite{SakSor-Null}.   Burtscher and García-Heveling extended this theorem in \cite{Burtscher-Garcia-Heveling-Global} and Galloway generalized their result further in \cite{Galloway-causal} to the following theorem: 
 
 \begin{thm} \label{thm:G} \cite{Galloway-causal}
 Suppose $(N,g)$ is a space-time with a regular cosmological time function, $\tau_g$, such that 
 all the level sets, $S=\tau_g^{-1}(t)$, are future causally complete:
\be \label{eq:fcc}
\forall \, p\in J^+(S), \textrm{ the set }
J^-(p)\cap S \textrm{ has compact closure in }S.
\ee
Then the null distance encodes causality
as in (\ref{eq:encode-causality}).
\end{thm}

All three of these causality encoding theorems 
are proven for a larger class of time functions, but we are focusing on the cosmological time function in this paper because it defines a canonical null distance in Theorem~\ref{thm:null-canonical}.  Note that not
all space-times have regular cosmological time functions (see \cite{AGH}) and not all space-times with regular cosmological time functions have null distances that encode causality (see \cite{SakSor-Null}).  

In Section~\ref{sect:space-times} we will present a large class of causally-null- compactifiable space-times (as in Definition~\ref{defn:cncst}) 
that arise naturally in General Relativity. We will apply
the Sakovich-Sormani causality theorem 
to prove that a large class of
big bang space-times and also future developments of compact initial data sets are causally-null-compactifiable space-times by proving their cosmological time functions are proper.   For asymptotically flat space-times, like Minkowski space-time and Schwarzschild space-time, we will describe regions within the space-times that exhaust the space-times and can be shown to
be causally null compactifiable by applying Galloway's causality theorem.

We end this background section with a theorem that immediately follows from Definition~\ref{defn:cncst}
and Theorem 1.3 of \cite{SakSor-Null} by Sakovich-Sormani:

\begin{thm} 
\label{thm:dist-time-Lor} 
A bijection, 
$F:N_1\to N_2$, between causally-null-compactifiable space-times 
preserves distance,
\be \label{eq:pres-dist}
\hat{d}_{g_2}(F(p),F(q))=\hat{d}_{g_1}(p,q) \qquad \forall \, p,q\in N_1,
\ee
and preserves cosmological time,
\be \label{eq:pres-time}
\tau_2(F(p))=\tau_1(p) \qquad \forall \, p \in N_1,
\ee
iff this bijection 
is a Lorentzian isometry,
\be\label{eq:Lor-isom}
F: N_1\to N_2 \textrm{ such that }F^*g_2=g_1 \textrm{ everywhere in } N_1,
\ee
which is time oriented.
\end{thm}

We will apply this last theorem in Section~\ref{sect:definite} to prove some of the intrinsic distances between causally-null-compactifiable space-times defined in Section~\ref{sect:defns} are definite.

\section{{\bf Causally-Null-Compactifiable Space-Times}}
\label{sect:space-times}

In this section we will introduce large classes of causally-null-compactifiable space-times as defined in Definition~\ref{defn:cncst}.  

We will begin in Section~\ref{sect:space-times-induced} 
 by proving the following constructive theorem describing the metric-space associated with a cosmic strip and its cosmological time function:

\begin{thm}\label{thm:induced}
If $(N,g)$ has a regular cosmological time function, $\tau_g:N\to (0,\infty)$, and $0<s<t<\infty$ then the cosmic strip, 
\be
N_{s,t}=\tau_g^{-1}(s,t)\subset N
\textrm{ with metric tensor } g,
\ee
has a 
regular cosmological time function,
\be\label{eq:taust-1}
\tau_{g,s,t}:N_{s,t}\to (0,t-s)
\textrm{ such that } \tau_{g,s,t}(p)=\tau_g(p)-s
\ee
and an associated metric-space,
$({N}_{s,t},\hat{d}_{g,s,t})$,
where $\hat{d}_{g,s,t}=\hat{d}_{g,\tau_{g,s,t}}$ is defined using $\tau_{g,s,t}$
and piecewise causal curves in $N_{s,t}$.  Furthermore,
$\hat{d}_{g,s,t}$ is the
induced length metric on $N_{s,t}$:
\be\label{eq:induced}
\hat{d}_{g,\tau_{g,s,t}}(p,q)
=\inf L_{\hat{d}_{g}}(C)
\ge \hat{d}_{g}(p,q)
\ee
where the infimum is taken over all $\hat{d}_{g}$
rectifiable curves 
$C:[0,1]\to N_{s,t}$ 
such that $C(0)=p$ and $C(1)=q$.
\end{thm}

In Section~\ref{sect:space-times-proper} we will apply this theorem combined with Sakovich-Sormani's encoding causality theorem \cite{SakSor-Null} to prove Theorem~\ref{thm:proper} which states that if $(N,g)$ has a proper regular cosmological time function, then the cosmic strip is a causally-null-compactifiable space-time, and we extend the time function, the causal structure, and the generators to the corresponding compact timed-metric-space. 

In Sections~\ref{sect:space-times-big-bang} and ~\ref{sect:space-times-initial-data} we discuss how certain classes of big bang space-times and certain classes of future developments from compact initial data sets are causally-null-compactifiable space-times.  This includes a description of the unpublished work of Sormani-Vega in \cite{SV-BigBang}, 
Conjecture~\ref{conj:BB}, and our new Theorem~\ref{thm:FD}.

In Section~\ref{sect:space-times-strip} we consider space-times which do not have finite cosmological time functions, because their past extends infinitely.  We define the notion of a local cosmological time function, $\tau: N\to (-\infty,\infty)$, in Definition~\ref{defn:local-cosmo}.  We prove that cosmic strips, $\tau^{-1}(\tau_0,\tau_1)$, defined by such time functions (when proper) are causally-null-compactifiable space-times when $\tau$ is a proper local cosmological time function in Theorem~\ref{thm:local-cosmo}.  We close with Questions~\ref{quest:Busemann} and~\ref{quest:local-unique}
on the existence and uniqueness of these local cosmological time functions.

In Section~\ref{sect:space-times-sub}
we consider asymptotically flat spaces
and define the notion of a null compactifiable sub-space-time
(see Definition~\ref{defn:sub-space-time}).  In Sections~\ref{sect:space-times-Minkowski} and~\ref{sect:space-times-Schwarzschild} we describe various regions in Minkowski space and Schwarzschild space which are causally-null-compactifiable sub-space-times.  We believe the method described there can be extended more generally to handle the exterior regions of more complicated space-times with many black holes and other asymptotically flat space-times, but leave this to others.
One challenge is defining the appropriate definition for the exterior region outside of a collection of black holes in a space-time.
In Section~\ref{sect:space-times-ex} we describe how one might possibly exhaust space-times with causally-null-compactifiable space-times in a canonical way.

We close with Sections~\ref{sect:space-times-measure} and~\ref{sect:space-times-current} describing work towards the development of timed metric measure spaces and timed integral current spaces, and finally a brief review of existing notions of weak space-times in Section~\ref{sect:space-times-other}.
Throughout Section~\ref{sect:space-times} we include open problems in addition to our theorems.

\begin{rmrk}\label{rmrk:bndry}
Note that causally-null-compactifiable space-times,
$(N,g)$ have a boundary set $\partial N \subset \bar{N}\setminus N$.  It would be interesting to explore how this boundary is related to prior notions of boundary in
work of Geroch, Kronheimer, Penrose, Beem and Harris
\cite{Geroch-Kronheimer-Penrose-Ideal}
\cite{Beem-topology}
\cite{Harris-Topology-boundary}.
See also the
$C^0$ extensions to boundaries 
as in the work of 
Chrusciel, Grant, Graf, van den Beld-Serrano, Ling,
Saemann, and Sbierski in
\cite{Chrusciel-2010-JDG}
\cite{Chrusciel-Grant-continuous}
\cite{Ling-Aspects}
\cite{GBS}\cite{Saemann-Global}\cite{Sbierski-extensions}.  See also the boundaries studied in the work of Costa e Silva, Flores, Herrera, and Burgos in \cite{CeSFH-c-bndry} \cite{CeSFH-c-bndry-add} 
\cite{Burgos-Flores-Herrera}\cite{Flores-boundary}.  
There should be a Lipschitz one map from their notions of boundary to ours and, perhaps under certain hypotheses, an isometry.
\end{rmrk}

\subsection{{\bf Metric-Space-Times defined by Cosmic Strips}}
\label{sect:space-times-induced}

In this subsection we prove Theorem~\ref{thm:induced}. It can help to keep a very simple example in mind while reading the proof so we present an example first.   Recall Definition~\ref{defn:null-distance}.

\begin{ex}\label{ex:induce}
Let $(N,g)$ be a big bang space-time with a linear warping factor,
\be
N=(0,t)\times {\mathbb S}^1
\textrm{ and } g=-d\tau^2+ \tau^2 g_{{\mathbb S}^1},
\ee
which is easy to work with and which is topologically the same as a disk with the center point removed.
Then the cosmological time function is $\tau_g(\tau,\theta)=\tau$.   On this space one can integrate along the causal curves to compute:
\begin{eqnarray}
J_{-}((\tau_0,\theta_0))&=&
\{(\tau,\theta):\, 0<\tau\le \tau_0 \exp(-|\theta-\theta_0|)\}\\
J_{+}((\tau_0,\theta_0))&=&
\{(\tau,\theta):\, \tau\ge \tau_0 \exp(|\theta-\theta_0|)\}
\end{eqnarray}
as depicted on the left of
Figure~\ref{fig:induce1}.
Note that $J_{-}((\tau_0,\theta_0))$ contains an open neighborhood around the origin, so
\begin{eqnarray}
\hat{d}_{g}((\tau_0,0),
(\tau_0,\pi))&\le&
|\tau_g(\tau_0,0)-\tau_g(\tau_0 e^{-\pi},\pi/2)|\\
&&\quad +
|\tau_g(\tau_0 e^{-\pi},\pi/2)-\tau_g(\tau_0,\pi)|
<2\tau_0,
\end{eqnarray}
as depicted on the right in Figure~\ref{fig:induce1}.

\begin{figure}[h] 
   \centering
   \includegraphics[width=2.8in]{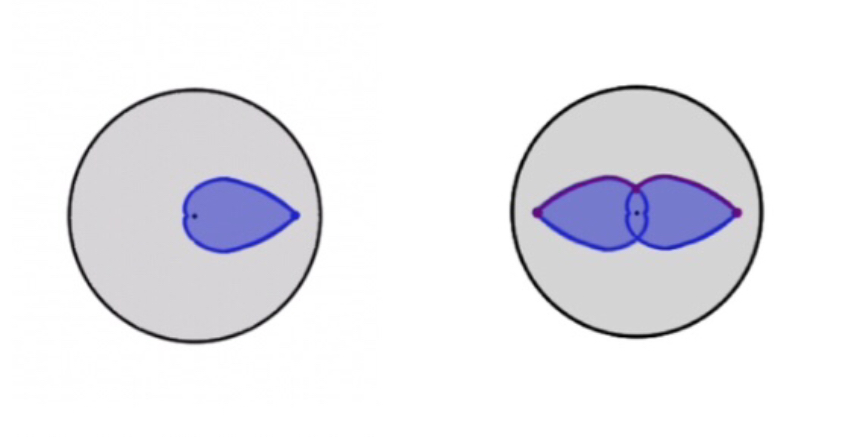} 
   \caption{Past sets of points in $N$ of Example~\ref{ex:induce} are easily computed.  The past sets of antipodal points intersect allowing for easy estimation of $\hat{d}_{g}$.
   }
\label{fig:induce1}
\end{figure}
Choose $t>s>0$ with $s>\tau_0 e^{-\pi/2}$, so that 
\be
J_{-}((\tau_0,\theta_0))\cap
J_{-}((\tau_0,\theta_0+\pi/2))\cap
\tau_{g}^{-1}(s,t)=\emptyset,
\ee
as depicted on the left in Figure~\ref{fig:induce2}.
\begin{figure}[h] 
   \centering
   \includegraphics[width=4.5in]{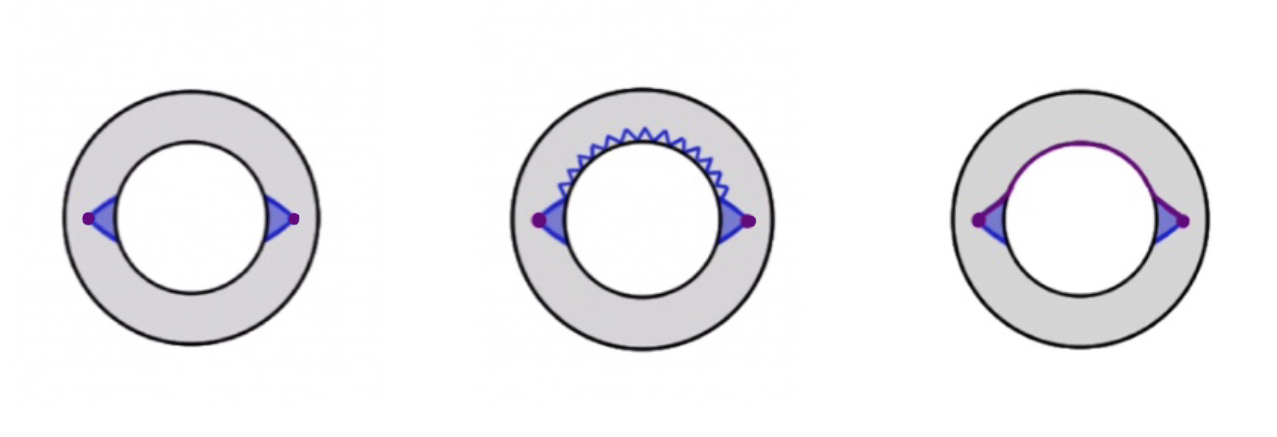} 
   \caption{Past sets of antipodal points in $N_{s,t}$ of Example~\ref{ex:induce} do not intersect so we need to take many
   zig-zags to estimate $\hat{d}_{g,s,t}$ and see that this distance is only achieved by a curve that leaves $N_{s,t}$, and in fact,
   $\hat{d}_{g,s,t}>\hat{d}_{g}$.
   }\label{fig:induce2}
\end{figure}
Then we need more than one point
to compute the null distance on $N_{s,t}=\tau_g^{-1}(s,t)$.  Recall that for any $\epsilon>0$ there
exists points,
\be
\{p=(\tau_0,0), p_1=(\tau_1,\theta_1),p_2=(\tau_2,\theta_2),...,p_{2N}=(\tau_0,\pi)\}\subset N_{s,t},
\ee
such that
\be
\left| \hat{d}_{g,s,t}((\tau_0,0),
(\tau_0,\pi))- \sum_{i=1}^{2N}|\tau_i-\tau_{i-1}|\right|<\epsilon
\ee
where each point is causally related to the previous point.  We can choose $\tau_i$ to
alternate in $(s,t)$
\be
\tau_0>\tau_1<\tau_2>\cdots<\tau_{2N}=\tau_0
\ee
and we can choose $\theta_i$ increasing as depicted in the center in Figure~\ref{fig:induce2}.   So if we connect these points by causal curves we get a piecewise causal curve, 
\be
C_\epsilon:[0,\pi]\to N_{s,t} \textrm{ where } C_\epsilon(\theta)=(\tau_\epsilon(\theta), \theta),
\ee
whose rectifiable length with respect to $\hat{d}_{g}$ is
\be
L_{\hat{d}_{g}}(C_\epsilon)=\sum_{i=1}^{2N}|\tau_i-\tau_{i-1}|
\ee
which agrees with its rectifiable
length with respect to $\hat{d}_{g,s,t}$.
Due to the shape of the past sets, as $\epsilon\to 0$, $C_\epsilon$, will converge to a curve, $C:[0,\pi]\to \tau^{-1}[s,t]$ that runs along a null curve directly to $(s,\theta_C)$ and then runs along the $\tau^{-1}(s)$ level set to $(s, \pi-\theta_C)$ 
and back up along a null curve as depicted on the right in Figure~\ref{fig:induce2}.  Its
$\hat{d}_{g}$ rectifiable length is:
\be
L_{\hat{d}_{g}}(C)=
|\tau_0-s|+s(\pi - 2\theta_C)+ |\tau_0-s|
\ee
where
\be
s=\tau_0 \exp(-|\theta_C|)\,\, \implies\,\, 
\theta_C=-\Ln(s/\tau_0)
\ee
Putting this all together we see that not only
is the distance between these points not achieved as
the length of a curve in $N_{s,t}$ but also that
\begin{eqnarray}
\hat{d}_{g,s,t}((\tau_0,0),
(\tau_0,\pi))&=&L_{\hat{d}_{g}}(C)\\
&=&2(\tau_0-s) + s(\pi+\Ln(s/\tau_0) 
\end{eqnarray}
decreases as $s$ decreases so that
\be
\hat{d}_{g,s,t}((\tau_0,0),
(\tau_0,\pi))>\hat{d}_{g}((\tau_0,0),
(\tau_0,\pi)).
\ee
\end{ex}

We now prove Theorem~\ref{thm:induced}.

\begin{proof}
Let us start with our full space-time, $(N,g)$, and its
cosmological time function,
\be
\tau_g(p)=\tau_{N,g}(p)=\sup
\{d_g(p,q):q\in J^-(p)\subset N\}
\ee
and its null distance, $\hat{d}_{g}$.  

Since the cosmological time function is regular, by Anderson-Galloway-Howard \cite{AGH}
there is a past inextensible timelike generator, $\gamma_p:(0,\tau_g(p))\to N$ such that $\gamma_p(\tau_g(p))=p$
and $\tau_g(\gamma_p(t))=t$.  

Consider $p\in N_{s,t}$, then
\be
\tau_{g,s,t}(p)=\sup
\{d_g(p,q):q\in J^-(p)\subset N_{s,t}\}
\ge \tau_g(p)-s
\ee
because we can take $q_j=\gamma_p(s_j)$
where $s_j$ decrease to $s$.
On the
other hand, for any $q\in J^-(p)$, the Lorentzian distance
\be
d_q(p,q) \le \tau_g(p)-\tau_g(q)=\hat{d}_{g}(p,q)
\ee
so
\be
\tau_{g,s,t}(p)\le 
\sup
\{\tau_g(p)-\tau_g(q):q\in J^-(p)\subset N_{s,t}\}
=\tau_g(p) -s
\ee
because we can take $q_j=\gamma_p(s_j)$
where $s_j$ decrease to $s$. 
Thus we have (\ref{eq:taust-1}) and
so $\tau_{g,s,t}$ is also a proper cosmological time function.

Let us now consider the null distance
$\hat{d}_{g,s,t}$ defined using $\tau_{g,s,t}$ on $N_{s,t}$.   
By (\ref{eq:taust-1}), for any piecewise
null curve $C:[0,1]\to N_{s,t}$, its null length defined with respect to $\tau_g$
equals its null length with respect to 
$\tau_{g,s,t}$:
\begin{eqnarray}
\hat{L}_{\tau_g}(C)&=&\sum_{i=1}^N
|\tau_g(C(t_{i-1}))-\tau_g(C(t_i))|\\
&=&\sum_{i=1}^N
|\tau_{g,s,t}(C(t_{i-1}))-\tau_{g,s,t}(C(t_i))|=\hat{L}_{\tau_{g,s,t}}(C).
\end{eqnarray}
In Lemma 3.1 of \cite{SakSor-Null},
such a piecewise causal $C$ is rectifiable with respect to that null distance and
\be
L_{\hat{d}_{g}}(C)=\hat{L}_{\tau_g}(C)
\textrm{ and }
L_{\hat{d}_{g,s,t}}(C)=\hat{L}_{\tau_{g,s,t}}(C)
\ee
so that $\hat{d}_{g}$ is a length metric whose distance is the infimum over the lengths of all its $\hat{d}_{g}$-rectifiable curves
which agrees with the infimum over the lengths of all its piecewise null curves, and the same is true for 
$\hat{d}_{g,s,t}$.   Since all these
lengths agree and the only difference in the definitions is the allowed range for the curves, we see that
 $\hat{d}_{g,s,t}$ is the
induced length metric on $N_{s,t}$
as described in (\ref{eq:induced}).
\end{proof}

\subsection{{\bf Space-Times with Proper Regular Cosmological Time Functions}}
\label{sect:space-times-proper}

Here we prove the following theorem:

\begin{thm} \label{thm:proper}
Suppose $(N,g)$ has a proper regular cosmological time function, $\tau_g:N\to (0,\infty)$, and $0<s<t<\infty$ then the cosmic strip, $(N_{s,t},g)$, is a causally-null-compactifiable space-time,
with a proper cosmological time function,
$\tau_{g,s,t}: N_{s,t}\to (0,t-s)$ and a 
compact associated timed-metric-space,
$(\bar{N}_{s,t},\hat{d}_{g,s,t},\tau_{g,s,t})$, satisfying all the properties described in Theorem~\ref{thm:induced}.
Then we can extend $\tau_{g,s,t}: N_{s,t}\to (0,t-s)$ as a Lipschitz $1$ function with respect to
$\hat{d}_{g,s,t}$ to the metric completion, $\tau_{g,s,t}:\bar{N}_{s,t}\to [s,t]$, has the distance from initial data property, which we define as: 
\be\label{eq:init-data-prop}
\tau_{g,s,t}(p)=\min\left\{
\hat{d}_{g}(p,q):\, q\in \tau_{g,s,t}^{-1}(s)\subset \bar{N}_{s,t}\right\}.
\ee
The causal structure defined by
\be \label{eq:causal-bar}
p\in J^+(q) \iff \tau_{g,s,t}(p)-\tau_{g,s,t}(q)=\hat{d}_{g,s,t}(p,q)
\ee
extends the causal structure on $(N_{s,t},g)$
to $(\bar{N}_{s,t},\hat{d}_{g,s,t},\tau_{g,s,t})$.
For all $p,q\in \bar{N}_{s,t}$, there is a 
$\hat{d}_\tau$-length minimizing curve joining
the points and achieving the $\hat{d}_\tau$ distance between them (but is not necessarily a $g$-geodesic).  Finally, $\forall \, p\in \bar{N}_{s,t}$,
there exists a Lipschitz $1$ generator, 
$\bar{\gamma}_p:[s, \tau_g(p)]\to \bar{N}_{s,t}$,
which agrees with a $\tau_g$ generator $\gamma_p$ for $p$  when $p\in N_{s,t}\subset N$ and is also a $\hat{d}_\tau$-length minimizing curve such that 
\be\label{eq:generator-bar}
\tau_{g,s,t}(\bar{\gamma}_p(t'))=t'-s
\qquad \forall \, t'\in [s,\tau_{g,s,t}(p)+s].
\ee
\end{thm}

Before we begin, we present an illustrative example.

\begin{ex} \label{ex:proper}
Let $(N,g)$ be the big bang space-time with a linear warping factor,
\be
N=(0,t)\times {\mathbb S}^1
\textrm{ and } g=-d\tau^2+ \tau^2 g_{{\mathbb S}^1}
\ee
as in Example~\ref{ex:induce}.
Then the cosmological time function, $\tau_g(\tau,x)=\tau$, is proper and regular 
with circular level sets and the
generators are radial:
\be
\gamma_{\tau_0,x_0}(t')=(t',x_0) \textrm{ for } t'\in (0, \tau_0],
\ee
as depicted on the left side in Figure~\ref{fig:proper}.
The topology on $N$ agrees with that of 
${\mathbb D}^2\setminus\{0\}$, and so is the topology induced by $\hat{d}_{g}$, so if we take the metric completion, $\bar{N}$, we add in the point $p_0=0$, and we can consider this point to be the big bang itself. 
In $\bar{N}$ we can take the limit, 
\be
\lim_{t'\to 0} \gamma_{\tau_0,x_0}(t')=p_0
\ee
and we can extend $\tau_{g}$ to $\bar{N}$ as a Lipschitz one function, to be equal to $0$ at $p_0$, so that $\tau_g$ is
the $\hat{d}_{g}$ distance to this point:
\be
\tau_g^{-1}(0)\subset\bar{N}=\{p_0\}.
\ee

\begin{figure}[ht] 
   \centering
   \includegraphics[width=4.5in]{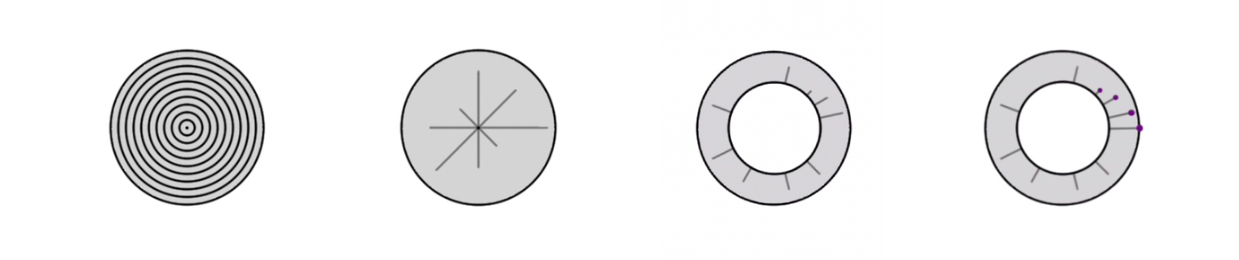} 
   \caption{Example~\ref{ex:proper}: On the left we see the level sets of $\tau_g$ in $N$ and its generators in $N$ achieving the distance to the big bang point at the center. On the right, we see
   the generators in $N_{s,t}$ achieving the distance to the initial set on the inner ring, and sequences of generators converging to a generator starting at a point in $\bar{N}_{s,t}$ on the outer ring.}\label{fig:proper}
\end{figure}

If we consider $0<s<t<\infty$, then $N_{s,t}$ has the topology of an annulus and $\tau_{g,s,t}(\tau,\theta)=\tau-s$ and $\bar{N}_{s,t}$
has the topology of the closed annulus. 
In $\bar{N}_{s,t}$ we can take the limit, 
\be
\lim_{t'\to s} \gamma_{\tau_0,x_0}(t')=
\bar{\gamma}_{\tau_0,x_0}(s)
\ee
We can extend $\tau_{g,s,t}$ to $\bar{N}_{s,t}$ as a Lipschitz one function, to be equal to $0$ on the inner boundary (which can be considered to be the initial surface) so that $\tau_{g,s,t}$ is
the $\hat{d}_{g,s,t}$ distance to this initial surface:
\be
\tau_{g,s,t}^{-1}(0)=\tau_g^{-1}(s).
\ee
We can also extend $\tau_{g,s,t}$ to the outer boundary,
\be
\tau_{g,s,t}^{-1}(t-s)=\tau_g^{-1}(t).
\ee
and define generators for points on the outer boundary by taking limits of generators, as depicted on the right in Figure~\ref{fig:proper}.
\end{ex}

We now prove Theorem~\ref{thm:proper}:

\begin{proof}
We have all the properties of
Theorem~\ref{thm:induced}.  Applying work of Sormani-Vega in \cite{SV-Null}, 
 $\hat{d}_{g,s,t}$ is a definite metric on $N_{s,t}$ and 
$\tau_{g,s,t}: N_{s,t}\to (0,t-s)$
in Lipschitz one with respect to $\hat{d}_{g,s,t}$.   Applying work of Anderson-Galloway-Howard, 
for all $\forall \, p\in \bar{N}_{s,t}$,
there exists a $\tau_g$ generator, 
$\gamma_p:(0, \tau_g(p)]\to {N}$,
which is a timelike past inextensible geodesic such that
\be\label{eq:generator-1}
\tau_{g,s,t}(\gamma_p(t'))=
\tau_g(\gamma_p(t'))-s=t'-s
\ee
for all $t'\in (s,\tau_{g,s,t}(p)+s]$.
Since $\gamma_p$ is causal, we see by Lemma 3.11 of
\cite{SV-Null} that:
\be
\hat{d}_{g,s,t}(\gamma_p(s'),\gamma_p(t'))=
|\tau_{g,s,t}(\gamma_p(s'))-
\tau_{g,s,t}(\gamma_p(t'))|=|s'-t'|
\ee
for all $s',t'\in (s,\tau_g(p)]$.

We now construct the metric completion,
$(\bar{N}_{s,t},\hat{d}_{g,s,t})$. See Section~\ref{sect:back-completion} for the definition and properties of metric completions.   In particular,
\be
\forall \, p\in \bar{N}_{s,t}
\,\,\exists \, q_i \in N_{s,t}
\textrm{ such that } \hat{d}_{g,s,t}(q_i,p)\to 0.
\ee
Since $\tau_{g,s,t}$ is Lipschitz one, we can uniquely define its 
Lipschitz one extension
to the metric completion,
$\tau_{g,s,t}:\bar{N}_{s,t}\to [s,\tau_g(p)]$.
Note that,
\be
\tau_{g,s,t}(p)=\lim_{i\to \infty}
\tau_{g,s,t}(p_i).
\ee

Now we will prove that the metric completion, $(\bar{N}_{s,t}, \hat{d}_{g,s,t})$ of
$(N_{s,t}, \hat{d}_{g,s,t})$ is compact. Consider any $p_i\in \bar{N}_{s,t}$,
we need only show that a subsequence converges in $\bar{N}_{s,t}$.
Let
$q_i\in {N}_{s,t}$ such that
$\hat{d}_{g,s,t}(p_i,q_i)\to 0$.
We need only show a subsequence of $q_i$ converges in $\bar{N}_{s,t}$.  

Since $\tau_i=\tau_g(q_i)\in [s,t]$, a subsequence of the $q_i$ also denoted $q_i$ has
\be
\tau_i=\tau_g(q_i)\to \tau_\infty \in [s,t].
\ee
Consider the $\tau_{g}$-generators of $q_i$ as $\hat{d}_{g}$-Lipshitz one maps:
\be \label{eq:gen-1}
\gamma_{q_i}:[s,\tau_g(q_i)]\to \tau_g^{-1}[s,t]\subset N
\textrm{ s.t. } \tau_g(\gamma_{q_i}(t'))=t'.
\ee
Since $\tau_g^{-1}[s,t]\subset N$ is compact, by the Arzela-Ascoli theorem, a subsequence of these generators, also denoted $\gamma_{q_i}$, converges to a 
$\hat{d}_{g}$ Lipschitz one curve, 
\be \label{eq:gen-2}
\gamma_\infty:[s,\tau_\infty]\to \tau_g^{-1}[s,t]\subset N
\textrm{ s.t. } \tau_g(\gamma_\infty(t'))=t'.
\ee
In particular there exists $\epsilon_i \to 0$ such that:
\be\label{eq:arzasc}
\sup\hat{d}_{g,s,t}\left(\gamma_{p_i}(s'),\gamma_\infty(s')\right)\le \epsilon_i\to 0,
\ee
where the sup is over all $s'\in [s,t]$.

We claim the $q_i$ of this subsequence converge. If we restrict the limit curve, 
\be\label{eq:gen-3}
\gamma_\infty:(s,\tau_\infty)\to N_{s,t}\subset N
\textrm{ s.t. } \tau_{g,s,t}(\gamma_\infty(t'))=t'-s.
\ee
we can use the fact that it is Lipschitz one to extend it as a Lipschitz one function to the metric completion,
\be \label{eq:gen-4}
\bar{\gamma}_\infty:[s,\tau_\infty]\to \bar{N}_{s,t}
\textrm{ s.t. } \tau_{g,s,t}(\bar{\gamma}_\infty(t')=t'-s.
\ee
We claim $q_i \to p=\bar{\gamma}_\infty(\tau_\infty)\in \bar{N}_{s,t}$.
To see this we apply the triangle inequality:
\begin{eqnarray*}
\hat{d}_{g,s,t}(q_i,p)&=&
\hat{d}_{g,s,t}(\bar{\gamma}_i(\tau_i),\bar{\gamma}_\infty(\tau_\infty))\\
&\le &
\hat{d}_{g,s,t}(\bar{\gamma}_i(\tau_i),\bar{\gamma}_\infty(\tau_i))
+
\hat{d}_{g,s,t}(\bar{\gamma}_\infty(\tau_i),\bar{\gamma}_\infty(\tau_\infty))\\
&\le & 
|\tau_i-\tau_\infty|+\epsilon_i \to 0.
\end{eqnarray*}
In the last line above, we apply (\ref{eq:arzasc}) and the fact that
\be
\hat{d}_{g,s,t}(\bar{\gamma}_i(\tau_i),\bar{\gamma}_\infty(\tau_i))=
\hat{d}_{g,s,t}({\gamma}_i(\tau_i),{\gamma}_\infty(\tau_i))
\ee
because $\tau_i\in (s,t)$.  Thus we have shown $\bar{N}_{s,t}$ is compact.

Note that for any $p\in \bar{N}_{s,t}$,
we have $q_i \to p$, where $q_i\in N_{s,t}$ 
have generators $\gamma_{q_i}$. Furthermore, $\tau_g(p)=\lim_{i\to \infty}\tau_g(q_i)$.  We can repeat the process above in 
(\ref{eq:gen-1})-(\ref{eq:gen-3}) 
to find a Lipschitz one curve as in (\ref{eq:gen-4}) such that 
\be
p=\bar{\gamma}_\infty(\tau_\infty)
=\bar{\gamma}_\infty(\tau_g(p)).
\ee
This curve is not unique but it does define a generator for $p$, 
\be
\bar{\gamma}_p(t'):=\bar{\gamma}_\infty(t') \quad \forall \, t'\in [s,\tau_g(p)]=[s,\tau_{g,s,t}(p)+s],
\ee
satisfying (\ref{eq:generator-bar}).


By the Sakovich-Sormani Causality Theorem
in \cite{SakSor-Null}, we have
\be
\tau_{g,s,t}(p)=\tau_{g,s,t}(q)+\hat{d}_{g,s,t}\tau(p,q)
\iff q\in J^-(p).
\ee
We can extend this continuously to define a causal structure on $\bar{N}_{s,t}$.

We now prove that for any $p,q\in \bar{N}_{s,t}$ there
is a $\hat{d}_\tau$-length minimizing curve
between them which achieves the distance between them.  By Theorem~\ref{thm:induced}, we know
$(N_{s,t}, \hat{d}_\tau)$ is a length space. Its metric completion, $(\bar{N}_{s,t}, \hat{d}_\tau)$, is then also a length space.  
Since it is compact, it has length minimizing curves which are distance achieving (see for example Theorem 2.5.23 from the text of Burago-Burago-Ivanov \cite{BBI}).   

Finally we prove the initial data property in (\ref{eq:init-data-prop}).  For any $p\in N_{s,t}$, we take $q_j={\gamma}_p(s_j)$ with $s_j$ decreasing to $s$ we see that
\be
\tau_{g,s,t}(p)=\tau_{g,s,t}(q_j)+\hat{d}_{g,s,t}\tau(p,q_j)
= s + \hat{d}_{g,s,t}\tau(p,q_\infty),
\ee
where $q_\infty\in \tau_{g,s,t}^{-1}(s)\subset \bar{N}_{s,t}$.
Since cosmological times are Lipschitz one with respect to their null distances, for any $q\in \tau_{g,s,t}^{-1}(s)\subset \bar{N}_{s,t}$, we have
\be
\tau_{g,s,t}(p)\le s+\hat{d}_{g,s,t}\tau(p,q)
\ee
Thus $\tau_{g,s,t}$ satisfies (\ref{eq:init-data-prop}) for any $p\in N_{s,t}$.   This extends continuously to $p\in \bar{N}_{s,t}$.   
\end{proof}

\subsection{{\bf Big Bang Space-Times}}
\label{sect:space-times-big-bang}

We introduce the following notion which is consistent with the notion of a big bang FLRW space-time but applies to a much larger class of space-times.

\begin{defn} \label{defn:BB}
A {\bf causally-null-compactifiable space-time has a big bang} if its associated compact metric-space, $(\bar{N},\hat{d}_g,\tau_g)$,
is a {\bf big bang timed-metric-space}: which means it 
has an initial level set, $\tau^{-1}(0)\subset \bar{N}$, that is a single point, $p_{BB}\in \bar{N}$, called the big bang point,
and 
its cosmological time function is the
distance from the big bang point:
\be\label{eq:dist-BB}
\tau(p)=\hat{d}_\tau(p,p_{BB}).
\ee
\end{defn}

\begin{conj}\label{conj:BB}
If $(N,g)$ has a smooth proper regular cosmological time function, $\tau:N \to (0,\tau_{max})$, and the smooth level sets,
$(M_t=\tau^{-1}(t), h_t)$ with the induced Riemannian metric have
\be \label{eq:diam-BB}
\lim_{t\to 0}\diam_{h_t}(M_t)=0
\ee
then $(N,g)$ is a causally-null-compactifiable space-time with a big bang.
\end{conj}

\begin{rmrk}\label{rmrk-SV-FLRW}
In unpublished work, Sormani-Vega proved the
above conjecture in the special case where
\be
g=-dt^2+f^2(t)\,h_M
\ee
where $h_M$ is the metric tensor of a homogeneous Riemannian manifold, as in
FLRW space-times, and $\lim_{t\to 0}f(t)=0$
\cite{SV-BigBang}.
Only the causality was published as 
Lemma 3.24 in  \cite{SV-Null}.   The rest was done by estimating the null distance on such warped product manifolds.   Such estimates have been published independently in work of Allen-Burtscher \cite{Allen-Burtscher-22}.   The above conjecture can be proven by first applying Theorem~\ref{thm:induced} and then Theorem~\ref{thm:proper}.  The challenge is to prove that there is a single point in the level set $\tau_g^{-1}(0)$ even with only the diameter estimate on the level sets.
\end{rmrk}

\subsection{{\bf Future Developments from Compact Initial Data Sets}}
\label{sect:space-times-initial-data}

Given a compact initial data set, $(M,h,k)$, satisfying the Einstein constraint equations, 
Choquet-Bruhat proved there is a globally hyperbolic space-time, $(N,g)$,
satisfying the Einstein vacuum equations such that $M$ lies as a smooth Cauchy surface in $N$ with restricted metric tensor $h$ and second fundamental form $k$
\cite{Ch-Br}. Similar solutions have been found satisfying Einsteins Equations for a variety of matter models.

Such spaces have future sets,
\be
N_+=(J^+(M)\setminus M) \subset N
\ee
with regular cosmological time functions, $\tau_g:N_+\to (0,\infty)$, whose generators,
$\gamma_p:(0,\tau_g(p))\to N_+$
can be extended in $N$ smoothly through $\gamma_p(0)\in M$ with
$\gamma_p'(0)$ perpendicular to $M$.   Keeping these key properties in mind we make the following definition:

\begin{defn} \label{defn:FD}
A causally-null-compactifiable future-developed space-time, $(N_+,g)$,
with initial surface, $(M,h,K)$,
has an associated compact timed-metric-space, $(\bar{N}_+,\hat{d}_g,\tau_g)$,
that is {\bf a future-developed timed-metric-space}: which means it has
 an initial level set, 
\be
\bar{M}=\tau_g^{-1}(0)\subset \bar{N},
\ee
where $\tau_{g}:\bar{N}_+\to [0,\tau_{max}]$
is the Lipschitz extension of $\tau_g$.
In addition, $\tau_g$ has the distance from initial data property: 
\be\label{eq:init-data-prop-2}
\tau_{g}(p)=\min\left\{
\hat{d}_{g}(p,q):\, q\in M\right\}.
\ee
and the causal structure 
\be \label{eq:causal-bar-2}
p\in J^+(q) \iff \tau_{g}(p)-\tau_{g}(q)=\hat{d}_{g}(p,q)
\ee
extends the causal structure on $(N_+,g)$
to $(\bar{N}_+,\hat{d}_{g},\tau_{g})$
so that $\bar{N}_+$ lies in the
causal future of $M$:
\be\label{eq:future-of-M}
\bar{N}_+=J^+(M)\subset \bar{N}_+.
\ee
Every $p\in \bar{N}_{s,t}$,
has a Lipschitz $1$ generator, 
$\bar{\gamma}_p:[0, \tau_g(p)]\to \bar{N}_+$,
which is the Lipschitz extension of a $\tau_g$ generator $\gamma_p$ for $p$  when $p\in N_+$ such that 
\be\label{eq:generator-bar-2}
\tau_{g}(\bar{\gamma}_p(t'))=t'
\qquad \forall \, t'\in [0,\tau_{g}(p)].
\ee
\end{defn}

The following theorem applies to a large class of space-times which are subdomains of
future developments

\begin{thm}\label{thm:FD}
Suppose $(M,h,K)$ is a compact
Cauchy surface with restricted metric, $h$, and second fundamental form, $K$, of a globally hyperbolic space-time, $(N,g)$, and suppose $N_+=J^+(M)\setminus M\subset N$
and cosmological time $\tau:N_+\to (0,\tau_{max})$ is proper.   Then
$(N_+,g)$ is a causally-null-compactifiable future-developed space-time
with initial surface, $(M,h,K)$
as in Definition~\ref{defn:FD}.
\end{thm}

\begin{proof}
We would like to use Theorem~\ref{thm:proper} with $s=0$ but that theorem requires $s>0$. 
So we need to extend our region below $M$.

For each point $q\in M$, we have normal geodesics,
$\eta_q:(-\epsilon_q,\epsilon_q) \to N
$ with no cut or conjugate points such that
$\eta_q(0)=q$,
$\eta'_q(0)$ is $g$-perpendicular to $TM_q$,
and
$g(\eta_q'(0),\eta_q'(0))=-1$.   Since $M$ is compact, we can find a minimum value for 
$\epsilon=\min\{\epsilon_q:\, q\in M\}>0$.
The function
\be
F:M \times(-\epsilon,\epsilon)\to U_\epsilon\subset N
\ee
is a diffeomorphism onto its image $U_\epsilon$ which is thus an open set about $M$.  See work of Treude regarding the
properties of this map \cite{Treude-Diploma}.

The region $N_{\epsilon,+}=N_+\cup U_\epsilon$
has a cosmological time function
\be
\tau_{\epsilon,g}: N_{\epsilon,+}\to (0,\tau_{max,\epsilon})
\textrm{ s.t. } \tau_{\epsilon,g}(F(q,t'))=t'+\epsilon
\ee
such that 
\be
\tau_{\epsilon,g}(p)=\tau_g(p)+\epsilon
\qquad \forall \, p\in N_+.
\ee
So $\hat{d}_{\epsilon,g}=\hat{d}_g$ on $N_+$.
and $\tau_{\epsilon,g}$ is
proper and regular. So we can apply 
Theorem~\ref{thm:proper} to $(N_{\epsilon,+},g)$ with $s=\epsilon$
and $t=\tau_{max,\epsilon}=\tau_{max}+\epsilon$.
Since $M=\tau_{\epsilon,g}^{-1}(\epsilon)$,
 (\ref{eq:init-data-prop})
 gives us (\ref{eq:init-data-prop-2})
 and
the generators extend
 satisfying all our claims including (\ref{eq:generator-bar-2}) from
 (\ref{eq:generator-bar}).
 
 We get our first
 causal claim in (\ref{eq:causal-bar-2}) from (\ref{eq:causal-bar}).  
To prove our second causal claim in (\ref{eq:future-of-M}), 
 we need only show
 \be\label{eq:future-of-M-2}
 \forall \, p\in \bar{N}_+ \,\,\exists \, q\in M\subset \bar{N}_+
 \textrm{ such that } 
 \tau_g(p)-\tau_g(q)=\hat{d}_\tau(p,q).
 \ee
 We know there exists $\bar{\gamma}_p:[0,\tau_g(p)]\to \bar{N}_+$, so we
 can take $q=\bar{\gamma}_p(0)\subset M$
 and use the fact that $\gamma_p$ is causal
 to conclude (\ref{eq:future-of-M-2}).
\end{proof}

\begin{rmrk}\label{rmrk:low-reg-init-data}
It would be interesting to extend Theorem~\ref{thm:FD} above to space-times with lower regularity initial data like those mentioned in Remark~\ref{rmrk:tau-compact-stable} below.
\end{rmrk}

The next two examples can be useful to keep in mind: 

\begin{ex}\label{ex:lowerhyp1}
Consider the lower hyperboloid,
\be
M=\{(t,x):\, t=-\sqrt{1+|x|^2}\}
\subset {\mathbb R}^{1,3},
\ee
with $h$ restricted from the
Minkowski metric, $g_{Mink}$, and $K$ equal to the second fundamental form of $M$ as a subset of Minkowski space, $({\mathbb R}^{1,3}, g_{Mink})$.   The global maximal development is
not full Minkowski space but instead:
\be
N=\{(t,x):\, t<-|x|\}\subset {\mathbb R}^{1,3}.
\ee
The future maximal development is
\be
N_+=\{(t,x):\, -\sqrt{1+|x|^2}<t<-|x|\}\subset {\mathbb R}^{1,3}.
\ee
and the generators, $\gamma_p$,
extend radially away from the origin to meet $M$ perpendicularly with respect to $g_{Mink}$.  The cosmological time function,
$\tau_{N_+,g}:N_+\to (0,1)$,
has level sets that are hyperboloids.   Changing coordinates, we see that,
$N_+=(0,1)\times M$ with
\be
g= -d\tau_{N_+,g}^2 + (1-\tau_{N_+,g})^2 h 
\ee
This is a classic FLRW space-time with hyperbolic space as an initial data set and a big crunch 
at $\tau_{N_+,g}=1$.   The null distance
is now easy to study and we can see that, $(N_+,\hat{d}_{N_+,g})$, is a metric-space whose metric completion has a single point in
the level set, $\tau_{N_+,g}^{-1}(1)$.  Note that the metric completion also includes a copy of the initial surface, but that it does not include any additional points in the level sets between zero and one because hyperbolic space is already complete.
\end{ex}

\begin{ex}\label{ex:lowerhyp2}
To make a compact example, we can take any compact hyperbolic manifold as our initial data.  Its universal cover is hyperbolic space, $(M,h)$, we can
use the same tensor, $K$, as in Example~\ref{ex:lowerhyp1} and then divide $N_+\subset N$ by the same group of isometries to find its future maximal development.
\end{ex}

\begin{rmrk}\label{rmrk:cmpct-remove-holes}
Suppose $(N_+,g)$ is a future maximal development of the Einstein equations containing a collection of black holes.
It would be natural to consider, $(N'_+,g)$, which has the interior of the black holes removed.   It would be of interest to show, $(N'_+,g)$, is a
causally-null-compactifiable future-developed space-time.   We discuss this more in Remark~\ref{rmrk:exterior-region}
for regions in asymptotically flat space-times.  It is also of interest when studying future developments of compact initial data sets.
\end{rmrk}

\subsection{{\bf Cosmic Strips}}
\label{sect:space-times-strip}

Many space-times like Minkowski space do not have a finite cosmological time function. However, such space-times might be studied using a null distance defined by a time function with the following properties:

\begin{defn}\label{defn:local-cosmo}
A time function, $\tilde{\tau}: N \to (-\infty,\infty)$, is a {\bf locally cosmological
time function}, if for all $(s,t)\subset (-\infty,\infty)$ the cosmic strip, $N_{s,t}=\tilde{\tau}^{-1}(s,t)\subset N$,
endowed with the restricted Lorentzian metric tensor $g$ has a bounded cosmological time function, $\tau_{g,s,t}:N_{s,t}\to (0, t-s)$, such that
\be \label{eq:tau-shift-3}
\tau_g(p)=\tilde{\tau}(p)-s.
\ee
The locally cosmological time function, $\tilde{\tau}$,
is regular (respectively proper) iff $\tau_{g,s,t}$ is regular (respectively proper) for all $(s,t)\subset (-\infty,\infty)$.
\end{defn}

The following theorem is an immediate consequence of the above definition combined with
Theorem~\ref{thm:induced} and Theorem~\ref{thm:proper}:

\begin{thm}\label{thm:local-cosmo}
If a space-time, $(N,g)$, has a proper regular
locally cosmological time function, $\tilde{\tau}:N\to (-\infty,\infty)$, then for all $(s,t)\subset (-\infty,\infty)$, the cosmic strip, $(N_{s,t},g)$ is a causally-null-compactifiable space-time satisfying (\ref{eq:taust-1}),
whose null distance, $\hat{d}_{g,s,t}$, agrees with the induced length distance defined using $\hat{d}_{\tilde{\tau},g}$-rectifiable lengths 
of curves restricted to $N_{s,t}$:
\be
\sup \sum_{i=1}^N
\hat{d}_{g,s,t}(C(r_i),C(r_{i-1}))
=
\sup \sum_{i=1}^N
\hat{d}_{\tilde{\tau},g}
(C(r_i),C(r_{i-1}))
\ee
where the suprema are both over all
partitions, $0=r_0<r_1<r_2<\cdots r_N=1$,
and over all $N\in {\mathbb N}$.

Furthermore, the
associated compact timed-metric-space, 
\be
(\bar{N}_{s,t},\hat{d}_{g,s,t},\tau_{g,s,t}),
\ee
has generators satisfying (\ref{eq:generator-bar}), has the distance from initial data property as in (\ref{eq:init-data-prop}),
and has the same causal structure as in (\ref{eq:causal-bar}) extended from the causal structure on $N$.
\end{thm}

Later we will define distances between null compactifiable space-times and we can in particular apply those distances to two different cosmic strips in the same space-time.

\begin{quest}\label{quest:Busemann}
Do all globally hyperbolic space-times have regular locally cosmological time functions?
Are Lorentzian Busemann functions regular locally cosmological time functions?
\end{quest}

\begin{quest}\label{quest:local-unique}
If a space-time has a proper regular locally cosmological time function, is it unique up to a constant?
\end{quest}

\begin{rmrk}\label{rmrk:local-unique}
Note that without the properness assumption, the answer to the above question is no.  For example, on Minkowski space-time, $({\mathbb R}^2, -dt^2+dx^2)$, for any $a,b$ such that $a^2-b^2=1$ the function, 
\be
\tilde{\tau}(x,t)=at+bx,
\ee
is a regular locally cosmological time function,
as can be seen by taking a Lorentzian transformation. 
\end{rmrk}

\begin{rmrk}\label{rmrk:slab}
There is a notion called ``Cauchy slab''
studied by M\"uller in various papers
including \cite{Mueller-Maximality}.   A cosmic strip,
$\tau^{-1}(t_1,t_2)$ 
is a Cauchy slab when 
$\tau^{-1}(t_1)$ and $\tau^{-1}(t_2)$
are smooth Cauchy surfaces.  It would be interesting to relate his work with ours. 
\end{rmrk}

\subsection{{\bf Sub-Space-Times that are Causally-Null-Compactifiable}}\label{sect:space-times-sub}

In this subsection we introduce a larger class of sub-space-times which are not just cosmic strips bounded in time, but are compactified by bounding spacelike directions as well.   We state a few definitions here and then present examples of sub-space-times in Minkowski and Schwarzschild space in the next subsections.

\begin{defn}\label{defn:sub-space-time} 
A region, $N$,
within a smooth space-time, $(\tilde{N},g, \tilde{\tau})$, 
that has a time function, 
$\tilde{\tau}: \tilde{N} \to (-\infty,\infty)$, 
defining a null distance, 
\be
\hat{d}_{g,\tilde{\tau}}: \tilde{N}\times \tilde{N}\to [0,\infty),
\ee
is called a 
{\bf sub-space-time} if its cosmological time function, $\tau_{g}:N\to [0,\infty)$, is regular and bounded, and shifted:
\be \label{eq:tau-shift-4}
\tau_g(p)=\tilde{\tau}(p)-\tilde{\tau}_{min}
\ee
where $\tilde{\tau}_{min}=\min_N \tilde{\tau}$,
so that its null distance, $\hat{d}_{N,g}$,
is the induced $\hat{d}_{\tilde{\tau}}$-rectifiable length metric on $N$.   It is a {\bf null compactifiable sub-space-time} if its
{\bf associated metric-space},
$\left(\bar{N}, \hat{d}_{g} \right)$, is compact
and it is a {\bf causally-null-compactifiable sub-space-time}
if it encodes causality:
\be \label{eq:causal-sub}
p\in J^+_N(q) \iff \tau_{g}(p)-\tau_{g}(q)=\hat{d}_{N,g}(p,q) \quad \forall \, p,q \in N
\ee
where $J^+_N(q)$ is the induced causal future of $q$ in $N$ defined using causal curves in $N$ rather than the restricted causal structure from
$\tilde{N}$.
\end{defn}

\begin{rmrk}\label{rmrk:sub-space-time}
Note that in the definition above we are  not assuming that $\tilde{\tau}$ is a regular cosmological time or a
locally cosmological time function.  Nor are we assuming that
its associated null distance, $\hat{d}_{g,\tilde{\tau}}$,
is definite nor do we assume it encodes causality. We only
require these additional properties for the cosmological time of the region, $N$, and we only study the causal structure within $N$ defined using causal curves within $N$.  This allows us to study regions, $N$, within space-times, $\tilde{N}$,
that are not well understood as a whole. 
\end{rmrk}

\begin{ex} \label{ex:sub-space-times}
Future Minkowski space-time, ${\mathbb R}^{1,3}_+$, of
Example~\ref{ex:future-Mink-tau} is easily seen to be a sub-space-time of
Minkowski space-time, ${\mathbb R}^{1,3}$.

However, the future, $N_+$, of the lower hyperboloid, $M\subset {\mathbb R}^{1,3}$, of Example~\ref{ex:lowerhyp1} is not a sub-space-time of Minkowski space because 
the cosmological time function is not shifted.   These two subregions have the same Lorentzian metric and causal structures, and both satisfy (\ref{eq:causal-sub})
but their cosmological times and null distances are very different from one another. 
\end{ex}

\subsection{{\bf Regions in Minkowski Space-Time and other Asymptotically Flat Space-Times}}
\label{sect:space-times-Minkowski}

Recall Minkowski, ${\mathbb{R}}^{1,3}$ and Future Minkowski space-times, ${\mathbb{R}}^{1,3}_+$, as in Examples~\ref{ex:Mink-tau} and~\ref{ex:future-Mink-tau}.

In this section, we will find 
some null compactifiable sub-space-times of Minkowski space.

\begin{ex}\label{ex:Mink-strip}
Note that the time function in Minkowski space has a local cosmological time function as in Definition~\ref{defn:local-cosmo}.
Consider the strip,
\be\label{eq:Mink-strip}
N=N_{\tau_{min},\tau_{max}}
=\{(t,x):\, t\in(\tau_{min},\tau_{max}),\, x\in {\mathbb R}^3\}
\ee
with $g=g_{Mink}$ and its bounded and regular but not proper cosmological time, 
\be
\tau_g(x,t)=t-\tau_{min},
\ee
and recall that for any $q=(t_q,x_q)\in N$, its generator, 
\be
\gamma_q(t)=(t+\tau_{min},x_q)\subset N \quad \forall \, t\in (0, \tau_g(q)).
\ee
It is easy to see that $(N,g)$ is not null compactifiable because its diameter is not bounded:
\be\label{eq:null-dist-Mink}
\hat{d}_{g_{Mink}}(p,q)=\max\{|t_p-t_q|,|x_p-x_q|\}\to \infty
\ee
when $|x_p-x_q|\to \infty$. 
\end{ex}

\begin{prop}\label{prop:past-in-Mink}
Given any point in Minkowski space-time, $p\in {\mathbb{R}}^{1,3}_+$,
and consider a strip intersected with its timelike past,
\be
N_{\tau_{min},\tau_{max},p}=I^-(p)\cap N_{\tau_{min},\tau_{max}}
\subset {\mathbb{R}}^{1,3}_+, 
\ee
which can also be written as
\be
N_{\tau_{min},\tau_{max},p}=\{(t,x):\, \tau_{min}<t<\min\{t_p, \tau_{max}\},
\, |x-x_p|<t_p-t \},
\ee
then $(N_{\tau_{min},\tau_{max},p},g_{Mink})$ is a causally-null-compactifiable sub-space-time
of future Minkowski space.  
See Figure~\ref{fig:past-in-Mink}.
Note that by fixing 
$\tau_{min}<\tau_{max}$ and taking a sequence of $p_i=(t_i,x_0)$ with $t_i\to \infty$ we get an exhaustion of the cosmic strip by causally-null-compactifiable sub-space-times:
\be
N_{\tau_{min},\tau_{max}}=\bigcup_{i=1}^\infty 
N_{\tau_{min},\tau_{max}, p_i}
\ee
\end{prop}

\begin{figure}[h] 
   \centering
   \includegraphics[width=6.5in]{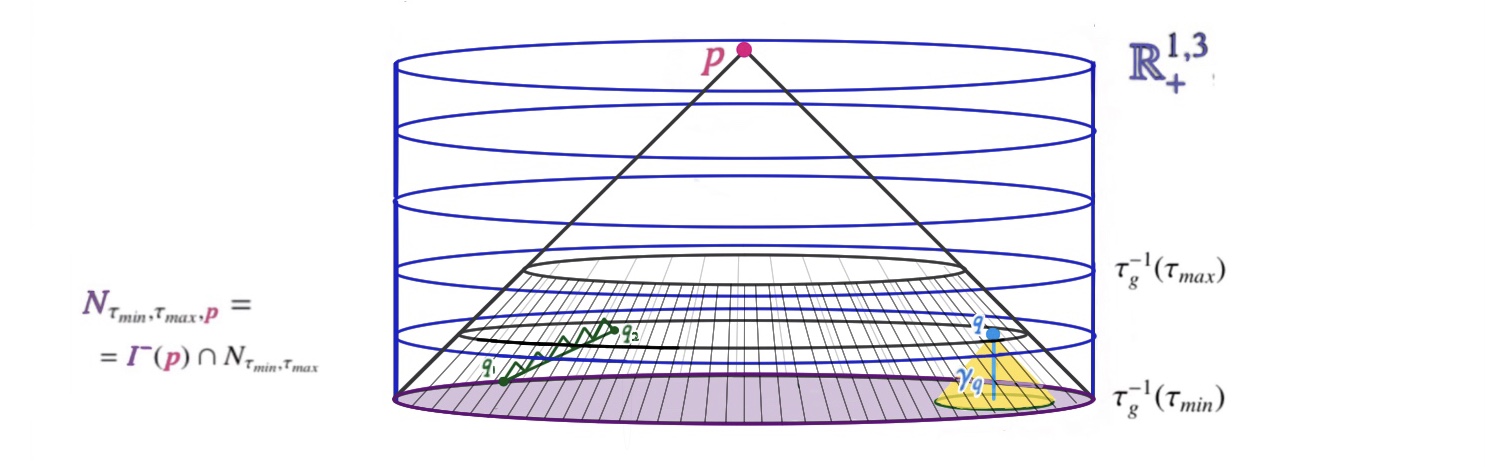} 
   \caption{Here we see $q, q_1,q_2 \in N_{\tau_{min},\tau_{max},p}\subset {\mathbb R}^{1,3}$ from the proof of Proposition~\ref{prop:past-in-Mink}.
   }\label{fig:past-in-Mink}
\end{figure}

\begin{proof}
Recall that the cosmological time
function of ${\mathbb{R}}^{1,3}_+$
with $g=g_{Mink}$
is 
\be
\tau_g(t,x)=\tau_{g_{Mink}}(t,x)=t.
\ee
Consider now the sub-domain, $N=N_{\tau_{min},\tau_{max},p}\subset
{\mathbb{R}}^{1,3}_+$ with the same metric tensor $g=g_{Mink}$.
We claim $(N, g)$ has
a bounded cosmological time: 
\be \label{eq:tauN(t,x)}
\tau_{N,g}(t,x)=
\tau_g(t,x)-\tau_{min}=t-\tau_{min}.
\ee
To see this, fix any $q=(t_q,x_q) \in N_{\tau_{min},\tau_{max},p}$.
Observe that any
timelike curve, $C$, in the past null cone $J^-(q)$
that runs from $q$ to the level set $t=\tau_{min}$
has Lorentzian length,
$L_{Mink}(C)\le t-\tau_{min}$.
The maximum is achieved by the
generator, which is the timelike geodesic, 
\be
\gamma_q:[0,t_q-\tau_{min}]\to N_{\tau_{min},\tau_{max},p},
\ee
defined by
\be\label{eq:gen-c}
\gamma_q(t)=(t+\tau_{min},x_q)\in 
J^-(q)\cap N_{\tau_{min},\tau_{max}}\subset I^-(p)\cap N_{\tau_{min},\tau_{max}}.
\ee
This cosmological time function is not proper, because its level sets are diffeomorphic to open balls in ${\mathbb R}^3$, however it is easily seen to be regular.

We next show that
the rectifiable lengths of the null distances are the same in $N_{\tau_{min},\tau_{max},p}$ and in ${\mathbb R}^{1,3}_+$.  We claim that, in fact, the null distances are the same due to the convexity of the region $N_{\tau_{min},\tau_{max},p}$ with respect to the
future Minkowski null distance in (\ref{eq:null-dist-Mink}).  Recall that Sormani-Vega proved in
\cite{SV-Null} that
\be
\hat{d}_{g_{Mink}}(q_1,q_2)=\max\{|t_1-t_1|,|x_1-x_1|\}
\ee
We can easily find piecewise null curves achieving this null distance from $q_1$ to $q_2$ 
that lie in $N_{\tau_{min},\tau_{max},p}$ 
by first taking a straight Minkowski geodesic
between them and then taking piecewise null curves lying above and below that straight geodesic that stay within the strip.   See Figure~\ref{fig:past-in-Mink}.  
Finally we see that the
sum in the definition of the null distance for 
$N_{\tau_{min},\tau_{max},p}$ and
 for ${\mathbb R}^{1,3}_+$ are the same
using (\ref{eq:tauN(t,x)}).

The diameter of ${N}_{\tau_{min},\tau_{max},p}$
is then easily computed to be
\be
\diam(N_{\tau_{min},\tau_{max},p})
\le |\tau_{max}-\tau_{min}| + 2 (t_p-\tau_{min})
\ee
so it is bounded with respect to the
Minkowski null distance, and thus its metric completion is compact.   

Finally note that any level set, $S=\tau_g^{-1}(s)\subset N=N_{\tau_{min},\tau_{max},p}$, is future causally complete  because
\be \label{eq:N-fcc}
\forall \, q\in J^+(S)\cap N, \textrm{ the set }
J^-(q)\cap S \subset I^-(p)\cap S \subset 
N=N_{\tau_{min},\tau_{max},p}
\ee
is just a closed Euclidean disk which is compact.
See Figure~\ref{fig:past-in-Mink}.  Thus by Galloway's theorem in
\cite{Galloway-causal}, we have causality encoded
as in (\ref{eq:causal}) and so the space is causally null compactifiable.
\end{proof}

In light of the above proposition we make the following conjecture:

\begin{conj}\label{conj:past-P}
Suppose we have a smooth space-time, $(\tilde{N},g,\tilde{\tau})$,
that has a locally cosmological time function, 
$\tilde{\tau}: \tilde{N} \to (-\infty,\infty)$, 
defining a null distance, $\hat{d}_{g,\tilde{\tau}}: \tilde{N}\times \tilde{N}\to [0,\infty)$.
Suppose that there is a point $p_0 \in \tilde{N}$
then the timelike past of $p_0$ intersected with a 
cosmic strip,
\be
(I^-(U)\cap \tilde{\tau}^{-1}(s,t),g),
\ee
is a causally-null-compactifiable sub-space-time of
$(\tilde{N},g,\tilde{\tau})$.   Perhaps an additional hypothesis might be required like, for example, global hyperbolicity.
\end{conj}

\begin{rmrk}\label{rmrk:no-claim-convex}
Note that in the above conjecture and in Definition~\ref{defn:sub-space-time}, we are not assuming or concluding any kind of convexity.  Null distances between points in the larger space are not necessarily the same as null distances between points in the sub-space-time.  This is only true on convex sub-domains of the sub-space-time and thus only the rectifiable null lengths agree.   This is
also true of Riemannian distances in sequences of balls that exhaust complete Riemannian manifolds.
\end{rmrk}

Next let us consider a nonconvex region in a Minkowski strip which is also defined as a past but this time the past of a large ring rather than a single point.

\begin{figure}[h] 
   \centering
   \includegraphics[width=5in]{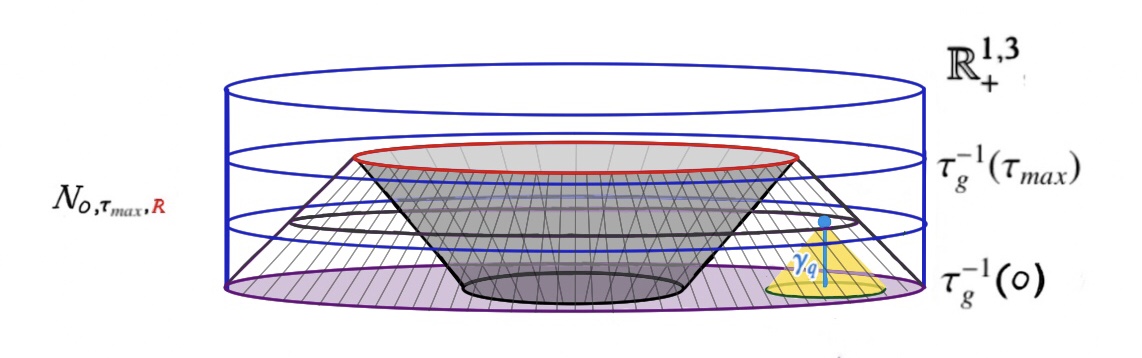} 
   \caption{
   Here we see $N_{0,\tau_{max},R}\subset {\mathbb R}^{1,3}$ of
   Example~\ref{ex:Mink-ext} and a generator,
   $\gamma_q$, of its cosmological time.
   }\label{fig:Mink-ext}
\end{figure}

\begin{ex} \label{ex:Mink-ext}
Given $0<\tau_{max}<3\tau_{max}<R$, let us consider
\be
N=N_{0,\tau_{max},R}
=I^-(\{(3\tau_{max},y):\, |y|=R\})\cap N_{0,\tau_{max}}\subset {\mathbb R}^{1,3}.
\ee
See Figure~\ref{fig:Mink-ext}
Note that this region can also be described as
\be
N_{0,\tau_{max},R}=\{(t,x):\, t\in(0, \tau_{max}), \, 
|\,|x|-R\, |< 3\tau_{max}-t \}\subset {\mathbb R}^{1,3},
\ee
which has a hole in the center as seen in Figure

Exactly as in (\ref{eq:gen-c}) we see that the
generators of the cosmological time agree and in this case the cosmological time is not even shifted.
Exactly as in (\ref{eq:N-fcc}) we see that the levels of $\tau_{N,g}$ are future causally complete and thus we have causality encoded.   However our region, $N_{0,\tau_{max},R}$, is not convex with respect to the future Minkowski null distance and so the null distance
for $N_{0,\tau_{max},R}$ does not agree with the future Minkowski null distance.   
Nevertheless, for any convex set inside $N_{0,\tau_{max},R}$, the null distances do agree, and so it is
easy to see that the null distance on $N_{0,\tau_{max},R}$ is the induced rectifiable length metric.   Furthermore, the null distance is bounded, and so $N_{0,\tau_{max},R}$ is a causally-null-compactifiable sub-space-time of Minkowski space
endowed with the time function, $\tau(t,x)=t$.
\end{ex}

\begin{rmrk} \label{rmrk:Mink-est}
Observe that in Example~\ref{ex:Mink-ext}, for any $[s,t]\subset (0,\tau_{max}$, we have a null compactifiable
sub-space-time,
\be
N_{s,t,R}=\tau^{-1}(s,t)\cap N_{0,\tau_max,R},
\ee
and, in fact this sub-space-time is the time-like past of an open set intersected with a cosmic strip
\be
N_{s,t,R}=I^-(U)\cap \tilde{\tau}^{-1}(s,t)
\ee
where $U=N_{t,\tau_{max},R}$.
\end{rmrk}

In light of this, we make the following
conjecture:

\begin{conj}\label{conj:past-U}
Suppose we have a smooth space-time, $(\tilde{N},g,\tilde{\tau})$,
that has a locally cosmological time function, 
$\tilde{\tau}: \tilde{N} \to (-\infty,\infty)$, 
defining a null distance, $\hat{d}_{g,\tilde{\tau}}: \tilde{N}\times \tilde{N}\to [0,\infty)$.
Suppose that there is an open region $U \subset \tilde{N}$
with a bounded $\hat{d}_{g,\tilde{\tau}}$-diameter,
then the past of $U$ intersected with a cosmic strip,
\be
\left(J^-(U)\cap \tilde{\tau}^{-1}(s,t),g\right),
\ee
is a causally-null-compactifiable sub-space-time of
$(\tilde{N},g,\tilde{\tau})$.
\end{conj}

\begin{rmrk}\label{rmrk:past-U}
It is possible that the above conjecture may require additional hypotheses.   A counter example without additional hypotheses would be interesting. 
\end{rmrk}

\subsection{{\bf Regions in Exterior Schwarzschild Space-Time}}
\label{sect:space-times-Schwarzschild}

In this section we find sub-space-times of Schwarzschild space-time.

\begin{ex}\label{ex:ext-Sch}
The Schwarzschild space-time, $N_{Sch,m}$, of mass $m>0$ is globally hyperbolic
for $t\in (-\infty,\infty)$, $r> 0$ :  
\be
g_{Sch,m}=-\left(\frac{r^2-2mr}{r^2}\right) dt^2 
+ \left(\frac{r^2}{r^2-2mr} \right) dr^2 + r^2 g_{{\mathbb{S}}^2}.
\ee
This metric appears to be singular at the event horizon, $\Sigma=r^{-1}(2m)$, however one can extend smoothly as a space-time across $\Sigma$.   In fact, $\Sigma$ is an event horizon of a black hole: 
\be\label{eq:event-horizon}
\forall \, p\in \Sigma, \,\, J^+(p)\cap r^{-1}(2m, \infty)=\emptyset.
\ee
The exterior region of the Schwarzschild black hole is the region,
\be\label{eq:ext-Sch}
N'_{Sch,m}= t^{-1}(-\infty,\infty) \cap r^{-1}(2m,\infty)\subset N_{Sch,m}.
\ee
The $t=const$ level sets are
isometric to the Riemannian Schwarzschild space, $(M_{Sch},h_{Sch})$, have second fundamental form, $k=0$, and are a foliation of Schwarzschild space by maximal spacelike hypersurfaces (such that $\tr k=0$).   

Exterior Riemannian Schwarzschild is
the surface,
\be\label{eq:ext-R-Sch}
M'_{Sch,m}=t^{-1}(0)\cap N'_{Sch,m}=t^{-1}(0)\cap r^{-1}(2m,\infty)\subset N_{Sch,m}.
\ee
See for example, O'Neill's text \cite{O'Neill-text}.
Finally we define the future exterior Schwarzschild space-time:
\be\label{eq:ext-Sch-+}
N'_{Sch,m,+}= J^+(M'_{Sch,m})\cap N'_{Sch,m}= t^{-1}(0,\infty) \cap r^{-1}(2m,\infty)\subset N_{Sch,m}.
\ee
By (\ref{eq:event-horizon}), we have
\be \label{eq:ext-Sch-past}
\forall \, p \in N'_{Sch,m,+},\,\, J^-(p) \cap M_{Sch,m}\subset M'_{Sch,m}
\ee
as depicted in Figure~\ref{fig:ext-Sch}.
In particular, future exterior Schwarzschild space-time, $N'_{Sch,m,+}$,
satisfies the hypothesis of Galloway's causality theorem \cite{Galloway-causal}.
So it has a regular cosmological time function, $\tau_g:N'_{Sch,m,+}\to (0,\infty)$,
and the null distance, 
\be
\hat{d}_{Sch',m,+}=\hat{d}_{N'_{Sch,m,+},g_{Sch,m}},
\ee
encodes causality as in (\ref{eq:causal}).

The generators, $\gamma_p:(0,\tau(p)]\to N'_{Sch,m,+}$, fall inward,
\be\label{eq:falling}
\tfrac{d}{d\tau}r(\gamma_p(\tau))\le 0 \textrm{ and } 
\tfrac{d}{d\tau}t(\gamma_p(\tau))\ge 0.
\ee
with
\be
-1=g(\gamma_p'(\tau),\gamma_p'(\tau))=-\left(\tfrac{r^2-2mr}{r^2}\right) \left(\tfrac{d(t\circ \gamma_p)}{d\tau}\right)^2
+ \left(\tfrac{r^2}{r^2-2mr} \right) \left(\tfrac{d(r\circ \gamma_p)}{d\tau}\right)^2
\ee
so that
\be
\left(\tfrac{d(t\circ \gamma_p)}{d\tau}\right)^2>
\left(\tfrac{r^2-2mr}{r^2}\right) \left(\tfrac{d(t\circ \gamma_p)}{d\tau}\right)^2=
1+
\left(\tfrac{r^2}{r^2-2mr} \right) \left(\tfrac{d(r\circ \gamma_p)}{d\tau}\right)^2
> 1+ \left(\tfrac{d(r\circ \gamma_p)}{d\tau}\right)^2>1.
\ee
Thus $d(t\circ \gamma_p)/d\tau >1$ and 
$d(t\circ \gamma_p)/d\tau >d(r\circ \gamma_p)/d\tau$.
Travelling upwards from the $\tau=t=0$ level set we see that,
\be\label{eq:Sch-tau<t}
\tau(p)<t(p) \qquad \forall \, p\in N'_{Sch,m,+} \textrm{ and }\tau(p)=0\quad\forall \, p\in M'_{Sch,m}.
\ee 
\end{ex}

\begin{figure}[h] 
   \centering
   \includegraphics[width=5in]{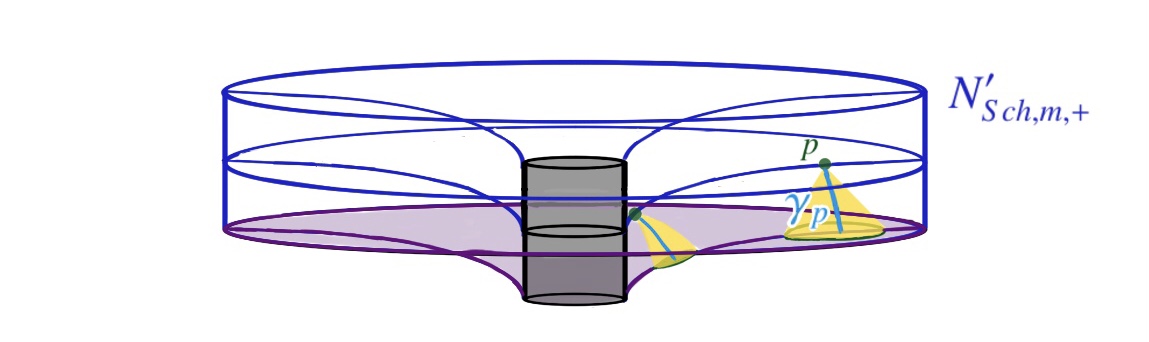} 
   \caption{Here we see the future
   exterior Schwarzschild space-time, $N'_{Sch,m,+}$, as in Example~\ref{ex:ext-Sch}.
   }\label{fig:ext-Sch}
\end{figure}

\begin{ex} \label{ex:cpct-ext-Sch}
The future exterior
Schwarzschild space-time,
\be
N'_{Sch,m,+}=r^{-1}(2m,R)\cap t^{-1}(0,\infty)
\subset N_{Sch,m},
\ee
can be exhausted 
by domains,
 \be\label{eq:WR}
\Omega(R)=J^-(U_R) \cap N'_{Sch,m,+} \textrm{ s.t. } \bigcup_{R>2m} \Omega(R)=N'_{Sch,m,+},
\ee
where
 \be \label{eq:UR}
U_R=r^{-1}(2m,R) \subset N'_{Sch,m,+}.
\ee
Observe that for all $p\in \Omega(R)$, the generators, $\gamma_p$, of 
$\tau_g:N'_{Sch,m,+}\to (0,\infty)$ have images which lie in $J^-(p)\subset \Omega(R)$
and are not past extensible within $\Omega(R)$,
so the cosmological time function, 
$\tau_{\Omega(R),g}: \Omega(R) \to (0,\infty)$ is regular and satisfies 
\be
\tau_{\Omega(R),g}(p)=\tau_g(p) \qquad \forall \, p\in \Omega(R).
\ee
Thus the
$\Omega(R)$ are sub-space-times of future exterior Schwarzschild space-time as in Definition~\ref{defn:sub-space-time}.
\end{ex}

\begin{prop}\label{prop:cpct-ext-Sch}
The cosmic strips in future exterior Schwarzchild space-time, 
\be
\Omega_{0,\tau_{max}}(R)=\Omega(R) \cap \tau_g^{-1}(0,\tau_{max})
\ee
where
$\Omega(R)$ is as in Example~\ref{ex:cpct-ext-Sch}
depicted in Figure~\ref{fig:cpct-ext-Sch}
are causally-null-compactifiable
space-times. 
\end{prop}

\begin{proof}
The fact that the timed-metric-space in Example~\ref{ex:cpct-ext-Sch} encodes
causality as in
(\ref{eq:causal}) follows from (\ref{eq:ext-Sch-past}),
the properties of $\tau_g$ we showed in Example~\ref{ex:cpct-ext-Sch}
and Galloway's causality theorem in \cite{Galloway-causal}.

To prove compactness of the metric completion of $\Omega_{0,t}(R)$ we proceed as in the
proof of Theorem~\ref{thm:proper} now studying sequences of $p_i\in \Omega_{0,t}(R)$.
  Again we can choose a subsequence
so that $\tau_i=\tau_g(p_i)$ converges to some $\tau_\infty$.  

Next
we consider the
generators, $\gamma_{p_i}$.
By (\ref{eq:ext-Sch-past}) and (\ref{eq:falling}), we see that
\be
\gamma_{p_i}[0,\tau_g(p_i)]\in N'_{Sch,m}\cap r^{-1}[2m,R_{max}] \cap \tau^{-1}[0,t]
\ee
so a subsequence converges to a limiting geodesic, $\gamma_\infty$,
just as in (\ref{eq:gen-2}) in the proof of Theorem~\ref{thm:proper}.

If $\tau_\infty>0$,
the limiting geodesic is still a generator of $\tau_g$ in exterior Schwarzschild, $N'_{Sch,m,+}$, so it is timelike and thus cannot run along $\Sigma$ nor along the causal outer boundary of 
$\Omega_{0,t}(R)$.
Thus
$\gamma_{\infty}(0,\tau_\infty)\subset \Omega_{0,t}(R)$, so we can continue
exactly as in Theorem~\ref{thm:proper} to produce a limit point, $p_\infty=\bar{\gamma}_\infty(\tau_\infty)
$ inside
the metric completion of the region.

If $\tau_\infty=0$, then inside the closure of the region in Schwarzschild, the sequence $p_i$ converges to $\gamma_\infty(0)$ in Riemannian Schwarzschild.   Since $\tau_g$ is smooth near this initial data, we can extend it to a neighborhood of $\gamma_\infty(0)$ and thus extend the null distance to the same neighborhood.   Since the null distance has the same 
topology as the manifold, we see that $p_i$ converges with respect to this extended null distance to $p_\infty$.  Thus $p_i$ is a Cauchy sequence with respect to the null distance in $\Omega_{0,t}(R)$
and so it has a limit in the metric completion.
\end{proof}

\begin{figure}[h] 
   \centering
   \includegraphics[width=5in]{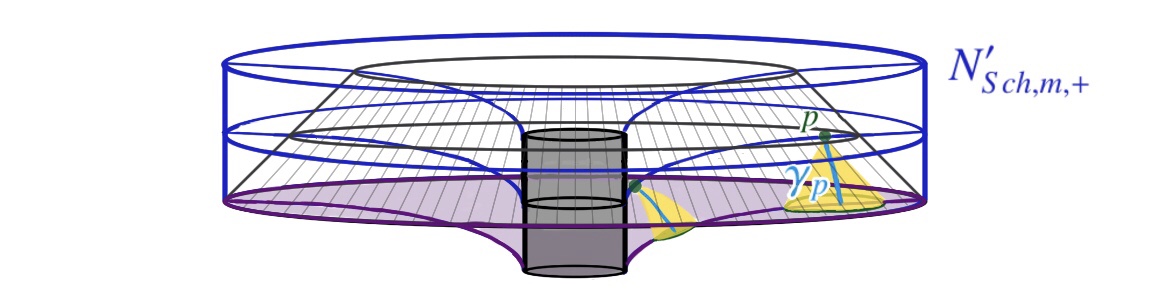} 
   \caption{A causally-null-compactifiable sub-space-time of Future Exterior Schwarzschild space as in Example~\ref{ex:cpct-ext-Sch}.
   }\label{fig:cpct-ext-Sch}
\end{figure}

\begin{rmrk} \label{rmrk:Kerr-exterior}
Note that Kerr space-time, $(N_{Kerr},g_{Kerr})$, also has a foliation by maximal spacelike hypersurfaces, one of which may be used as the initial data set to define future Kerr space-time, $N_{Kerr,+}$.    We define the exterior region, $N'_{Kerr,+}$, to be the region outside the event horizon.   The geodesics and causal structure of Kerr space-time is well known. 
We believe one can construct causally-null-compactifiable
sub-space-times of exterior regions in Kerr space-times using methods similar to those described above for Schwarzschild space-times. 
\end{rmrk}

Naturally the goal is to test ideas so that one can answer the following more challenging question:

\begin{rmrk}\label{rmrk:exterior-region}
For a general asymptotically flat space-time, $N$, possibly with multiple black holes, we could define the the exterior region, $N'$, to be the region outside the outermost event horizons or use some other definition.  The key concern is ensuring the exterior region can be exhausted by causally null compactifiable sub-space-times.
\end{rmrk}

\subsection{{\bf Exhaustions by Causally Null Compactifiable Sub-space-times}}\label{sect:space-times-ex}

In this section we define exhaustions by causally null compactifiable sub-space-times as in Definition~\ref{defn:sub-space-time}.  These exhaustions can be applied to define convergence of space-times which are not null compactifiable themselves.  Such exhaustions would be needed for example to define the convergence of asymptotically flat space-times.

\begin{defn}\label{defn:exhaustion}
We say that a smooth space-time, $(\tilde{N},g,\tilde{\tau})$,
that has a global time function, 
$\tilde{\tau}: \tilde{N} \to (-\infty,\infty)$, 
defining a null distance, $\hat{d}_{g,\tilde{\tau}}: \tilde{N}\times \tilde{N}\to [0,\infty)$, is {\bf exhausted by (causally) null compactifiable sub-space-times},
\be
\Omega(R_1)\subset \Omega(R_2)\subset ... \subset N
\textrm{ such that } \bigcup_{i=1}^\infty \Omega(R_i)=
\tilde{N}
\ee
if each $(\Omega(R_i),g_i)$ is a (causally) null compactifiable sub-space-time of $(\tilde{N},g, \tilde{\tau})$.
\end{defn}

\begin{rmrk}\label{rmrk:sub-sub}
Note that in Definition~\ref{defn:exhaustion} each $(\Omega(R_i),g)$ endowed with its own cosmological time function is a
(causally) null compactifiable sub-space-time of the
next $(\Omega(R_{i+1}),g)$ endowed with its own cosmological time function.   These time functions at each step may shift significantly so it is better to use the global time function, 
$\tilde{\tau}: \tilde{N} \to (-\infty,\infty)$ to unify
the sequence keeping in mind that all the null distances
define the same rectifiable lengths.
\end{rmrk}

\begin{figure}[t] 
   \centering
   \includegraphics[width=5in]{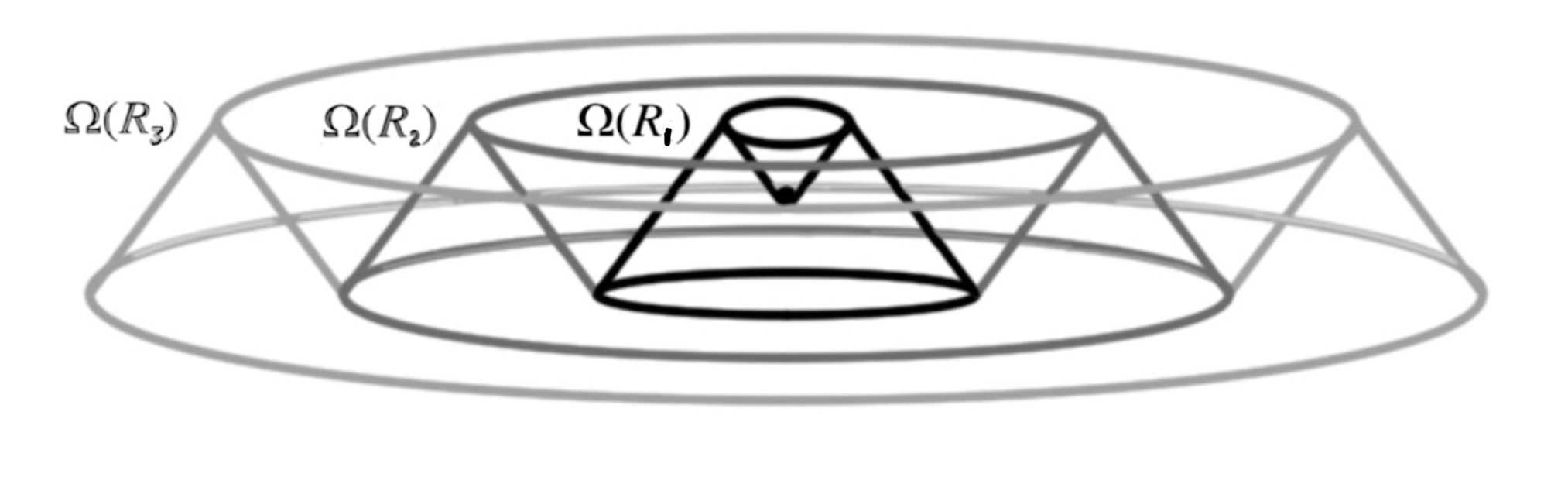} 
   \caption{The first three convex causally null compactifiable sub-space-times, 
   $\Omega(R_3)\supset \Omega(R_2)\supset \Omega(R_1)$,
   in the pointed exhaustion  of a cosmic strip in future Minkowski space-time as a special case of Remark~\ref{rmrk:pt-exhaustion}.
   }\label{fig:pt-exhaustion}
\end{figure}

\begin{rmrk}\label{rmrk:exhaust}
Note that in Definition~\ref{defn:exhaustion} we are not assuming that $\tilde{\tau}$ is a regular cosmological time or a locally cosmological time function.  Nor are we assuming that
its associated null distance, $\hat{d}_{g,\tilde{\tau}}$,
is definite nor do we assume it encodes causality. We only
require these additional properties within the regions, $\Omega(R_i)$. However, as a consequence of the fact that the union of the $\Omega(R_i)$ is the entire space-time, $N$, we see that the 
null distance, $\hat{d}_{g,\tilde{\tau}}$,
is definite afterwards.  Furthermore, it encodes causality when the exhaustion is causal.  Is the global time function, $\tilde{\tau}$, a regular locally cosmological time function?
This is easy to see if it is proper, but not immediately obvious to us otherwise.
\end{rmrk}

\begin{rmrk} \label{rmrk:pt-exhaustion}
One should be able to prove a general theorem
which allows one to exhaust an asymptotically flat space-time with causally-null-compactifiable sub-space-times as in Definition~\ref{defn:exhaustion} using methods similar to those in Sections~\ref{sect:space-times-Minkowski}
and~\ref{sect:space-times-Schwarzschild} in a canonical way depending only on the choice of a basepoint.  Perhaps we might 
fix $p\in N$ and fix $\tau_{min}<\tau_{max}$
and take 
\be
\Omega(R_1)=I^-\left(J^+(p)\cap \tilde\tau_g^{-1}(\tau_{min},\tau_{max})\right)
\cap \tilde\tau_g^{-1}(\tau_{min},\tau_{max})
\ee
and then iteratively define 
\be
\Omega(R_{i+1})=I^-\left(J^+(\Omega(R_i))\cap \tilde\tau_g^{-1}(\tau_{min},\tau_{max})\right)
\cap \tilde\tau_g^{-1}(\tau_{min},\tau_{max})
\ee
as in Figure \ref{fig:pt-exhaustion} and
take
\be\label{eq:R_i-pt}
R_{i}=2i(\tau_{max}-\tau_{min}) 
\ee
so that
\be
\forall q\in \Omega(R_{i}) \quad
d(q,p)< R_i.
\ee
We conjecture that $(\Omega(R_i),g)$
are causally-null-compactifiable
and can be used to define an
exhaustion of cosmic strip, 
\be\label{eq:exhaustion-Ri}
\bigcup_{i=1}^\infty \Omega(R_i) = \tilde\tau_g^{-1}(\tau_{min},\tau_{max})\subset N.
\ee
In Remark~\ref{rmrk:exhausted-conv}
we will explain how such a canonical choice of exhaustion could help define space-time convergence for asymptotically flat space-times.
\end{rmrk}

\begin{rmrk}\label{rmrk:CMC}
Another approach to constructing exhaustions
of future developments of time symmetric initial data set, $(M,h,0)$, would be to use the canonical Constant-Mean-Curvature (CMC) foliations studied by Huisken-Yau, Qing-Tian, Ma,
Chodosh, Eichmair, Koerber, Metzger, Huang, Alaee and Lesaurd   
in 
\cite{Huisken-Yau-CoM}
\cite{Qing-Tian-Foliations}
\cite{Metzger-Foliations}
\cite{Eichmair-Metzger-Unique}
\cite{Chodosh-Eichmair-Global}
\cite{Eichmair-Koerber}
\cite{Nerz-Foliations}
\cite{Huang-Foliations}
\cite{Ma-Foliations}
\cite{Alaee-Lesourd-Yau-stable}.
Cederbaum and Sakovich have extended these ideas to develop a canonical space-time-constant-mean-curvature (STCMC) foliation
for more general initial data sets in \cite{Cederbaum-Sakovich}.  This defines the canonical STCMC exhaustion of an initial data set, 
\be
\bigcup_{R>0} W(R) \subset M, 
\ee
where $\partial W(R)=\Sigma_R$
is the canonical space-time-constant-mean-curvature (STCMC) foliation.   Here $R$ is the area radius of the surface.
\end{rmrk}

\begin{figure}[h] 
   \centering
   \includegraphics[width=4in]{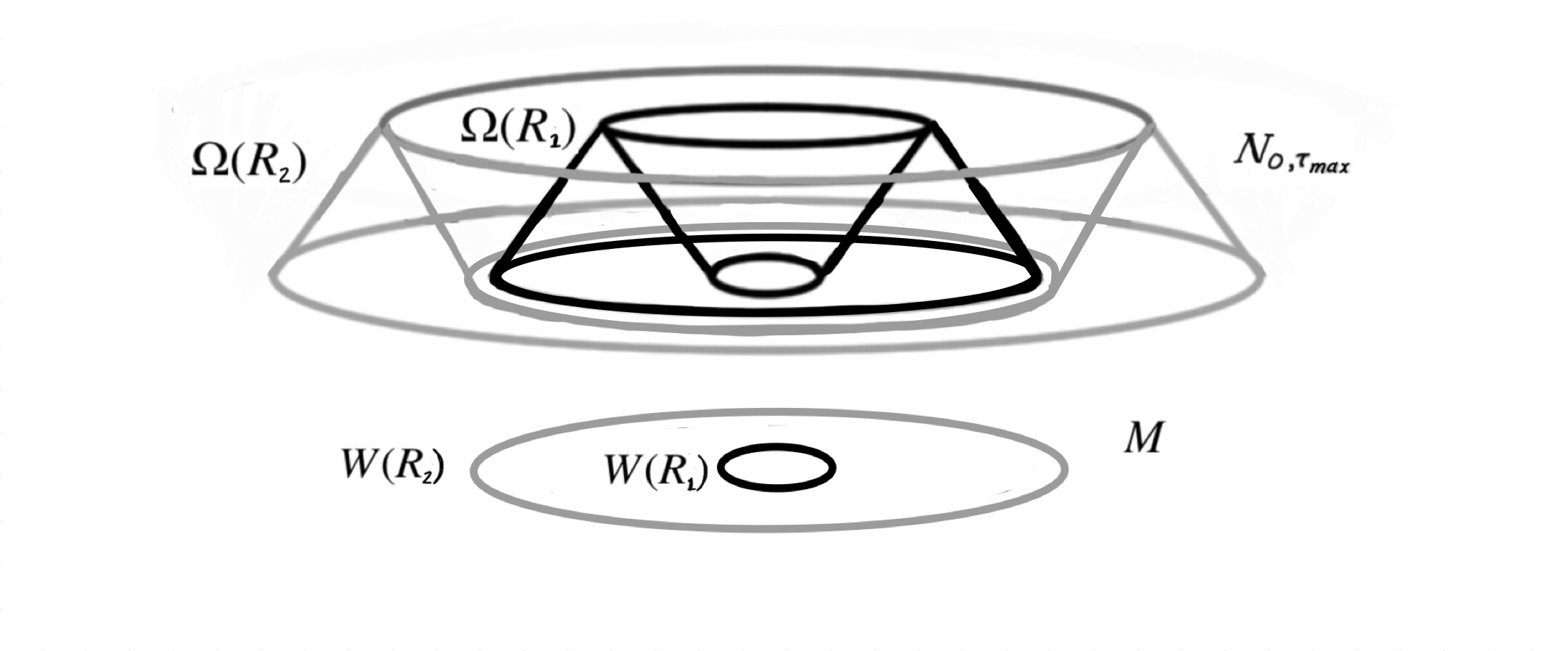}
   \caption{Two convex causally null compactifiable sub-space-times, $\Omega(R_2)\supset \Omega(R_1)$, in an exhaustion  of a cosmic strip in future Minkowski space-time defined using an exhaustion by regions, $W(R_2)\supset W(R_1)$, bounded by CMC surfaces in Euclidean space (viewed as initial data for future Minkowski space) as a special case of Remark~\ref{rmrk:CS-exhaustion}.
   }\label{fig:CS-exhaustion}
\end{figure}

\begin{rmrk} \label{rmrk:CS-exhaustion}
On an asymptotically flat future-developed space-time, $(N,g)$, starting from an asymptotically flat initial data set, $(M,h,K)$, there is 
a natural exhaustion by
$\Omega(R)\subset N$ of $N$ that we define as follows.   Start with a canonical
Cederbaum-Sakovich STCMC exhaustion of the initial data set 
\be
\bigcup_{i=1}^\infty W(R_i) \subset M 
\ee 
(as explained in Remark \ref{rmrk:CMC}).
Let
\be
U(R_i)=
J^+(W(R_i))\cap \tau^{-1}(0,\tau_{max})
\subset N
\ee
and take
\be
\Omega(R_i)=I^-(U(R_i)) \subset N
\ee
as in Figure~\ref{fig:CS-exhaustion}.
We conjecture that $(\Omega(R),g)$
is causally-null-compactifiable
and can be used to define an
exhaustion of cosmic strip, $\tau^{-1}(0,\tau_{max})$
as in (\ref{eq:exhaustion-Ri}).  Further work in this direction is in progress by Sakovich and others.  
\end{rmrk}

\subsection{{\bf Metric Measure Space-Times}}
\label{sect:space-times-measure}
\footnote{This section may be skipped by those only interested in $GH$ and $\mathcal{F}$ notions of distance.}

The following definition could possibly be viewed as something vague and somewhat conjectural.

\begin{defn}\label{defn:measure-space-time}
A timed metric measure space could possibly be defined to be a timed-metric-space endowed with a measure, $\mu$,
that is defined in a canonical way.   Sometimes $\mu$ might be the volume measure, $\mu_g$, defined by the Lorentzian metric tensor.   
Sometimes the measure might be required to be a doubling measure with respect to the null distance.
In settings where one needs a probability measure, the measure may always be renormalized. 
\end{defn}

\begin{rmrk}\label{rmrk:measure-space-time}
Prior work developing the notion of metric measure space-times using optimal transport and other measure theoretic approaches appear in work of     
Beran, 
Braun, 
Brenier,
Calisti,
Cavalletti, 
Eckstein,
Frisch,
Gigli,
Manini,
Matarrese, 
McCann,
Miller,
Minguzzi,
Mohayaee,
Mondino,
Ohanya,
Ohta,
Rott,
Sobolevski,
and Suhr
\cite{BBCGMORS}
\cite{Braun-entropy}
\cite{Braun-23} 
\cite{Braun-McCann}
\cite{Braun-Ohta} 
\cite{Brenier-GR}
\cite{Cavalletti-Mondino}
\cite{Cavalletti-Manini-Mondino-Optimal}
\cite{Eckstein-Miller-causality}
\cite{FMMS-Nature}
\cite{McCann-Synthetic}
\cite{Minguzzi-Further}
\cite{Minguzzi-Suhr-24}
\cite{Mondino-Suhr}
\cite{Mueller-GH}
building upon work by Lott-Villani \cite{Lott-Villani-09} and by Sturm \cite{Sturm-2006-I}.
These papers use the relationship between Ricci curvature and optimal transport to 
explore various aspects of the Einstein vacuum equations on these spaces.
Mondino and Perales are
working to relate the work of Mondino-Suhr
with our notion of a timed-metric-space endowed with a measure so that our theories may be used
together.   
\end{rmrk}

\begin{rmrk}\label{rmrk:space-time-measure-disclaimer}
In our definitions of intrinsic distances between causally-null-compactifiable space-times below, we will not emphasize where the measure comes from, but rather just assume that there is a measure on the associated timed-metric-space.  In this way any notion of metric measure space-time developed in the future may be applied in combination with one of our notions of intrinsic distance between causally-null-compactifiable space-times.
\end{rmrk}

\subsection{{\bf Integral Current Space-Times}}
\label{sect:space-times-current}
\footnote{This section may be skipped by those only interested in $GH$ and $mm$ notions of distance.}

Recall the definition of an integral current space of Sormani-Wenger \cite{SW-JDG} reviewed vaguely in Section~\ref{sect:back-F}.

\begin{defn}\label{defn:integral-current-space-time}
A timed integral current space (or integral current space-time) can be defined to be a timed-metric-space endowed with an integral current structure, $T$, defined by an atlas of
$\hat{d}_g$-biLipschitz charts
so that the resulting 
settled completion, $(N'=\set(T), \hat{d}_g,T)$ is an integral current space.  In particular, the manifold $N$ and its boundary, $\partial N\subset \bar{N}$ are $\hat{d}_g$-rectifiable sets.  Sub-space-times only have integral current structures if their regions and the boundaries of their regions are also $\hat{d}_g$ rectifiable sets.
\end{defn}

\begin{rmrk}\label{rmrk:int-curr-space-time}
Future work by Sakovich-Sormani will explore the notion of integral current space-times which are timed-metric-spaces endowed with an integral current structure \cite{SakSor-SIF}.   Preliminary work in this direction has begun in work of
Allen, Burtscher and Garc/'ia-Heveling in
\cite{Allen-Burtscher-22}
\cite{Burtscher-Garcia-Heveling-Global}
 \cite{Burtscher-Garcia-Heveling-Time}.
\end{rmrk}

\begin{rmrk}\label{rmrk:space-time-current-disclaimer}
In our definitions of intrinsic distances between causally-null-compactifiable space-times below, we will not emphasize where the integral current structure comes from, but rather just assume that there is an integral current structure on the associated timed-metric-space.  In this way any notion of integral current space-time developed in the future may be applied in combination with one of our notions of intrinsic distance between space-times.
\end{rmrk}

\subsection{{\bf Other Weak Notions of Space-Times}}
\label{sect:space-times-other}

\begin{rmrk}\label{rmrk:C0-spaces}
Perhaps the simplest weak notion of a space-time is a manifold endowed with a Lorentzian metric tensor that is only continuous
For a discussion of such space-times see the
work of Klainerman-Rodnianski-Szeftel 
and Chrusciel-Grant
in \cite{KlRodSz} 
\cite{Chrusciel-Grant-continuous} 
and the survey by Ling \cite{Ling-Aspects}. 
See also Steinbauer's paper
\cite{Steinbauer-singularity-survey}.
It would be interesting to explore the
cosmological time function and null distance
on such space-times. 
\end{rmrk}

\begin{rmrk}\label{rmrk:Lor-length}
There is a weak notion of space-time called a ``Lorentzian length space" which is not a metric space
but has a causal structure.   Early work of Busemann has causal sets  \cite{Busemann-Timelike}.
Alexander-Bishop, Harris, Kunzinger-Saemann study causal sets with length structures
and Lorentzian distances that satisfy the reverse triangle inequality
in \cite{Harris-triangle}\cite{Alexander-Bishop-Lorentz}\cite{Kunzinger-Saemann}. See also work of Alexander, Beran, Graf, Kunzinger, Naper, Rott, S\"amann, Ak\'e Hau, Cabrera Pacheco, Solis, Ebrahimi, Vatandoost, Pourkhandani, Barrera, Montes de Oca in 
\cite{AGKS}\cite{Hau-Cabrera-Pacheco-Solis}\cite{Neda-Mehdi-Rahimeh}\cite{BMdOS}\cite{BNR-TAMS}. 
Kunzinger-Steinbauer and Burtscher-Garc\'ia-Heveling have studied the null distance
on these spaces in
\cite{Kunzinger-Steinbauer-22}
\cite{Burtscher-Garcia-Heveling-Time}\cite{BMdOS}.  
\end{rmrk}

\begin{rmrk}\label{rmrk:not-general}
Our notion of a timed-metric-space, $(X,d,\tau)$, is a metric-space $(X,d)$ endowed with a Lipschitz one time function $\tau:X\to [\tau_{min},\tau_{max}]$
whose causal structure is defined by 
\be
p\in J^+(q) \iff \tau(p)-\tau(q)=d(p,q).
\ee
It is unclear when a timed-metric-space corresponds to a Lorentzian length space or a metric space-time.   This would be interesting to explore especially for $GH$ limits because $GH$ based notions of convergence preserve lengths. 
\end{rmrk}

Note that our notion of a timed-metric-space might be discrete or even finite, so only special cases would ever have any sort of length structure. Including discrete spaces in our definition allows us to use it to study numerical solutions and their convergence
 to smooth solutions in General Relativity. 

 \begin{rmrk}\label{rmrk:discrete}
Discrete spaces are also included in the notion of ``Lorentzian metric spaces" of Minguzzi-Suhr \cite{Minguzzi-Suhr-24}. 
It should be noted that their ``Lorentzian metric spaces" are not metric spaces but causal sets that are topological spaces with a ``Lorentzian distance function" that satisfies some natural properties similar to the Lorentzian distance on a Lorentzian manifold.  It would be interesting to see when one can define a canonical cosmological time function on such a causal set and 
 define a null distance.
 \end{rmrk}
 
\begin{rmrk}\label{rmrk:Dowker}
 There are also the 
 discrete causal sets used as a ``deep structure for space-time" in the work of Dowker  \cite{Dowker} used in the field of Causal Set Cosmology.  It would be interesting to explore when such causal sets have canonical cosmological time functions and a null distances.  
 \end{rmrk}

 \begin{rmrk}\label{rmrk:Mueller-POM}
There are also the partially ordered measured spaces
 studied by M\"uller in \cite{Mueller-Maximality}
 and other papers.  
 \end{rmrk}

Recall also Remark~\ref{rmrk:measure-space-time} on metric measure notions of space-times.

\section{{\bf{Defining Various Notions of Intrinsic Distances between Space-Times}}}
\label{sect:defns}

In this section we will define a variety of intrinsic distances between pairs of causally-null-compactifiable space-times,
\be
d_{S-dist}\Big((N_1,g_1),(N_2,g_2)\Big)
=d_{dist}
\Big((\bar{N}_1,d_{g_1},\tau_{g_1}),(\bar{N}_2,d_{g_2},
\tau_{g_2}))\Big),
\ee
as in Definition~\ref{defn:sd}.
Note that any notion of distance between smooth space-times can be applied to define convergence of smooth space-times:
\be \label{eq:Sdistto}
(N_j,g_j)\Sdistto (N_\infty,g_\infty)
\iff d_{S-dist}( (N_j,g_j), (N_\infty,g_\infty))\to 0.
\ee
We wish to have a notion of convergence to limit spaces that
may no longer be smooth space-times, so we define convergence to timed-metric-spaces (possibly endowed with a measure or integral current structure) by defining distances, $d_{dist}$, between timed-metric-spaces, and setting,
\be \label{eq:distto}
(\bar{N}_j,d_j,\tau_j)\distto (\bar{N}_\infty,d_\infty,\tau_\infty)
\iff d_{dist}\Big( (N_j,d_{j},\tau_{j}),(\bar{N}_\infty,d_\infty,\tau_\infty) \Big)\to 0.
\ee
In this section, we consider any pair of compact timed-metric-spaces, $(\bar{N}_j,d_j,\tau_j)$, as in Definition~\ref{defn:cncst}, and we
introduce a variety of choices of intrinsic distances,
$d_{dist}$, between them.
   
Naturally we may define $d_{dist}$ building upon the Gromov-Hausdorff distance ($GH$) and Sturm-Wasserstein metric measure distance ($mm$) between metric measure spaces, Sormani-Wenger intrinsic flat distance ($\mathcal{F}$) and volume preserving intrinsic flat distance ($\mathcal{VF}$) between integral current spaces.   To keep things simple, we will define everything using only the $GH$ distance, with remarks discussing how each can be adapted for $mm$, $\mathcal{F}$, and $\mathcal{VF}$. 
Many of our distances will involve the following quantities:
\be\label{eq:tau-min}
\tau_{min}=\min\left\{\min_{p\in N_1}\tau_1(p), \min_{p\in N_2}\tau_2(p)\right\}
\ee
and
\be \label{eq:tau-max}
\tau_{max}=\max\left\{\max_{p\in N_1}\tau_1(p), \max_{p\in N_2}\tau_2(p)\right\}
\ee
which are always finite on a compact timed-metric-space because
the cosmological time functions, $\tau_i$, are continuous.

\subsection{{\bf Timeless Intinsic Distances}}
\label{sect:defns-tl}

When Sormani-Vega first defined the null distance in \cite{SV-Null}, their announced plan was to convert a space-time into a metric-space using the cosmological time and then apply the Gromov-Hausdorff, intrinsic flat, or metric measure distance between the spaces as follows:  

\begin{defn}\label{defn:S-tls-GH}
The {\bf timeless 
space-time Gromov-Hausdorff distance between null compactifiable space-times} is defined
\index{fxc@$d^{tls}_{S-GH}$}
\be
d^{tls}_{S-GH}\Big((N_1,g_1),(N_2,g_2)\Big)
=d_{GH}
\left(\big(\bar{N}_1,\hat{d}_{g_1}\big),
\big(\bar{N}_2,\hat{d}_{g_2}\big)\right).
\ee
\end{defn}

\begin{rmrk}
\label{rmrk:timeless}
The timeless intrinsic flat distances, $d^{tls}_{S-{\mathcal F}}$ 
and 
$d^{tls}_{S-{\mathcal VF}}$,
\index{fxda@$d^{tls}_{S-{\mathcal F}}$ and $d^{tls}_{S-{\mathcal VF}}$}
and the timeless metric measure, $d^{tls}_{S-{mm}}$, 
\index{fxe@$d^{tls}_{S-{mm}}$}
distances can be defined in a similar way using the time only to define the null distance but then not involving time in the definition of the space-time distance between the spaces.
\end{rmrk}

These timeless notions have been studied in work by Vega, Sormani, Sakovich, Kunzinger, Steinbauer, Allen, Burtscher, Graf, Garci\'a-Heveling  
\cite{Sor-Ober-18}
\cite{Kunzinger-Steinbauer-22}
\cite{Allen-Burtscher-22} 
\cite{Graf-Sormani}
\cite{Burtscher-Garcia-Heveling-Global} 
\cite{Allen-Null}
\cite{SakSor-SIF}.

Note that these timeless intrinsic space-time distances do not control 
the cosmological time functions and thus have no control over the causal structure.  They are not definite.  See examples in Section~\ref{sect:definite-tl}.

\subsection{{\bf Level Matching Intrinsic Distances}}
\label{sect:defns-S-L}

It is then natural to consider the level sets
$\tau_i^{-1}(t)\subset \Omega(R_i)$ as metric-spaces with
the restricted or induced length distance functions.   
Since one can always derive the induced length distance
functions from the restricted distance functions, we will
use the restricted ones below.

\begin{defn}\label{defn:S-L-sup-GH}
The {\bf level sup $GH$ distance between null compactifiable space-times}:
\index{ga@$d_{S-GH}^{L-sup}$ and $d_{GH}^{L-sup}$}
\be \label{eq:S-L-sup-GH}
d_{S-GH}^{L-sup}\Big((N_1,g_1),(N_2,g_2)\Big)=
d_{GH}^{L-sup}
\left(\big(\bar{N}_1,\hat{d}_{g_1},\tau_{g_1}\big),
\big(\bar{N}_2,\hat{d}_{g_2},\tau_{g_2}\big)\right),
\ee
where
the {\bf level sup $GH$ distance between
compact timed-metric-spaces} is
\be
d_{GH}^{L-sup}
\left(\big(\bar{N}_1,\hat{d}_{1},\tau_{1}\big),
\big(\bar{N}_2,\hat{d}_{2},\tau_{2}\big)\right)
=
\sup d_{GH}\left(\big(\tau_1^{-1}(t),\hat{d}_{1}\big),
\big(\tau_2^{-1}(t),\hat{d}_{2}\big)\right).
\ee
with the sup taken over all
$t\in [\tau_{min},\tau_{max}]$ and
where $\tau_{min}$ and $\tau_{max}$
are as defined in (\ref{eq:tau-min}) and (\ref{eq:tau-max}).
\end{defn}

\begin{rmrk}\label{rmrk:SL-sup-F}
We can similarly define 
$d^{L-sup}_{S-{\mathcal{F}}}$,
$d^{L-sup}_{S-{\mathcal{VF}}}$
and $d^{L-sup}_{S-mm}$,
\index{gc@$d^{L-sup}_{S-{\mathcal{F}}}$ and $d^{L-sup}_{S-{\mathcal{VF}}}$}
\index{ge@$d^{L-sup}_{S-mm}$}
between space-times
with integral current structures and measures
on their level sets.
For $\mathcal{F}$ and $\mathcal{VF}$ distances it would be more natural to use the essential sup here because intrinsic flat distances are only
defined for almost every level set.  
This will appear in future work of Sakovich-Sormani \cite{SakSor-SIF}.
We leave to others to develop the corresponding notion for metric measure spaces.
\end{rmrk}

As a sup norm is often far stronger than one can hope to control time, it is natural to make the following weaker definition:

\begin{defn}\label{defn:S-L-ell-GH}
The {\bf $\ell^p$ 
 level $GH$ distance between null compactifiable space-times}:
 \index{h@$d_{S-GH}^{L-\ell p}$ and $d_{GH}^{L-\ell p}$}
 \be \label{eq:SL-ell-GH}
d_{S-GH}^{L-\ell p}\Big((N_1,g_1),(N_2,g_2)\Big)=
d_{GH}^{L-\ell p}
\left(\big(\bar{N}_1,\hat{d}_{g_1},\tau_{g_1}\big),
\big(\bar{N}_2,\hat{d}_{g_2},\tau_{g_2}\big)\right),
\ee
where
$$
d_{GH}^{L-\ell p}
\left(\big(\bar{N}_1,\hat{d}_{1},\tau_{1}\big),
\big(\bar{N}_2,\hat{d}_{2},\tau_{2}\big)\right)
=\int \left| d_{GH}\left(\big(\tau_1^{-1}(t),\hat{d}_{1}\big),
\big(\tau_2^{-1}(t),\hat{d}_{2}\big)\right)
\right|^p \, dt.
$$
where the integral
is over 
$t\in [\tau_{min},\tau_{max}]$ and
where $\tau_{min}$ and $\tau_{max}$
are as defined in (\ref{eq:tau-min}) and (\ref{eq:tau-max}).
\end{defn}

\begin{rmrk}\label{rmrk:SL-ell-F}
We can similarly define the  $\ell^p$ 
 level ${\mathcal{F}}$ 
 and ${\mathcal{VF}}$ distances between space-times, $d_{S-{\mathcal{F}}}^{L-\ell p}$
 and $d_{S-{\mathcal{VF}}}^{L-\ell p}$,
 \index{ja@$d_{S-{\mathcal{F}}}^{L-\ell p}$ and $d_{S-{\mathcal{VF}}}^{L-\ell p}$}
 .
In light 
 of the Ambrosio-Kirchheim Slicing Theorem \cite{AK} the $\ell^1$ level ${\mathcal{F}}$ distance between space-times is particularly interesting and will be studied further in upcoming work of Sakovich-Sormani \cite{SakSor-SIF}.  We leave to others to develop the corresponding notion for metric measure spaces,
 $d_{S-{mm}}^{L-\ell p}$
 \index{jc@$d_{S-{mm}}^{L-\ell p}$}.
 \end{rmrk}

Example~\ref{ex:levels-match} in Section~\ref{sect:definite-S-L} show all these level-set based notions of intrinsic distances between space-times are not definite.   However they are of some interest because convergence with respect to these intrinsic pseudodistances has some interesting consequences.  

\subsection{{\bf Big Bang Intrinsic Distances}}
\label{sect:defns-BB}

Sormani-Vega had unpublished joint work 
\cite{SV-BigBang} defining convergence of big bang space-times, 
first announced 
by Sormani in \cite{Sor-Ober-18}, using the null distance.
Their plan was to use the fact that cosmological time is the distance from the big bang point,
\be\label{eq:BB-dist-tau-2}
\tau_i(x)=\hat{d}_i(x,p_{BB,i}) \quad \forall \, x\in \bar{N}_i,
\ee
and planned to use pointed $GH$ or pointed intrinsic flat convergence
\be
(\bar{N}_i, \hat{d}_i, p_{BB,i})\to (\bar{N}_\infty, \hat{d}_\infty, p_{BB,\infty}),
\ee
to keep track of the time in the limit:
$\tau_i\to \tau_\infty$ as well.
Note that Vega went in other directions with the idea of what a big bang space-time should be 
and so they never developed and published their  ideas in
\cite{SV-BigBang}
\cite{Sor-Ober-18} and most likely never will. 
See further work by Vega in \cite{Vega-space-time}.

In this paper, we define a big bang space-time as in Definition~\ref{defn:BB}
as a causally-null-compactifiable space-time with a single big bang point in the initial data set satisfying (\ref{eq:BB-dist-tau-2})
which includes FLRW big bang space-times as mentioned in Remark~\ref{rmrk-SV-FLRW}. 
We imitate the definition of pointed $GH$ convergence in the following definition:

\begin{defn} \label{defn:BB-GH}
The {\bf intrinsic $S-BB-GH$ distance between null compactifable big bang space-times}
is 
\index{ka@$d_{S-BB-GH}$ and $d_{BB-GH}$}
\be \label{eq:BB-GH-1}
d_{S-BB-GH}\Big((N_1,g_1),(N_2,g_2)\Big)=
d_{BB-GH}\left(\big(\bar{N}_1,\hat{d}_{g_1}, \tau_1 \big),
\big(\bar{N}_2,\hat{d}_{g_2}, \tau_2\big)\right)
\ee
where the $BB-GH$ distance between 
big bang timed-metric-spaces is
the pointed $GH$ distance,
\be
d_{BB-GH}\left(\big(\bar{N}_1,\hat{d}_{1}, \tau_1\big),
\big(\bar{N}_2,\hat{d}_{2}, \tau_2 \big)\right)=
d_{pt-GH}\left(\big(\bar{N}_1,\hat{d}_{1}, p_{BB,1}\big),
\big(\bar{N}_2,\hat{d}_{2}, p_{BB,2}\big)\right),
\ee
based at the big bang points, $p_{BB,i}=\tau_i^{-1}(0)$
and where the pointed $GH$ distance 
is defined similarly to the
GH distance in Definition~\ref{defn:GH}  
by the
following infimum over all distance preserving maps $\varphi_i: \bar{N}_i \to Z$  and over all compact $Z$:
$d_{pt-GH}\left(\big(\bar{N}_1,\hat{d}_{1}, p_{BB,1}\big),
\big(\bar{N}_2,\hat{d}_{2}, p_{BB,2}\big)\right)=$
\be \label{eq:BB-GH-2}
=\inf \Bigg( d_H^Z\big(\varphi_1(\bar{N}_1),\varphi_2(\bar{N}_2)\big)
+ d_Z(\varphi_1(p_{BB,1}),\varphi_2(p_{BB,2}))
\Bigg)
\ee
\end{defn}

We will prove this distance between causally-null-compactifiable space-times is definite in Theorem~\ref{thm:BB-GH-definite}.

\begin{rmrk}\label{rmrk:BB-GH}
With the strong control over the basepoints in Definition~\ref{defn:BB-GH}, convergence of big bang space-times with respect to the intrinsic $S-BB-GH$ distance implies the big bang points converge to big bang points in the limit as in the original plan by Sormani-Vega \cite{SV-BigBang}.  The distances between points converge so this implies that the $\tau_i \to \tau_\infty$ as well.      
\end{rmrk}

\begin{rmrk}\label{rmrk-BB-F}
We can similarly define the intrinsic $S-BB-{\mathcal{F}}$, $S-BB-{\mathcal{VF}}$, and $S-BB-mm$ distances,
$d_{S-BB-{\mathcal{F}}}$, $d_{S-BB-{\mathcal{VF}}}$
and $d_{S-BB-mm}$,
 \index{kb@$d_{S-BB-{\mathcal{F}}}$ and $d_{S-BB-{\mathcal{VF}}}$}
\index{kd@$d_{S-BB-mm}$}
between big bang space-times
with integral current structures and measures
respectively.
However, there are some subtleties now.   It is possible for a sequence of big bang space-times to become thin near the big bang points so that the big bang points disappear under $BB-{\mathcal{F}}$, $BB-{\mathcal{VF}}$, and $BB-mm$.   To compensate for this possibility and allow for such sequences of space-times to converge, we will define the intrinsic timed notions of convergence in Section~\ref{sect:defns-tau-K}.
\end{rmrk}

\subsection{{\bf Future-Developed Intrinsic Distances}}
\label{sect:defns-FD}

Given a compact initial data set, $(M,h,k)$, satisfying the Einstein constraint equations, 
Choquet-Bruhat proved there is a globally hyperbolic space-time, $(N,g)$,
satisfying the Einstein vacuum equations such that $M$ lies as a Cauchy surface in $N$ with restricted metric tensor $h$ and second fundamental form $k$
\cite{Ch-Br}. In joint work with Geroch, she proved
there is a unique maximal development, $(N,g)$, which contains all such space-times \cite {Ch-Br-Ge}. Similar results are also proven for various other solutions to Einstein's Equations with various matter models. Many mathematicians have studied the stability of solutions to the
Einstein equations: Suppose $(M_j,h_j,K_j)\to (M_\infty, h_\infty,K_\infty)$, can one conclude that the global maximal developments converge, $(N_j,g_j)\to(N_\infty,g_\infty)$. 
For a survey of known results see the
books edited by Choquet-Bruhat \cite{Choquet-Bruhat-compact} and 
written by Ringstrom 
\cite{Ringstrom-Cauchy-text}.

To use the ideas we are developing in this paper we first convert the $(N_j,g_j)$ into timed-metric-spaces using the null distance
in Section~\ref{sect:space-times-initial-data}.  Consider a pair of
causally-null-compactifiable future-developed space-times, $(N_{j,+},g_j)$,
with initial surfaces, $(M_j,h_j,K_j)$,
as in Definition~\ref{defn:FD}.
These have
associated compact timed-metric-space, $(\bar{N}_{j,+},\hat{d}_{g_j},\tau_{g_j})$
with $M_j=\tau_{g_j}^{-1}(0)\subset \bar{N}_{j,+}$, and
$\tau_{g_j}$ is the distance to this initial data set.  Thus we can control time and causal structure
as long as we control this initial data set.   We introduce the following notion:

\begin{defn}\label{defn:FD-HH}
The {\bf intrinsic $S-FD-HH$ distance between causally-null-compactifiable future-developed space-times} is defined 
\index{la@$d_{S-FD-HH}$ and $d_{FD-HH}$}
\be \label{lc@eq:FD-HH-1}
d_{S-FD-HH}\Big((N_{1,+},g_1),(N_{2,+},g_2)\Big)=
d_{FD-HH}\left(\big(\bar{N}_{1,+},
M_1,\hat{d}_{g_1}\big),
\big(\bar{N}_{2,+}, M_2,\hat{d}_{g_2}\big)\right),
\ee
where the $FD-HH$ distance between
timed-metric-spaces satisfying the distance to
initial data property as in
(\ref{eq:init-data-prop-2})
is 
\be \label{eq:FD-HH-2}
=\inf \Bigg( d_H^Z\big(\varphi_1(\bar{N}_{1,+}),\varphi_2(\bar{N}_{2,+})\big)
+ 
d_H^Z\big(\varphi_1(M_1),\varphi_2(M_2)\big)
\Bigg),
\ee
where the infimum is over all distance preserving maps $\varphi_i: \bar{N}_i \to Z$  and over all complete $Z$.
\end{defn}

We will prove this distance between causally-null-compactifiable space-times is definite in Theorem~\ref{thm:FD-HH}.   So we can define convergence using this notion as follows:

\begin{defn}\label{defn:FD-conv}
We say that a sequence of causally-null-compactifiable future-developed space-times $(N_{j,+},g_j)$ converges to $(N_{\infty,+},g_\infty)$ in the
$d_{S-FD-HH}$ sense iff
\be \label{eq:FD-HH-conv-1}
d_{S-FD-HH}\Big((N_{j,+},g_j),(N_{\infty,+},g_\infty)\Big)\to 0.
\ee
See Figure~\ref{fig:FD-conv}.
\end{defn}

\begin{figure}[h] 
   \centering
   \includegraphics[width=5in]{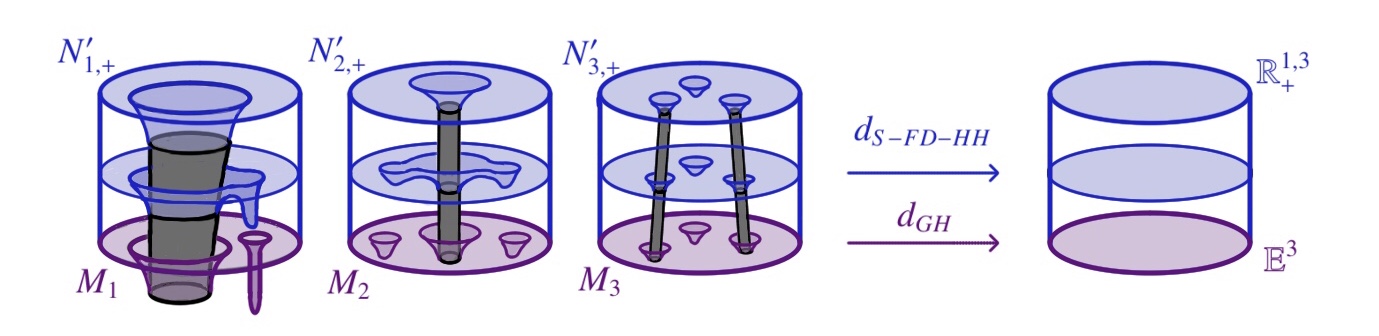}
   \caption{ Here $(N'_{j,+},g_j)$ are exteriors of space-times with merging black holes and stars where the stars and the black holes are decreasing in size as $j\to \infty$ so that $(N'_{j,+},g_j)$ intuitively appears to converge in the $d_{S-FD-HH}$ sense to future Minkowski space, $({\mathbb R}^{1,3}_+, g_{Mink})$, and the initial data $(M'_j,h_j)$ converge in the $GH$ sense to Euclidean
   space $({\mathbb E}^3,g_{\mathbb E})$
   as in Definitions~\ref{defn:FD-HH}-~\ref{defn:FD-conv}.
   }\label{fig:FD-conv}
\end{figure}

\begin{rmrk}
This notion of $FD-HH$ convergence agrees with
the
notion
of Gromov-Hausdorff Convergence for Metric Pairs introduced by Che, Galaz-Garc\'ia, Guijarro,
and Membrillo Solis in \cite{CGGGMS}.
\end{rmrk}

\begin{rmrk}\label{rmrk:FD-FF}
The intrinsic $S-FD-{FF}$ 
 distance, $d_{S-FD-FF}$,
 \index{ld@$d_{S-FD-FF}$}
 between causally-null-compactifiable future-developed space-times would be similarly
defined by taking the sum of the
flat distances in $Z$.   If we wish to add volume preserving we
could add a difference in volumes as well.  Further work in this direction is in progress by Sakovich and Sormani.
\end{rmrk}

\begin{rmrk} \label{rmrk:FD-HF}
It is also possible that one might wish to mix notions of distance
defining the intrinsic $S-FD-FH$,
 distance, $d_{S-FD-FH}$,
 \index{lf@$d_{S-FD-FH}$}, between causally-null-compactifiable future-developed space-times taking the flat distance between the images of
 the space-times and the Hausdorff distance between the images of the
 initial data sets. This has
 the advantage of controlling the initial data and thus the time function and causal structure more strongly than if the initial data set were allowed to have regions which disappear. 
 \end{rmrk}

 \begin{rmrk}\label{rmrk:FD-mm}
The intrinsic $S-FD-mm$
and $S-FD-mH$
 distances,
 $d_{S-FD-mm}$,
 \index{lw@$d_{S-FD-mm}$}
 and 
 $d_{S-FD-mH}$,
 \index{lx@$d_{S-FD-mH}$}
 between causally-null-compactifiable future-developed space-times can also be defined using Wasserstein distances or
 combinations of Wasserstein and
 Hausdorff distances.   Further work in this direction is in progress by Perales and Mondino.
 \end{rmrk}

\begin{rmrk} \label{rmrk:Che-Gomez}
Mauricio Che and Andr\'es Ahumada G\'omez
 have developed a pre-compactness theorem
for GH convergence of metric pairs 
in \cite{Che-Gomez-pairs}. 
They believe they may be able to apply to 
obtain a $S-FD-HH$ precompactness
theorem for sequences of causally null compactifiable future developed space-times.
Further work in this direction will appear soon.
\end{rmrk}

\subsection{{\bf Strip-based Intrinsic Distances}}
\label{sect:defns-strip}
The notion of an intrinsic distance between space-times
using cosmic strips, that we publish for the first time here in this section, was first announced by Sakovich-Sormani at their talks at the Fields Institute in 
2022.   

First let us recall a few facts about cosmic strips in space-times.  In Theorem~\ref{thm:proper} we proved that space-times, $(N_j,g_j)$,
with proper regular cosmological time functions, $\tau_j$, have causally-null-compactifiable cosmic strips,
$N_{j,s,t}=\tau_j^{-1}(s,t)$
with a number of nice properties.
In Theorem~\ref{thm:local-cosmo}
we extended this to a larger
class of space-times with more general time functions.
We also presented examples of regions which exhaust space-times, which can be viewed as space-times themselves with causally-null-compactifiable strips. In this section we study intrinsic distances
between such space-times.

\begin{defn}\label{defn:strip-sup-GH}
The {\bf strip sup $GH$ distance between null compactifiable space-times}:
\index{maa@$d_{S-GH}^{strip-sup}$ and $d_{GH}^{strip-sup}$}
\be \label{eq:strip-sup-GH}
d_{S-GH}^{strip-sup}\Big((N_1,g_1),(N_2,g_2)\Big)=
d_{GH}^{strip-sup}
\left(\big(\bar{N}_1,\hat{d}_{g_1},\tau_{g_1}\big),
\big(\bar{N}_2,\hat{d}_{g_2},\tau_{g_2}\big)\right),
\ee
where the {\bf strip sup $GH$ distance between
compact timed-metric-spaces} is
$$
d_{GH}^{strip-sup}
\left(\big(\bar{N}_1,\hat{d}_{1},\tau_{1}\big),
\big(\bar{N}_2,\hat{d}_{2},\tau_{2}\big)\right)
=
\sup d_{GH}\left(\big(\tau_1^{-1}[s,t],\hat{d}_{1}\big),
\big(\tau_2^{-1}[s,t],\hat{d}_{2}\big)\right).
$$
with the sup taken
over all $s,t\in [\tau_{min},\tau_{max}]$
with $s\le t$
and
where $\tau_{min}$ and $\tau_{max}$
are as defined in (\ref{eq:tau-min}) and (\ref{eq:tau-max}).
\end{defn}

\begin{rmrk}\label{rmrk:strip-sup-F}
We can similarly define 
$d^{strip-sup}_{S-{\mathcal{F}}}$,
$d^{strip-sup}_{S-{\mathcal{VF}}}$
and $d^{strip-sup}_{S-mm}$,
 \index{mc@$d^{strip-sup}_{S-{\mathcal{F}}}$ and $d^{strip-sup}_{S-{\mathcal{VF}}}$}
\index{me@$d^{strip-sup}_{S-mm}$}
between cosmic strips
with integral current structures and measures
respectively.
For $\mathcal{F}$ and $\mathcal{VF}$ distances it would be more natural to use the essential sup here because intrinsic flat distances are only
defined for almost every strip.  So we define the
{\bf esssup strip $\mathcal{F}$ distance between space-times}
as follows,
\be \label{eq:SL-sup-F}
d_{S-{\mathcal{F}}}^{strip-sup}\Big((N_1,g_1),(N_2,g_2)\Big)=
esssup \,d_{\mathcal{F}}\left(\big(\tau_1^{-1}(s,t),\hat{d}_{1}\big),
\big(\tau_2^{-1}(s,t),\hat{d}_{2}\big)\right),
\ee
with the $esssup$ taken
over all $s,t\in [\tau_{min},\tau_{max}]$
with $s\le t$
and similarly the {\bf esssup level $\mathcal{VF}$ distance between space-times}.   We leave to others to develop the corresponding notion for metric measure spaces.
\end{rmrk}

As a sup norm is often far stronger than one can hope to control time, it is natural to make the following weaker definition:

\begin{defn}\label{defn:strip-ell-GH}
The {\bf $\ell^p$ 
 level $GH$ distance between null compactifiable space-times}: 
 \index{na@$d_{S-GH}^{strip-\ell p}$ and $d_{GH}^{strip-\ell p}$}
 \be \label{eq:strip-ell-GH}
d_{S-GH}^{strip-\ell p}\Big((N_1,g_1),(N_2,g_2)\Big)=
d_{GH}^{strip-\ell p}
\left(\big(\bar{N}_1,\hat{d}_{g_1},\tau_{g_1}\big),
\big(\bar{N}_2,\hat{d}_{g_2},\tau_{g_2}\big)\right)
\ee
where
$$
d_{GH}^{strip-\ell p}
\left(\big(\bar{N}_1,\hat{d}_{1},\tau_{1}\big),
\big(\bar{N}_2,\hat{d}_{2},\tau_{2}\big)\right)
=
\int \left| d_{GH}\left(\big(\tau_1^{-1}(s,t),\hat{d}_{1}\big),
\big(\tau_2^{-1}(s,t),\hat{d}_{2}\big)\right)
\right|^p \, ds \,dt
$$
where the 
integration is over
\be
\{(s,t): \, \tau_{min}\le s<t \le \tau_{max}\,\}.
\ee
\end{defn}

\begin{rmrk}\label{rmrk:strip-ell-F}
We can similarly define 
$d^{strip-\ell p}_{S-{\mathcal{F}}}$,
$d^{strip-\ell p}_{S-{\mathcal{VF}}}$
and $d^{strip-\ell p}_{S-mm}$,
 \index{nc@$d^{strip-\ell p}_{S-{\mathcal{F}}}$ and $d^{strip-\ell p}_{S-{\mathcal{VF}}}$}
\index{ne@$d^{strip-\ell p}_{S-mm}$}
between cosmic strips
with integral current structures and measures
respectively.
In light 
 of the Ambrosio-Kirchheim Slicing Theorem \cite{AK} the $\ell^1$ level ${\mathcal{F}}$ distance between space-times is particularly interesting and will be studied further in upcoming work of Sakovich-Sormani.  We leave to others to develop the corresponding notion for metric measure spaces.
 \end{rmrk}

In Section~\ref{sect:definite-strip} we discuss the difficulties that arise when trying to prove these strip-based definitions of intrinsic flat distances are definite.   It is an open question whether they are definite or not.

\subsection{{\bf Intrinsic Timed Distances between Timed-Metric-Spaces}}
\label{sect:defns-tau-K}

In this section we introduce what may be the best intrinsic distance between causally-null-compactifiable space-times.   It applies to a far larger class of space-times than the intrinsic distances between big bang and future-developed space-times,
and we prove the notion is definite.

Recall that in Section~\ref{sect:back-Kuratowski} we reviewed the
notion of a Fr\'echet map,
$\kappa_X: X\to \ell^\infty$,
and discussed how alternate definitions of the Gromov-Hausdorff, metric measure, and intrinsic flat distances use these Fr\'echet maps in place of arbitrary distance preserving maps into complete metric-spaces.  

Here we introduce the notion of a timed-Fr\'echet map in
Definition~\ref{defn:tau-K}.\footnote{In our v1 post on the arxiv we accidentally named these maps timed-Kuratowski maps because we accidentally attributed the Fr\'echet maps to Kuratowski.}
We apply this notion to define 
intrinsic timed-Hausdorff distances and other timed distances between timed-metric-spaces in
Definition~\ref{defn:tau-K-dist}
and Remark~\ref{rmrk:tau-K-dist}.  
We prove the timed-Fr\'echet map is distance and time preserving in Propositions~\ref{prop:tau-K-dist-pres}
and~\ref{prop:tau-K-time-pres} which we will apply later to prove the intrinsic timed-Hausdorff distance  is definite
in Section~\ref{sect:definite-tau-K}.  

Recall that in Definition~\ref{defn:cncst}, a compact timed-metric-space, $(X,d,\tau)$, is a compact metric-space endowed with a Lipschitz one time function, $\tau:X \to [\tau_{min},\tau_{max}]$, and
a causal structure defined as in (\ref{eq:causal}).

\begin{defn}\label{defn:tau-K}
Given a compact timed-metric-space, $(X,d,\tau)$,
and a countably dense collection of points,
\be
{\mathcal N}=\{
x_1,x_2,x_3,...\}\subset X
\ee
we define the
timed-Fr\'echet map, 
\index{pa@$\kappa_{\tau,X}=\kappa_{\tau,X,{\mathcal N}}$}
$
\kappa_{\tau,X}
=\kappa_{\tau,X,{\mathcal N}}: X \to [0,\tau_{max}]\times \ell^\infty\subset \ell^\infty
$,
as follows:
\textcolor{blue}{
\be
\kappa_{\tau,X}(x)=
(\tau(x),\kappa_{X,{\mathcal N}}(x))=
(\tau(x), d(x_1,x),
d(x_2,x),
d(x_3,x),...).
\ee}
\end{defn}

We can now define new
intrinsic distances between causally-null-compactifiable space-times:

\begin{defn}
\label{defn:tau-K-dist}
The {\bf intrinsic 
Timed-Hausdorff
distance
between two causally-null-compactifiable space-times}, \index{pb@$d_{S-\tau-H}$ and $d_{\tau-H}$}
\be
d_{S-\tau-H}\Big((N_1,g_1),(N_2,g_2)\Big)
=d_{\tau-H}
\left(\big(\bar{N}_1,\hat{d}_{g_1},\tau_{g_1}\big),
\big(\bar{N}_2,\hat{d}_{g_2},\tau_{g_2}\big)\right),
\ee
where the {\bf
intrinsic 
timed-Hausdorff
distance
between two compact timed-metric-spaces}
\be \label{eq:tK-GH-1}
d_{\tau-H}
\Big((X_1,d_1,\tau_1),(X_2,d_2,\tau_2)\Big)=
\inf\, d^{\ell^\infty}_H(
\kappa_{\tau_1,X_1}(X_1),
\kappa_{\tau_2,X_2}(X_2))
\ee
where the infimum is taken over all possible 
timed-Fr\'echet maps,
$\kappa_{\tau_1,X_1}:X_1\to \ell^\infty$ and
$\kappa_{\tau_2,X_2}:X_2\to \ell^\infty$ which are
found by considering all countable dense collections of points in the $X_i$ and all reorderings of these collections of points.
\end{defn}

\begin{rmrk}\label{rmrk:tau-K-dist}
Similarly we may define the 
intrinsic timed-mm 
and intrinsic timed-flat distances, 
$d_{S-\tau-mm}$,
$d_{S-\tau-F}$ and
$d_{S-\tau-VF}$
\index{pd@$d_{S-\tau-mm}$}
\index{pe@$d_{S-\tau-F}$ and $d_{S-\tau-VF}$}.   We leave the metric measure notions to Perales and Mondino and explore the intrinsic flat notion ourselves in upcoming work
\cite{SakSor-SIF}.
\end{rmrk}

The following propositions provide some intuition about timed-Fr\'echet maps and will be applied later when proving definiteness:

\begin{prop}
\label{prop:tau-K-dist-pres}
Any
timed-Fr\'echet map, 
$
\kappa_{\tau,X}: X \to [0,\tau_{max}]\times \ell^\infty\subset \ell^\infty,
$
is distance preserving,
\be
d_{\ell^\infty}(\kappa_{\tau,X}(x),
\kappa_{\tau,X}(y))=d(x,y),
\ee
for any metric-space, $(X,d)$, with a Lipschitz one time function, $\tau:X\to[0,\tau_{max}]$
and any choice of countably dense points in $X$.
\end{prop}

\begin{proof}
Let $x,y\in X$, then
\be
d_{\ell^\infty}(\kappa_{\tau,X}(x),
\kappa_{\tau,X}(y))
=\max\{|\tau(x)-\tau(y)|,
d_{\ell^\infty}(\kappa(x),
\kappa(y))\}.
\ee
Since the Fr\'echet map is distance preserving (cf. Theorem~\ref{thm:K-dist-pres})
and $\tau$ is Lipschitz one,
\be
d_{\ell^\infty}(\kappa(x),
\kappa(y))=
d_X(x,y)\ge |\tau(x)-\tau(y)|.
\ee
Thus
\be
d_{\ell^\infty}(\kappa(x),
\kappa(y))=\max\{
|\tau(x)-\tau(y)|,
d_X(x,y)\}=d_X(x,y).
\ee
\end{proof}

\begin{prop}
\label{prop:tau-K-time-pres}
If we consider the infinite dimensional
timed-metric-space, 
\be
([0,\tau_{max}]\times \ell^\infty,
d_{\ell^\infty},\tau_\infty
),
\ee
where 
\be
\tau_\infty(w_0,w_1,w_2,...)=w_0
\ee
is endowed with a causal structure
defined by
\be
w\in J^+(z)
\iff 
\tau_\infty(w)-\tau_\infty(z)=
d_{\ell^\infty}(w,z)
\ee
then the
timed-Fr\'echet map 
$
\kappa_{\tau,X}: X \to [0,\tau_{max}]\times \ell^\infty\subset \ell^\infty
$
is time preserving:
\be\label{eq:tau-K-time-pres}
\tau_\infty(\kappa_{\tau,X}(x))=\tau(x)
\ee
and preserves causality:
\be\label{eq:tau-K-causal-pres}
\kappa_{\tau,X}(p)\in J^+(\kappa_{\tau,X}(q))
\iff p\in J^+(q)
\ee
for any metric-space, $(X,d)$, with a Lipschitz one time function, $\tau:X\to[0,\tau_{max}]$ that
encodes causality as in (\ref{eq:causal}).
\end{prop}

\begin{proof}
The first claim that
(\ref{eq:tau-K-time-pres}) holds follows from the definition of $\tau_\infty$ and the second claim that (\ref{eq:tau-K-causal-pres}) holds follows
from (\ref{eq:tau-K-time-pres}),
(\ref{eq:causal}), and
Proposition~\ref{prop:tau-K-dist-pres}.
\end{proof}

\begin{rmrk}\label{rmrk:tau-K-time-pres}
It may be possible to apply Proposition~\ref{prop:tau-K-time-pres} to provide alternate definitions of 
intrinsic timed-Hausdorff and timed-Flat distances which takes infima over all distance and time preserving maps into all compact or perhaps over all complete timed-metric-spaces rather than the infima over all timed-Fr\'echet maps.  Sakovich and Sormani will explore this in the intrinsic flat setting but leave both the $GH$ and $mm$ to Mondino and Perales.
\end{rmrk}

\begin{defn}\label{def:new-pres}
A map, $F: (X,d_X,\tau_X)\to (Y, d_Y, \tau_Y)$,
between timed-metric-spaces
is time preserving 
if $\tau_Y(F(x))=\tau_X(x)$ for all $x\in X$ and
distance preserving if
$
d_Y(F(x_1),F(x_2))=d_X(x_1,x_2) 
$
for all $x_1, x_2\in X$.
\end{defn}

\begin{rmrk}\label{rmrk-new-tau-K-1}
If we define the timed-metric-space,
$(Y',d_Y, \tau_Y)$, where
$
Y'=
[0,\tau_{max}]\times  \ell^\infty,$
and 
\be
d_Y((t_1,z_1),(t_2,z_2))=
\max\{|t_1-t_2|,d_{\ell^\infty}(z_1,z_2)\},
\ee
and
$\tau_Y(t,z)=t$,
then any
timed-Fr\'echet map,
$\kappa_{\tau,X}:X\to Y$ is distance and time preserving.   Furthermore, if $X$ is
compact and we take $Y=\kappa_{\tau,X}(X)$
then $(Y, d_Y, \tau_Y)$ is a compact 
timed-metric-space.
\end{rmrk}

\begin{rmrk}\label{rmrk-new-tau-K-2}
We conjecture that for any pair of compact
timed-metric-spaces, $N_i=(X_i,d_i,\tau_i)$,
we have
\be
d_{\tau-H}(N_1,N_2)=\inf 
\{d_H^Y(F_1(X_1),F_2(X_2))\}
\ee
where the infimum is over all compact timed-metric-spaces, $N=(Y,d,\tau)$,
and over all distance and time preserving maps $F_i:X_i\to Y$.  As we do not need this, we leave it 
to others to prove.
\end{rmrk}

\section{{\bf{Verifying which Intrinsic Space-Time Distances are Definite}}}
\label{sect:definite}

Recall that in the introduction we defined a definite intrinsic distance in Definition~\ref{defn:sd} as a distance between causally-null-compactifiable space-times satisfying the following:
\be
d_{S-dist}\Big((N_1,g_1),(N_2,g_2)\Big)=0
\ee
iff there is a bijection, $F: N_1\to N_2$ that is
distance preserving and time preserving,
and thus, by Theorem 1.3 in \cite{SakSor-Null}, is also a Lorentzian isometry. 

In Section~\ref{sect:definite-tl} below we demonstrate why a null distance preserving bijection without the time preserving does not suffice to achieve a Lorentzian isometry.  Thus the
timeless intrinsic distances between space-times of Section~\ref{sect:defns-tl} are not definite.
In Section~\ref{sect:definite-S-L} we provide level set matching examples which have no Lorentzian isometry demonstrating that the  intrinsic distances between space-times of Section~\ref{sect:defns-S-L} are not definite and neither are those from Remark~\ref{rmrk:Riem-levels}.

In Section~\ref{sect:definite-BB} and
Section~\ref{sect:definite-FD} we prove that the intrinsic distances between causally-null-compactifiable big bang space-times,
$d_{BB-GH}$, and between
causally-null-compactifiable
future-developed space-times,
$d_{FD-HH}$,
are definite and conjecture that the other notions described in Section~\ref{sect:defns-BB} and Section~\ref{sect:defns-FD} could possibly be definite as well.  However in Section~\ref{sect:definite-strip} we see the situation is not so clear for the strip based notions of intrinsic distances of Section~\ref{sect:defns-strip}. 

In Section~\ref{sect:definite-tau-K} we prove our main theorem, Theorem~\ref{thm:definite-tau-K}, that the intrinsic timed-Hausdorff distance between causally-null-compactifiable space-times is definite.
This theorem is quite challenging to prove but is done through a series of lemmas which might later be applied to prove our conjectures that the intrinsic timed flat distance and other notions defined in Section~\ref{sect:defns-tau-K} using the timed-Fr\'echet map are definite as well.

\subsection{{\bf The Timeless Intrinsic Space-Time Distances are not Definite}}
\label{sect:definite-tl}

In this section
we present very simple examples of distinct null compactifiable space-times which 
demonstrate that ignoring cosmological time when defining a timeless intrinsic space-time distance as in Section~\ref{sect:defns-tl}
leads to lack of definiteness.  The timeless space-time distances between these space-times are zero because they have null distance preserving isometries. However there is no Lorentzian isometry between these space-times.

\begin{ex}\label{ex:t-to-x}
Consider the following pair of Lorentzian manifolds,
\be
(N_j,g_j)= \bigg( (0,a_j)\times (0,b_j)\, , \, -dt^2+dx^2 \bigg).
\ee
Then by Lemma~\ref{lem:Mink-null}
\be
\hat{d}_{g_j}\bigg( (t_1,x_1), (t_2,x_2) \bigg)=\max\{|t_1-t_2|, |x_1-x_2)\}.
\ee
Thus if $a_1=b_2 \neq b_1=a_2$, then $F:N_1\to N_2$ defined by $F(t,x)=(x,t)$
is a null distance preserving bijection,
\be
\hat{d}_{g_2}\bigg( F(t_1,x_1), F(t_2,x_2) \bigg)=\hat{d}_{g_1}\bigg( (t_1,x_1), (t_2,x_2) \bigg) \forall \, (t_i,x_i)\in N_1,
\ee
so
\be
d^{tls}_{S-GH}((N_1,g_1),(N_2,g_2))=
d_{GH}\left(\big(\bar{N}_1,\hat{d}_{g_1},\tau_{g_1}\big),
\big(\bar{N}_2,\hat{d}_{g_2},\tau_{g_2}\big)\right)=0.
\ee
However $F:N_1\to N_2$ is not a Lorentian isometry $F^*g_2= -dx^2+dt^2\neq -dt^2+dx^2=g_1$.
It does not even map causal curves to causal curves.
\end{ex}

\begin{ex}\label{Ex:flip-t} 
Consider the pair of Lorentzian manifolds,
\be
(N_j,g_j)= \bigg( (0,1)\times M_j\, , \, -dt^2+ f_j(t) h_j^2 \bigg).
\ee
where $f_1(t)=1$ and $f_2(t)=1-t$.   Then $F:N_1\to N_2$ defined by $F(t,x)=(1-t,x)$,
maps causal curves to causal curves, so:
\be
\hat{d}_{g_1}\bigg( (t_1,x_1), (t_2,x_2) \bigg)=\hat{d}_{g_2}\bigg( (t_1,x_1), (t_2,x_2) \bigg)
\ee
Thus $F$ is a null distance preserving bijection, 
so
\be
d^{tls}_{S-GH}((N_1,g_1),(N_2,g_2))=
d_{GH}\left(\big(\bar{N}_1,\hat{d}_{g_1},\tau_{g_1}\big),
\big(\bar{N}_2,\hat{d}_{g_2},\tau_{g_2}\big)\right)=0.
\ee
While $F$ is a Lorentzian isometry, 
it is time-reversing and 
the causality is flipped,
\be
F(p)\in J^+(F(q)) \iff q\in J^+(p).
\ee
\end{ex}

\subsection{{\bf The Level Matching Intrinsic Space-Time Distances are not Definite}}
\label{sect:definite-S-L}

In this section
we present examples of distinct null compactifiable space-times which 
demonstrate that level matching notions of intrinsic space-times distances, like 
those mentioned in Remark~\ref{rmrk:Riem-levels}, and
$d_{S-L-sup-GH}$ of Definition~\ref{defn:S-L-sup-GH}
and $d_{S-GH}^{L-\ell p}$ of
Definition~\ref{defn:S-L-ell-GH}
of Section~\ref{sect:defns-S-L}
need not be definite.  However we do have
isometries between the level sets
as in Theorem~\ref{thm:levels} and
Remark~\ref{rmrk:levels}.

\begin{ex}\label{ex:levels-match}
Consider the following pair of Lorentzian manifolds,
\be
(N_j,g_j)= \bigg( (0,\pi)\times {\mathbb S}^2\, , \, -dt^2+ d\phi^2 + f_j^2(t,\theta) \sin^2(\phi) d\theta^2 \bigg).
\ee
where $h_j=d\phi^2 + f_j^2(t,\theta) \sin^2(\phi) d\theta^2$ is a smooth family of metric tensors on
${\mathbb S}^2$ with nonconstant $f_j$ such that 
\be
f_1(t,\theta,\phi)=f_2(t,t+\theta, \phi) 
\ee
Then it is easy to check that $\tau_{g_j}(t,x,y)=t$ and that
\be
F_t: \tau_{g_1}^{-1}(t) \to  \tau_{g_2}^{-1}(t) \textrm{ defined by } F_t(\theta,\phi)=(t+\theta, \phi)
\ee
is a Riemannian isometry for all
$t\in [\tau_{min},\tau_{max}]$.  
This
implies that
\be
d_{S-GH}^{L-sup}((N_1,g_1),(N_2,g_2))=0
\ee
and
\be
d_{S-GH}^{L-\ell p}((N_1,g_1),(N_2,g_2))=0.
\ee
However these manifolds need not have a Lorentzian isometry between them unless $f_j$ is constant.   
\end{ex}

\begin{thm} \label{thm:levels}
Given two compact timed-metric-spaces,
\be\label{eq:sup-0}
d_{GH}^{L-sup}
\left(\big(\bar{N}_1,\hat{d}_{1},\tau_{1}\big),
\big(\bar{N}_2,\hat{d}_{2},\tau_{2}\big)\right)=0
\ee
if and only if
they have null distance preserving maps 
\be\label{eq:F_t}
F_{t}: \tau^{-1}_1(t)\subset N_{1}\to
\tau^{-1}_2(t)\subset N_{2}
\ee
for every $t \in [\tau_{min}, \tau_{max}]$.
In addition,
\be\label{eq:ell-p-0}
d_{GH}^{L-\ell p}\left(\big(\bar{N}_1,\hat{d}_{1},\tau_{1}\big),
\big(\bar{N}_2,\hat{d}_{2},\tau_{2}\big)\right)=0,
\ee
iff there exists $F_t$ as in (\ref{eq:F_t})
for almost every $t\in [\tau_{min}, \tau_{max}]$.
In both cases,
\be \label{eq:min=min-max=max}
\min_{\bar{N}_1} \tau_1=
\min_{\bar{N}_2} \tau_2
\textrm{ and }
\max_{\bar{N}_1}\tau_1=
\max_{\bar{N}_2}\tau_2.
\ee
\end{thm}

\begin{proof}
We have (\ref{eq:sup-0})
iff
\be\label{eq:GH-level-0}
d_{GH}\left(\big(\tau_1^{-1}(t),\hat{d}_{1}\big),
\big(\tau_2^{-1}(t),\hat{d}_{2}\big)\right)=0
\ee
for all $t \in [\tau_{min}, \tau_{max}]$
iff (\ref{eq:F_t})
for all $t \in [\tau_{min}, \tau_{max}]$
by the definiteness of the $GH$ distance.
We also have (\ref{eq:ell-p-0})
iff (\ref{eq:GH-level-0}) for almost every
$t \in [\tau_{min}, \tau_{max}]$
iff (\ref{eq:F_t})
for almost every $t \in [\tau_{min}, \tau_{max}]$.

Note that the $GH$ distance between an empty metric-space and a nonempty metric-space is never zero,
so in either case,
\be
\tau_2^{-1}(t)=\emptyset \iff
\tau_1^{-1}(t)=\emptyset
\ee
for almost every $t\in [\tau_{min}, \tau_{max}]$.
Thus
\be
\min_{\bar{N}_1} \tau_1=\inf\{t\,|\tau_1^{-1}(t)\neq\emptyset \}
=\inf\{t\,|\tau_2^{-1}(t)\neq\emptyset \}
=\min_{\bar{N}_2} \tau_2
\ee
and the same for the maxima.
\end{proof}

\begin{rmrk}\label{rmrk:levels}
If any of the
intrinsic space-time distances defined using levels as in Section~\ref{sect:defns-S-L} between two causally-null-compactifiable space-times, or more generally, two compact-timed-metric-spaces is $0$, then we believe that for almost every $t\in [\tau_{min},\tau_{max}]$ there
is a distance preserving maps as in (\ref{eq:F_t}) and that (\ref{eq:min=min-max=max}) holds.   Sakovich-Sormani plan to prove this in the intrinsic flat setting and we leave the metric measure case to others, however the proof should be close to the proof above.
\end{rmrk}

\begin{rmrk}\label{rmrk:Riem-levels-2}
If two causally-null-compactifiable space-times
have
\be
d_{S-GH}^{L-sup}\Big((N_1,g_1),(N_2,g_2)\Big)=0
\ee
or
\be
d_{S-GH}^{L-\ell p}\Big((N_1,g_1),(N_2,g_2)\Big)=0
\ee
then we have the maps between levels as in Theorem~\ref{thm:levels}.  
Something similar
should hold for other intrinsic distances
from Section~\ref{sect:defns-S-L}.
\end{rmrk}

\subsection{{\bf The Big Bang Intrinsic Space-Time Distances are Definite}}
\label{sect:definite-BB}

Recall Definition~\ref{defn:BB} of a null compactifiable big bang space-time.  Here we prove that $d_{BB-GH}$ of
Definition~\ref{defn:BB-GH},
between such space-times, is definite in the sense of Definition~\ref{defn:sd}.   This proof was known to Sormani and Vega but never published
\cite{SV-BigBang}.

\begin{thm}\label{thm:BB-GH-definite}
Given two compact big bang 
timed-metric-spaces, 
\be \label{eq:BB-GH-here}
d_{BB-GH}\left(\big(\bar{N}_1,\hat{d}_{1},\tau_1 \big),
\big(\bar{N}_2,\hat{d}_{2}, \tau_2\big)\right)=0,
\ee
if and only if there is a
distance and time preserving
bijection, $F:N_1\to N_2$.
In particular,
given two causally-null-compactifiable big bang 
metric-space-times, we have
\be \label{eq:BB-GH-here2}
d_{S-BB-GH}\Big((N_1,g_1),(N_2,g_2)\Big)=0
\ee
iff there is a Lorentzian isometry satisfying $g_1=F^*g_2$ everywhere.
\end{thm}

\begin{rmrk}\label{rmrk:BB-definite}
A similar theorem will be proven for the 
$BB-{\mathcal F}$ distance in upcoming work by various authors and the $BB-mm$ distance is left for others.   Such proofs can imitate the proof below.
\end{rmrk}

\begin{proof}
If
(\ref{eq:BB-GH-here}) holds
then
\be \label{eq:BB-GH-here3}
d_{pt-GH}\left(\big(\bar{N}_1,\hat{d}_{1}, p_{BB,1}\big),
\big(\bar{N}_2,\hat{d}_{2}, p_{BB,2}\big)\right)=0
\ee
where $p_{BB,i}=\tau_i^{-1}(0)$.
By the definiteness of $d_{pt-GH}$ there is distance preserving isometry which preserves the big bang points:
\be
F: \bar{N}_1\to \bar{N}_2
\textrm{ with } F(p_{BB,1})=p_{BB,2}.
\ee
Recall that in Definition~\ref{defn:BB},
on a big bang timed-metric-space,
the time function is the null
distance from the big bang point,
\be
\tau_i(p)=\hat{d}_\tau(p,p_{BB})
\ee
so
$F$ is also time preserving.
Thus it is also a Lorentzian isometry satisfying $g_1=F^*g_2$ everywhere when the big bang timed-metric-spaces are associated with causally-null-compactifiable big bang space-times. 

The converse immediately follows the fact that if
$F:N_1\to N_2$ is distance and time preserving,
then $F(p_{BB,1})=p_{BB,2}$ and any distance
preserving map $\phi_2: N_2 \to Z$ defines a 
distance preserving map
$\phi_1=\phi_2\circ F: N_1\to Z$ such that
\be
\phi_1(N_1)=\phi_2(N_2)
\textrm{ and }
\phi_1(p_{BB,1})=\phi_2(p_{BB,2})
\ee
so the pointed $GH$ distance is $0$ and so
is the intrinsic big bang space-time distance.
\end{proof}

\begin{rmrk}\label{rmrk-second-proof-BB-GH}
A second proof of the definiteness of the BB-GH distance might possibly be found by showing that BB-GH distance is
zero iff the timed-Hausdorff distance between the spaces is zero.
It would be of interest to prove this.
\end{rmrk}

\subsection{{\bf The Future-Developed Intrinsic Space-Time Distances are Definite}}
\label{sect:definite-FD}

Recall Definition~\ref{defn:FD} of a causally-null-compactifiable future development.  Here we prove that $d_{FD-HH}$ of
Definition~\ref{defn:FD-HH},
between such space-times, is definite in the sense of Definition~\ref{defn:sd}.

\begin{thm}\label{thm:FD-HH}
Given a pair of 
future-developed timed-metric-spaces (whose cosmological time is the null distance from the
initial data as in
(\ref{eq:init-data-prop-2}),
\be \label{eq:FD-HH-here}
d_{FD-HH}\left(\big(\bar{N}_1,
M_1,\hat{d}_{1}\big),
\big(\bar{N}_2, M_2,\hat{d}_{2}\big)\right)=0
\ee
if and only if there is a null distance and cosmological time preserving bijection $F:N_1\to N_2$. 
In particular, 
two causally-null-compactifiable future-developed space-times have
\be 
d_{S-FD-HH}\Big((N_{1,+},g_1),(N_{2,+},g_2)\Big)=
d_{FD-HH}\left(\big(\bar{N}_1,
M_1,\hat{d}_{g_1}\big),
\big(\bar{N}_2, M_2,\hat{d}_{g_2}\big)\right)=0
\ee
iff there is a Lorentzian isometry, $F: N_1\to N_2$. 
\end{thm}

\begin{proof}
If (\ref{eq:FD-HH-here}) holds, then there is a sequence of
compact metric-spaces, $(Z_j,d_{Z_j})$
and 
distance preserving maps
\be
\varphi_{i,j}:\bar{N}_i \to Z_j
\ee
such that
\be \label{eq:FD-HH-2-here} d_H^{Z_j}\big(\varphi_{1,j}(\bar{N}_1),\varphi_{2,j}(\bar{N}_2)\big)
+ 
d_H^Z\big(\varphi_{1,j}(M_1),\varphi_{2,j}(M_2)\big)
<1/j.
\ee
We follow a common proof of the definiteness of the $GH$ distance.
By the definition of the Hausdorff distance, for each $j\in \mathbb N$
and for all $x\in M_1$, there exists
$y_x=F_j(x)\in M_2$ such that
\be\label{eq:Fj-for-FD-1}
d_{Z_j}(\varphi_{1,j}(x),\varphi_{2,j}(F_j(x)))<1/j.
\ee
and also for all $x\in N_1\setminus M_1$, there exists
$y_x=F_j(x)\in N_2$ such that
(\ref{eq:Fj-for-FD-1}) holds as well.
So we have a map 
\be
F_j:\bar{N}_1\to \bar{N}_2
\textrm{ such that } F_j(M_1)\subset M_2.
\ee
This map is $2/j$-almost distance preserving:
\be
|\hat{d}_{\tau_2}(F_j(x_1),F_j(x_2))
-\hat{d}_{\tau_1}(x_1,x_2)|<2/j,
\ee
because $\varphi_{i,j}$ are
distance preserving and
we can apply:
\be
|d_{Z_j}(p_1,p_2)-d_{Z_j}(q_1,q_2)|
\le 
d_{Z_j}(p_1,q_1)+d_{Z_j}(p_2,q_2)
\ee
to
\be
p_k=\varphi_{1,j}(F_j(x_k))
\textrm{ and }
q_k=\varphi_{2,j}(x_k).
\ee
We also have the reverse, switching $i=1$ and $i=2$, so that we have a a $2/j$ almost distance preserving map
\be
\tilde{F}_j:\bar{N}_2\to \bar{N}_1
\textrm{ such that } \tilde{F}_j(M_2)\subset M_1
\ee
defined by
\be\label{eq:Fj-for-FD-2}
d_{Z_j}(\varphi_{1,j}(\tilde{F}_j(y)),\varphi_{2,j}(y))<1/j.
\ee
Applying (\ref{eq:Fj-for-FD-1})
with $x=\tilde{F}_j(y)$ and
applying (\ref{eq:Fj-for-FD-2})
with $y={F}_j(x)$,
we see that these functions are
almost inverses of one another:
\begin{eqnarray*}
\hat{d}_{\tau_1}(\tilde{F}_j({F}_j(x)),x)
&=&
d_{Z_j}(
\varphi_{1,j}(\tilde{F}_j({F}_j(x)), 
\varphi_{1,j}(x) ))\\
&=&
d_{Z_j}(
\varphi_{1,j}(\tilde{F}_j({F}_j(x)), 
\varphi_{2,j}(F_j(x)) ))\\
&&\quad
+\,\,\,
d_{Z_j}(
\varphi_{2,j}({F}_j(x), 
\varphi_{1,j}(x) ))\quad <2/j\\
\hat{d}_{\tau_2}({F}_j(\tilde{F}_j(y)),y)
&\le &...\textrm{ as above }...\le 2/j .
\end{eqnarray*}
Since both $\bar{N}_i$ are compact, taking $j\to \infty$, by a theorem of Grove-Petersen
in \cite{Grove-Petersen}
we have
a subsequence which converges 
to a distance preserving pair of maps:
\be
F_\infty:\bar{N}_1\to \bar{N}_2
\textrm{ such that } F_\infty(M_1)\subset M_2.
\ee
and
\be
\tilde{F}_\infty:\bar{N}_2\to \bar{N}_1
\textrm{ such that } \tilde{F}_\infty(M_2)\subset M_1
\ee
which are inverses of each other because
\be
\hat{d}_{\tau_1}(\tilde{F}_\infty({F}_\infty(x)),x)=0
\textrm{ and }
\hat{d}_{\tau_2}({F}_\infty(\tilde{F}_\infty(y)),y)=0.
\ee
Furthermore, $F_\infty(M_1)\subset M_2$ and 
$\tilde{F}_\infty(M_2)=M_1$ so 
$F_\infty$ is a distance preserving 
bijection that maps $M_1$ onto $M_2$. 

Finally we show that $F$ is time preserving using the fact that
on a future-developed timed-metric-space, the cosmological time is the null distance from the
initial data as in
(\ref{eq:init-data-prop-2}).  Thus
\begin{eqnarray*}
\tau_2(F(p))&=&
\min_{y\in M_2} \hat{d}_{\tau_2}(F(p),y)\\
&=&
\min_{x\in M_1} \hat{d}_{\tau_2}(F(p),F(x))\\
&=&
\min_{x\in M_1} \hat{d}_{\tau_1}(p,x)\\
&=&
\tau_1(p)
\end{eqnarray*}
Once we know that $F$ is a distance and time preserving bijection,
it is a Lorentzian isometry because these are causally-null-compactifiable space-times.
\end{proof}

\begin{rmrk}\label{rmrk-second-proof-FD-HH}
A second proof of the definiteness of the  $FD-HH$ distance might possibly be found by showing that $FD-HH$ distance is
zero iff the intrinsic timed Hausdorff distance between the spaces is zero.
It would be of interest to prove this.
\end{rmrk}

\begin{rmrk}
Other intrinsic distances between future-developed space-times involving the intrinsic flat distance will be studied in
future work by Sakovich-Sormani \cite{SakSor-SIF}.   See also past work of Graf-Sormani with ideas towards compactness theorems in this direction \cite{Graf-Sormani}.  
\end{rmrk}

\subsection{{\bf The Strip-based Intrinsic Space-Time Distances may or may not be Definite}}
\label{sect:definite-strip}

In Section~\ref{sect:defns-strip} we introduced a few notions of intrinsic space-time distances using cosmic strips.   It is an open question whether these notions are definite or not:

\begin{rmrk} \label{rmrk:strip-definite}
Given a pair of space-times, $(N_1, g_1)$ and 
$(N_2,g_2)$ with time functions, $\tau_i:\Omega(R_i)\to {\mathbb R}$, such that their cosmic strips
$N_{i,s,t}=\tau_i^{-1}(s,t)$ are null compactifiable space-times, if any of the
intrinsic space-time distances defined in Section~\ref{sect:defns-strip} between them is $0$, then they have null distance preserving maps
\be
F_{s,t}: N_{1,s,t}\to N_{2,s,t}
\ee
for all $s<t$ in $[\tau_{min}, \tau_{max}]$
and in particular, 
\be
\min_{\bar{N}_1}\tau_1=
\min_{\bar{N}_2}\tau_2
\textrm{ and }
\max_{\bar{N}_1}\tau_1=
\max_{\bar{N}_2}\tau_2.
\ee
This is significantly stronger than level matching and neither counter example to the definiteness of level matching achieves this.   It is possible that one could somehow glue together the $F_{s,t}$ however this is not so easy as some $F_{s,t}$ might be time reversing and might involve an isometry in the spacelike direction.    
\end{rmrk}

\subsection{{\bf The Intrinsic Timed Hausdorff Distance is Definite}}
\label{sect:definite-tau-K}

Recall the Intrinsic Timed-Hausdorff Distance (and the other intrinsic timed distances defined using the timed-Fr\'echet map) in Section~\ref{sect:defns-tau-K} are defined between causally-null-compactifiable space-times and, more generally any pair of compact timed-metric-spaces. 
Here we prove the Intrinsic Timed-Hausdorff Distance defined as in Definition~\ref{defn:tau-K-dist} is definite:

\begin{thm}
\label{thm:definite-tau-K}
The intrinsic Timed-Hausdorff distance between causally-null-compactifiable space-times is definite
in the sense that
\be
d_{S-\tau-H}\Big((N_1,g_1),(N_2,g_2)\Big)=0
\ee
iff there is a time preserving Lorentzian isometry between them.
More generally, the intrinsic timed $H$ distance 
 between two compact timed-metric-spaces is definite
 in the sense that:
\be \label{eq:tau-K-definite-1}
d_{\tau-H}\Big((X,d_X,\tau_X),(Y,d_Y,\tau_Y)\Big)=
\inf\, d^{\ell^\infty}_H(
\kappa_{\tau_X,X}(X),
\kappa_{\tau_Y,Y}(Y))=0
\ee
iff there is a 
distance and time preserving bijection
\be\label{eq:tau-K-definite-2}
F:(X,d_X,\tau_X)\to (Y,d_Y,\tau_Y).
\ee
\end{thm}

Towards proving this theorem, we
prove Lemmas~\ref{lem:F-to-tau-K}
and~\ref{lem:tau-K-to-F} which hold more generally.   In Remark~\ref{rmrk:tau-K-definite} we discuss how these lemmas may possibly be applied in the future to 
prove
$d_{S-\tau-mm}$ and
$d_{S-\tau-F}$ are definite 
of Section~\ref{sect:defns-tau-K}
as well. 

\begin{lem}\label{lem:F-to-tau-K}
If $F:(X,d_X,\tau_X)\to (Y,d_Y,\tau_Y)$ is a distance and
time preserving map then there are
timed-Fr\'echet maps
\be
\kappa_{\tau_Y,Y}:Y\to \ell^\infty
\textrm{ and } \kappa_{\tau_X,X}:X\to \ell^\infty
\ee
such that 
\be\label{eq:tau-K-circ}
\kappa_{\tau_Y,Y}\circ F=\kappa_{\tau_X,X}.
\ee
\end{lem}

\begin{proof}
Given any countably dense set, 
\be
\{x_1,x_2, x_3,... \} \subset X,
\ee
we have a countably dense set
\be
\{y_1,y_2,y_3,....\}\subset Y
\textrm{ where } y_i=F(x_i).
\ee
Then the timed-Fr\'echet map
\begin{eqnarray*}
\kappa_{\tau_Y,Y}(F(x))
&=&(\tau_Y(F(x)), d_Y(y_1,F(x)),
 d_Y(y_2,F(x)),  d_Y(y_3,F(x)),...)\\
&=&(\tau_X(x), d_X(x_1,x),
d_X(x_2,x),d_X(x_2,x),....)
=\kappa_{\tau_X,X}(x).
\end{eqnarray*}
So (\ref{eq:tau-K-circ}) holds.
\end{proof}

The next lemma can be applied to any notion of distance between images of $X$ and $Y$ in $\ell^\infty$ such that zero distance between the images implies the two images are the same set.  We will apply this to the Hausdorff distance as part of the proof of Theorem~\ref{thm:definite-tau-K} but can also be applied
for Wasserstein and flat distances
as will be discussed in Remark~\ref{rmrk:tau-K-definite}.

\begin{lem}\label{lem:tau-K-to-F}
If there are
timed-Fr\'echet maps
\be
\kappa_{\tau_X,X}:X\to \ell^\infty
\textrm{ and } 
\kappa_{\tau_Y,Y}:Y\to \ell^\infty
\ee
with the same image,
\be\label{eq:same-image}
\kappa_{\tau_Y,Y}(Y)=
\kappa_{\tau_X,X}(X),
\ee
then there exists a
distance and time preserving bijection
$F: X \to Y$ such that
\be\label{eq:tau-K-circ-2}
\kappa_{\tau_Y,Y}\circ F=\kappa_{\tau_X,X}.
\ee
\end{lem}

\begin{proof}
By Proposition~\ref{prop:tau-K-dist-pres}
we know that
\be
\kappa_{\tau_Y,Y}:(Y,d_Y)\to 
(\kappa_{\tau_Y,Y}(Y), d_{\ell^\infty})
\subset (\ell^\infty,d_{\ell^\infty})
\ee
is distance preserving and thus one-to-one. Since $\kappa_{\tau_Y,Y}$ also maps
onto its image, 
it is a bijection and has an inverse:
\be
\kappa_{\tau_Y,Y}^{-1}:
(\kappa_{\tau_Y,Y}(Y), d_{\ell^\infty})
\to (Y,d_Y).
\ee
By (\ref{eq:same-image}),
\be
\kappa_{\tau_Y,Y}^{-1}: (\kappa_{\tau_X,X}(X),d_{\ell^\infty})\to (Y,d_Y).
\ee
Thus we can define the distance preserving bijection,
\be
F= \kappa_{\tau_Y,Y}^{-1}\circ
\kappa_{\tau_X,X}: (X,d_X)\to (Y,d_Y),
\ee
which satisfies \eqref{eq:tau-K-circ-2}.
To see that $F$ is time preserving,
note that if $y=F(x)$ then
\be
\kappa_{\tau_Y,Y}(y)=\kappa_{\tau_X,X}(x).
\ee
By the definition of the timed-Fr\'echet maps in
Definition~\ref{defn:tau-K}, we then have
\be
(\tau_Y(y),\kappa_Y(y))=(\tau_X(x),\kappa_X(x))
\ee
which implies $\tau_Y(y)=\tau_X(x)$.
\end{proof}


We now apply the above lemmas to prove Theorem~\ref{thm:definite-tau-K}:

\begin{proof} 
First we prove the easy direction,
starting with a distance and time preserving map, $F:(X,d_X,\tau_X)\to (Y,d_Y,\tau_Y)$,
we see by Lemma~\ref{lem:F-to-tau-K} that 
there are
timed-Fr\'echet maps with the
same image,
\be
\kappa_{\tau_X,X}(X)=\kappa_{\tau_Y,Y}(Y)\subset \ell^\infty
\ee
and so the Hausdorff distance
between these images is,
\be 
d^{\ell^\infty}_H(
\kappa_{\tau_X,X}(X),
\kappa_{\tau_Y,Y}(Y))=0.
\ee
Thus we have (\ref{eq:tau-K-definite-1}).

Conversely, suppose we have (\ref{eq:tau-K-definite-1}).
So there exists a sequence
of 
timed-Fr\'echet maps such that
\be \label{eq:ell-infty-j-tau}
d^{\ell^\infty}_H(
\kappa_{\tau_X,X,j}(X),
\kappa_{\tau_Y,Y,j}(Y))<1/j.
\ee  
Since $\ell^\infty$ is not compact,
we cannot apply Arzela-Ascoli to obtain a converging sequence of the timed-Fr\'echet maps and thus cannot directly apply  Lemma~\ref{lem:tau-K-to-F}.  Instead we imitate the proof that the Gromov-Hausdorff distance is definite.  

For each $j\in {\mathbb N}$, we claim that we can find a map $F_j:X\to Y$
which is 
almost time preserving:
\be\label{eq:almost-time}
|\tau_Y(F_j(x))-\tau_X(x)|<1/j.
\ee
and almost onto:
\be \label{eq:almost-onto}
\forall \, y\in Y \, \exists \, x\in X
\textrm{ s.t. }d_Y(F_j(x),y)<2/j
\ee
and almost distance preserving:
\be \label{eq:almost-dist}
|d_Y(F_j(x_1),F_j(x_2))-d_X(x_1,x_2)|<3/j.
\ee
Recall that by (\ref{eq:ell-infty-j-tau}) and the definition of
the Hausdorff distance, for all $x\in X$ there is $F_j(x)\in Y$ such that
\be \label{eq:ell-infty-j-tau-2}
d_{\ell^\infty}(
\kappa_{\tau_X,X,j}(x),
\kappa_{\tau_Y,Y,j}(F_j(x)))<1/j.
\ee  
Thus in particular the distance between the first components is also bounded which gives almost 
time preserving as in (\ref{eq:almost-time}). Applying the definition of
the Hausdorff distance again, for all $y\in Y$ there is $x\in X$ such that
\be 
d_{\ell^\infty}(
\kappa_{\tau_X,X,j}(x),
\kappa_{\tau_Y,Y,j}(y))<1/j.
\ee  
Combining this with (\ref{eq:ell-infty-j-tau-2}) and applying the
triangle inequality we have
\be
d_{\ell^\infty}(
\kappa_{\tau_Y,Y,j}(F(x)),
\kappa_{\tau_Y,Y,j}(y))<2/j
\ee  
which implies almost onto as in (\ref{eq:almost-onto})
because $\kappa_{\tau_Y,Y,j}$ is distance preserving.  Finally
we apply (\ref{eq:ell-infty-j-tau-2}) to a pair of points $x_1,x_2\in X$ and use the 
following inequality:
\be
|d(a_1,a_2)-d(b_1,b_2)|<d(a_1,b_1)+d(a_2,b_2)
\ee
with $d=d_{\ell^\infty}$ and with
\be
a_i=\kappa_{\tau_X,X,j}(x_i)
\textrm{ and }
b_i=\kappa_{\tau_Y,Y,j}(F(x_i))
\ee
to conclude almost distance
preserving as in (\ref{eq:almost-dist}) and complete our claim.

Since $(X,d_X)$ and $(Y,d_Y)$
are compact metric-spaces, 
and $F_j$ are almost onto
and almost distance preserving as in
(\ref{eq:almost-onto}) and (\ref{eq:almost-dist}),
a converging subsequence,
$F_{j_k}\to F$, exists such that
\be
\sup_{x\in X}|F_{j_k}(x)-F(x)|<1/k \to 0
\ee
and the limit function, $F$ is
an invertible distance preserving isometry.  The proof of this follows similar to the proof of the Arzela-Ascoli Theorem (even though the functions are not continuous).   It was known to Gromov and details are given in a paper by Grove-Petersen \cite{Grove-Petersen}.

Applying the fact that $\tau_Y$
is Lipschitz one and (\ref{eq:almost-time}) we see that
the limit function, $F$, is
also time preserving.
\end{proof}

\begin{rmrk}\label{rmrk:tau-K-definite}
To 
prove
$d_{S-\tau-mm}$ and
$d_{S-\tau-FF}$ are definite 
one would similarly apply the above lemmas and imitate the respective proofs that the metric measure and intrinsic flat distances are definite. Sakovich and Sormani will implement this for the $\mathcal{F}$ distance and leave the $mm$ distance to others.  
\end{rmrk}

\section{{\bf{Convergence of Space-Times}}}
\label{sect:conv}

Recall that once we have a notion of distance between space-times we have a notion of convergence for those space-times.
We may have smooth limits as in (\ref{eq:Sdistto}) or limits that are just
timed-metric-spaces as in (\ref{eq:distto}).   In this section we state natural conjectures related to our notions of convergence.

In Section~\ref{sect:relations}, we describe how we believe these various notions of convergence of space-times 
are related through a collection of conjectures. In Section~\ref{sect:prior} we review prior notions of convergence for space-times which might be related to our notions.

We present conjectures related to the
stability of Einstein's Equations with compact initial data in 
Section~\ref{sect:stability-compact}.
We present conjectures that might be
applied to justify the use of isotropic homogeneous FLRW space-times to describe the universe in Section~\ref{sect:FLRW}.   We refer to the book by Ringstrom for background \cite{Ringstrom-book-top}.

We discuss the definition of convergence
of asymptotically flat space-times in
Section~\ref{sect:conv-ex}. 
Using these notions, we state natural conjectures related to the Stability of Minkowski, Schwarzschild and Kerr  in Section~\ref{sect:stability}, and related to
the Final State Conjecture and Penrose's outline of the proof of the Penrose Inequality in Section~\ref{sect:Penrose}.
We recommend Giorgi's survey \cite{Giorgi-BAMS} for background.


\subsection{{\bf Relationships between the Notions of Convergence}}
\label{sect:relations}

Here we state some conjectures which might be interesting for doctoral students and postdocs.

\begin{conj}\label{conj:tK-conv-to-others}
If $(N_j,g_j)$ and $(N_\infty, g_\infty)$
are 
smooth causally-null-compactifiable space-times
converging in the intrinsic timed Hausdorff
sense,
\be 
d_{S-\tau-H}\Big((N_j,g_j),(N_\infty,g_\infty)\Big)\to 0,
\ee
then these spaces converge in the
timeless, sup level, and sup strip sense:
\begin{eqnarray}
&d^{tls}_{S-GH}
\Big((N_j,g_j),(N_\infty,g_\infty)\Big)&\to 0,\\
&d_{S-GH}^{L-sup}
\Big((N_j,g_j),(N_\infty,g_\infty)\Big)&\to 0,\\
&d_{S-GH}^{strip-sup}
\Big((N_j,g_j),(N_\infty,g_\infty)\Big)&\to 0.
\end{eqnarray}
More generally, if 
we have compact timed-metric-spaces
converging in the intrinsic timed-Hausdorff sense,
\be 
d_{\tau-H}\Big((\bar{N}_j,d_j,\tau_j),(\bar{N}_\infty,d_\infty,\tau_\infty)\Big)\to 0,
\ee
then we have timeless Gromov-Hausdorff convergence as well:
\begin{eqnarray}
&d^{tls}_{GH}\Big((\bar{N}_j,d_j,\tau_j),(\bar{N}_\infty,d_\infty,\tau_\infty)\Big)&\to 0,\\
&d_{GH}^{L-sup}\Big((\bar{N}_j,d_j,\tau_j),(\bar{N}_\infty,d_\infty,\tau_\infty)\Big)&\to 0,\\
&d_{GH}^{strip-sup}\Big((\bar{N}_j,d_j,\tau_j),(\bar{N}_\infty,d_\infty,\tau_\infty)\Big)&\to 0.
\end{eqnarray}
\end{conj}

\begin{conj}\label{conj:tK-conv-to-BB}
If, in addition to the hypotheses
in Conjecture~\ref{conj:tK-conv-to-others},
the causally-null-compactifiable
space-times, $(N_j,g_j)$ are big-bang
space-times as in Definition~\ref{defn:BB}, then the $\tau-H$ limit is a big bang space-time as well and we also have $BB-GH$
convergence. We expect counter examples exist when
replacing $dist=GH$ with $dist={\mathcal F}$ or
other notions of convergence that allow regions to
disappear.
\end{conj}

\begin{conj}\label{conj:tK-FD-to-BB}
One should be able to construct an example of
a sequence of causally-null-compactifiable 
future-developed space-times
satisfying the hypothesis in Conjecture~\ref{conj:tK-conv-to-others}
which ${S-\tau-H}$ 
converges to a compact big bang space-time
because their initial data sets shrink to a point.
\end{conj}

\begin{conj}\label{conj-BB-to-tK}
If $(N_j,g_j)$ and $(N_\infty, g_\infty)$
are smooth causally-null-compactifiable big bang space-times such that
\be 
d_{S-BB-GH}\Big((N_j,g_j),(N_\infty,g_\infty)\Big)\to 0
\ee
then 
\be 
d_{S-\tau-H}\Big((N_j,g_j),(N_\infty,g_\infty)\Big)\to 0.
\ee
More generally, if 
we have big bang timed-metric-spaces 
as in Definition~\ref{defn:BB}
with
\be 
d_{BB-GH}\Big((\bar{N}_j,d_j,\tau_j),(\bar{N}_\infty,d_\infty,\tau_\infty)\Big)\to 0,
\ee
then
\be 
d_{\tau-H}\Big((\bar{N}_j,d_j,\tau_j),(\bar{N}_\infty,d_\infty,\tau_\infty)\Big)\to 0.
\ee
\end{conj}

\begin{conj}\label{conj-FD-to-tK}
If $(N_j,g_j)$ and $(N_\infty, g_\infty)$
are smooth causally-null-compactifiable future-developed space-times such that
\be 
d_{FD-HH}\Big((N_j,g_j),(N_\infty,g_\infty)\Big)\to 0
\ee
then 
\be 
d_{S-\tau-H}\Big((N_j,g_j),(N_\infty,g_\infty)\Big)\to 0.
\ee
More generally, if 
we have future-developed timed-metric-spaces 
with
\be 
d_{FD-HH}\Big((\bar{N}_j,d_j,\tau_j),(\bar{N}_\infty,d_\infty,\tau_\infty)\Big)\to 0,
\ee
then
\be 
d_{\tau-H}\Big((\bar{N}_j,d_j,\tau_j),(\bar{N}_\infty,d_\infty,\tau_\infty)\Big)\to 0,
\ee
\end{conj}

\begin{rmrk}\label{rmrk:tK-conv-to-others}
We do not expect the above conjectures to hold for intrinsic flat convergence and volume preserving intrinsic flat convergence as stated above. 
By the Ambrosio-Kirchheim Slicing Theorem one expects an intrinsic timed-flat convergence defined using
an intrinsic timed flat distance
as in Remark~\ref{rmrk:tau-K-dist} to imply a level preserving convergence in the $\ell^p$ sense with $p=1$
and not sup convergence unless there is some stronger hypothesis on the curvature of the space.   Regions may disappear under intrinsic flat convergence, so we should be able to construct examples where big bang space-times converge to space-times without big bangs and future-developed space-times converge to space-times which are not future-developed.
This is being investigated by Sakovich-Sormani\cite{SakSor-SIF}.  The same issue
may happen with metric measure convergence which
we leave for others to explore.
\end{rmrk}

\subsection{{\bf Prior Weak Notions of Convergence of Space-Times}}
\label{sect:prior}

In this subsection we very briefly describe other notions of distances and convergence for sequences of smooth or less regular space-times and suggest
further work relating these ideas with ours so that they can be used together.

\begin{rmrk}\label{rmrk:smooth}
The most commonly applied prior notion of distance between two smooth Lorentzian manifolds, $(N_1,g_1)$ and
$(N_2,g_2)$ involves having a diffeomorphism, $\Psi:N_1\to N_2$ and then studying the various norms on $|\Psi_*g_1-g_2|$ either measured with respect to a background metric or with respect to one of the two metrics.   These can be sup norms or more commonly $C^k$ or $H^{1,p}$ norms and sometimes the manifolds considered are regions inside larger Lorentzian manifolds of interest. These notions
are used to define convergence of smooth Lorentzian manifolds to limit spaces which are also smooth or at
least are manifolds with $C^k$ or $H^{1,p}$ metric tensors.
They
have been applied to prove various Lorentzian stability
theorems.  See, for example, the beautiful surveys of Dafermos  \cite{Dafermos-ICM} and
Giorgi \cite{Giorgi-BAMS} and the many papers mentioned 
Section~\ref{sect:conv}.
\end{rmrk}

\begin{rmrk}\label{rmrk:timeless-smooth}
Allen and Allen-Burtscher relate the timeless notions of convergence
defined using the
null distance that we reviewed in Section~\ref{sect:defns-tl}
with these various smoother notions in \cite{Allen-Burtscher-22}\cite{Allen-Null}.   It would be interesting to extend their work to other notions of intrinsic space-time convergence introduced in this paper. 
\end{rmrk}

\begin{rmrk}\label{rmrk:KS-22}
Recall the
Lorentzian length spaces 
mentioned in Remark~\ref{rmrk:Lor-length}.
Kunzinger-Steinbauer have 
converted these spaces to timed-metric-spaces using the null distance
and studied the intrinsic
timeless $GH$ distance, $d^{tls}_{S-GH}$ between them 
in 
\cite{Kunzinger-Steinbauer-22}.
It would be interesting if they could study the notion of 
intrinsic big bang distance, 
$d_{BB-GH}$, and intrinsic timed distance, $d_{\tau-H}$,  
between the timed-metric-spaces associated with their Lorentzian length spaces.  
\end{rmrk}

There are other approaches to the
convergence of space-times which do not involve the null distance or metric spaces at all.  

\begin{rmrk}\label{rmrk:Noldus}
Noldus introduced the first notion of a ``Gromov-Hausdorff" distance between space-times which is not definite but has been explored in a few papers by Noldus and by Bombelli-Noldus in
\cite{Noldus-Topology} \cite{Noldus} \cite{Noldus-limit}\cite{Bombelli-Noldus}
\end{rmrk}

\begin{rmrk}\label{rmrk:Minguzzi-Suhr}
In 2022, Minguzzi-Suhr have introduced an intriguing notion of 
``Lorentzian Gromov-Hausdorff" 
convergence of
``Lorentzian metric-spaces" 
 \cite{Minguzzi-Suhr-24}.
 Recall that ``Lorentzian metric spaces" are not metric spaces but causets that are topological spaces with a ``Lorentzian distance function" that satisfies some natural properties similar to the Lorentzian distance on a Lorentzian manifold.  Their Lorentzian Gromov-Hausdorff distance is defined by adapting Gromov's definition for $GH$ distance using relations closely related to the $\epsilon$-almost isometries we reviewed in Section~\ref{sect:back-GH}. 
They prove their Lorentzian $GH$ distance is definite in the sense that their Lorentzian $GH$ distance between to Lorentzian metric spaces is $0$ iff there is a ``Lorentzian distance function"
preserving bijection between them.
It is not mentioned in \cite{Minguzzi-Suhr-24} if this implies there is a Lorentzian isometry between the spaces when the pair of Lorentzian metric spaces are Lorentzian manifolds.  
\end{rmrk}

\begin{rmrk}\label{rmrk-Muller}
In a very recently posted arxiv paper,  M\"uller has defined three $GH$ distances,
$d^+_{GH}$, $d^-_{GH}$, and
$d^x_{GH}$ between pairs of partially ordered measure spaces
in
\cite{Mueller-GH}.   We have not yet read this deep paper.
\end{rmrk}

\begin{rmrk}\label{rmrk:CMM}
Recall the metric measure approach to studying space-times appearing in the papers cited in Remark~\ref{rmrk:measure-space-time}.
It is natural to test their notions using $C^{1,loc}$ convergence of Lorentzian metric tensors and $C^0$ convergence of  measures.   See, for example,
the recently posted \cite{Cavalletti-Manini-Mondino-Optimal} by
Cavalletti-Manini-Mondino.   Mondino plans to study how their notion interacts with the notions of intrinsic metric measure  space-time convergence suggested within this paper.
\end{rmrk}

\begin{rmrk}\label{rmrk:both}
It would be very interesting to formulate conjectures and prove theorems relating our notions of space-time convergence with other notions of convergence for Lorentzian manifolds and Lorentzian length spaces.  Naturally one would need to restrict to a class of space-times where both notions can be defined.
\end{rmrk}

\subsection{{\bf Stability with Compact Initial Data Sets}} 
\label{sect:stability-compact}  

Given a sequence of compact initial data sets, $(M_j,h_j,k_j)$, satisfying the Einstein constraint equations, 
Choquet-Bruhat proved there are space-times, $(N_j,g_j)$,
satisfying the Einstein vacuum equations such that $M_j$ lies as a smooth Cauchy surface in $N_j$ with restricted metric tensor $h_j$ and second fundamental form $k_j$
\cite{Bruhat-52}. Similar solutions have been found satisfying Einstein's Equations for a variety of matter models. 
However, even if the initial data sets converge $C^k$ smoothly, $(M_j,h_j,k_j)\to (M_\infty,h_\infty,k_\infty)$,
the space-time solutions, $(N_j,g_j)$
need not converge smoothly to
$(N_\infty,g_\infty)$.  The standard approach is to use Sobolev spaces to
describe the convergence of the space-times, and thus also Sobolev spaces to
describe the convergence of the initial data.   See, for example, Ringstr\"om's text \cite{Ringstrom-Cauchy-text}. 

\begin{rmrk}\label{rmrk:low-reg-compact}
For lower regularity discussion of the Einstein constraint equations on initial data, lower regularity convergence of solutions to Einstein's Equations and examples demonstrating instability, see the work of Berger,
Choquet-Bruhat, 
Christodoulou,
Chrusciel, Ettinger, Isenberg,
Layne, LeFloch,
Lindblad,  Maxwell, Moncrief, Klainerman, Ringstr\"om, Rodnianski, Smith, Smulevici, Szeftel, Tataru,
 \cite{Berger-Isenberg-Layne}
\cite{Chrusciel-Maximal}
\cite{Christodoulou-BV}
\cite{Ettinger-Lindblad-sharp}
\cite{Klainerman-Rodnianski-Rough}
\cite{KlRodSz} 
\cite{LeFloch-Smulevici-JEMS}
\cite{Maxwell-Rough}
\cite{Moncrief-Gowdy}
\cite{Ringstrom-instability}
\cite{Smith-Tataro} 
and others in
Ringstr\"om's survey \cite{Ringstrom-origins} and books \cite{Ringstrom-Cauchy-text} \cite{Ringstrom-book-top}.
These results all fix a background space and study variations in the metric tensor and thus implicitly involve a diffeomorphism between the spaces being studied.  
\end{rmrk}

We would also like to allow for collapsing and changing topology.  So we state
the following stability conjecture: 

\begin{conj}\label{conj:Einstein-stable}
Suppose $(M_j,h_j,k_j)$ are a sequence of compact initial data sets that converge to $(M_\infty,h_\infty, 0)$ in the sense that
$(M_j,h_j)\to (M_\infty,h_\infty)$ 
in the $GH$, $mm$, or $\mathcal{F}$ sense, and $||k_j||\to 0$, 
and  $(N'_{j,+},g_j)$ are
future developments 
of the Einstein vacuum equations
for $(M_j,h_j, k_j)$ with the interiors of black holes removed as in Remark~\ref{rmrk:cmpct-remove-holes}.
Then
$(N'_{j,+},g_j)$ converge to 
the future development, $(N'_{\infty,+}, g_\infty)$, of $(M_\infty,h_\infty,0)$,
in some intrinsic space-time sense. 
 One might use one of the future developed intrinsic space-time distances
of Section~\ref{sect:defns-FD}
or one of the intrinsic timed space-time distances
defined in 
Section~\ref{sect:defns-tau-K}.  Certain matter models might also be considered.
One would need to first verify 
$(N'_{j,+},g_j)$ and $(N'_{\infty,+},g_j)$ are
causally-null-compactifiable future developments as in Definition~\ref{defn:FD}
using a theorem like Theorem~\ref{thm:FD}.
\end{conj}

\begin{rmrk}\label{rmrk:easier-Einstein-stable}
It might be easier to study Conjecture~\ref{conj:Einstein-stable} in special cases where we assume $(M_\infty,g_\infty)$ is a flat torus or a standard sphere or hyperbolic space
assuming various kinds of symmetries.  It is even of interest to better understand how the cosmological time behaves on such sequences of space-times.  See Remark~\ref{rmrk:BB-stab} for
an application of this special case. 
\end{rmrk}

\begin{rmrk} \label{rmrk:Einstein-stable}
In Conjecture~\ref{conj:Einstein-stable},
one might also consider various norms for
$||k_j|| \to 0$.  Graf and Sormani have completed a first paper considering various $L^p$ bounds on the $k_j$,
and obtaining estimates which would help prove $d_{S-FD-FF}$ convergence in \cite{Graf-Sormani}.  Further work by Graf, Sakovich, Sormani applying this is to appear.  
\end{rmrk}

\begin{rmrk}\label{rmrk:ALMM}
Anderson, Lott, Moncrief and Mondal have also explored
lower regularity approaches to
these stability questions:
including changing topology
and collapsing in
 \cite{Anderson-Asym}
 \cite{Lott-Collapsing-18}
\cite{Lott-Backreaction-18}
 \cite{Moncrief-Mondal-19}.
This is even a concern with a fixed future development, $(N_+,g)$ from a fixed compact initial data set, $(M,h,k)$,
when studying cosmic strips,
\be
N_j=\tau_g^{-1}(t_j, t_j+t_0)\subset N_+
\ee
as $t_j \to \infty$
as in Anderson's paper \cite{Anderson-Asym}
or other regions $N_j\subset N_+$
which advance to infinity like those studied by Fischer-Moncrief in \cite{Fischer-Moncrief-reduced}.
It would be interesting to see if their results and conjectures can be viewed in a new way using our notions of intrinsic distances between causally null compactifiable space-times.
\end{rmrk}

\begin{rmrk}\label{rmrk:tau-compact-stable}
As a preliminary step towards Conjecture~\ref{conj:Einstein-stable}, it would be interesting to explore the behavior of the cosmological time functions in the various lower regularity settings mentioned in Remark~\ref{rmrk:low-reg-compact}
and Remark~\ref{rmrk:ALMM}.
See also Remark~\ref{rmrk:tau-FLRW}.
One should be able to 
extend Theorem~\ref{thm:FD} 
as discussed in Remark~\ref{rmrk:low-reg-init-data} above to prove cosmic strips in these lower regularity space-times with compact initial data
are causally-null-compactifiable future-developed space-times
as in Definition~\ref{defn:FD}.
\end{rmrk}

\subsection{{\bf Approximating FLRW Space-Times}}
\label{sect:FLRW}

When physicists model the universe using Friedmann-LeMaitre-Robertson-Walker (FLRW) space-times, they are assuming that the universe can be
modeled by 
a homogeneous isotropic space-time with spacelike Cauchy surfaces that have constant sectional curvature.  At the same time, it is well known that the universe is only almost homogeneous and almost isotropic. 

\begin{rmrk}\label{rmrk:evol}
The evolution of space-times that are close to FLRW space-times has been studied by Had\v{z}i\'c,
 LeFloch, Moncrief, Reiris, Ringstr\"om, Rodnianski, Smulevici, Speck 
 in
\cite{LeFloch-Smulevici-JEMS}
\cite{Mahir-Speck-FLRW}
\cite{Moncrief-Gowdy}
\cite{Reiris-FLRW}
\cite{Rod-Speck-FLRW}
and others surveyed in book by Ringstr\"om
\cite{Ringstrom-book-top}. 
Moncrief, Hao, Chrusciel, Berger, Isenberg, and Ringström all have work on the stability of FLRW space-times which is surveyed in Chapter XVI of 
Choquet-Bruhat's Oxford Monograph \cite{Choquet-Bruhat-OxMono}.
Most of these studies assume the space-times are diffeomorphic to an FLRW space-time.
\end{rmrk}

\begin{rmrk}\label{rmrk:anisotropic}
In fact, the universe is a space-time full of black holes each of which adds to its topological complexity.  See for example the work of physicists Dyer-Roeder,
Gabrielli-Labini, Kantowski, and
Krasinski
 \cite{Dyer-Roeder}
   \cite{Gabrielli-Labini}
   \cite{Kantowski}
\cite{Krasinski-text}
and physics survey on Cosmological Perturbation Theory by Bernardeau, Colombi, Gaztanaga, and Scoccimarro \cite{bernardeau2002large} trying to explore this anisotropic universe.
\end{rmrk}

We can rephrase the idea that an anisotropic universe evolves ``close" to an FLRW space-time as a conjecture using our notions of intrinsic space-time distances because our notions allow  settings where there are no diffeomorphisms between the spaces that we are comparing. 
 
\begin{conj}\label{conj:almost-isotropic}
An almost isotropic and almost homogeneous null compactifiable space-time with a big bang as in Section~\ref{sect:space-times-big-bang} (with appropriate natural assumptions) is close in 
one of our intrinsic space-time distances to an FLRW space-time.  We might estimate using one of the big-bang intrinsic space-time distances
of Section~\ref{sect:defns-BB}
or the intrinsic timed distances
defined in 
Section~\ref{sect:defns-tau-K}.
\end{conj}

\begin{rmrk}\label{rmrk:GAFA} 
If we assume a space-time is a warped product
then this becomes a Riemannian question and we need only understand the level sets of time.
However, even when a Riemannian manifold is almost isotropic, it need not be smoothly close to Riemannian manifold of constant sectional curvature.  See for example the paper of
Sormani \cite{Sormani-cosmos} which reviews counter examples and proves that, under additional hypotheses, such a Riemannian manifold is at least $GH$ close to a manifold of constant sectional curvature. Those additional hypotheses are strong but necessary to achieve $GH$ convergence.  With more natural weaker assumptions,
the Riemannian manifold might have thin deep wells sinilar to those in Lee-Sormani's \cite{LeeSormani1}.  In the weaker setting, Sormani has conjectures that these almost isotropic Riemannian manifolds are close in the $\mathcal{F}$ sense but this has not yet been proven.  Let Sormani know if you would like to work on this question.
\end{rmrk}

\begin{rmrk}\label{rmrk:BB-stab}
Note that given an FLRW big bang space-time, $(N,g)$, with a proper cosmological time function,
$\tau:N\to [0,\infty)$, the cosmic strips,
$(N_{s,t},g)$ are null compactifiable and are
future developments of level sets, $(M=\tau^{-1}(s),h)$.   So Conjecture~\ref{conj:Einstein-stable} can also be seen as studying the stability of FLRW space-times.  
\end{rmrk}

\begin{rmrk}\label{rmrk:tau-FLRW}
As a preliminary step
towards Conjecture~\ref{conj:almost-isotropic}, it would be interesting to explore the behavior of the cosmological time function in the various settings mentioned in Remark~\ref{rmrk:evol}
and Remark~\ref{rmrk:anisotropic}
where 
the stability of the FLRW space-times has already been explored.
\end{rmrk}

\subsection{{\bf Convergence with Exhaustions}}
\label{sect:conv-ex}

It is also of interest to study the convergence of space-times that are asymptotically flat and 
are thus not null compactifiable.

When studying sequences of complete noncompact Riemannian manifolds, $(M_j, h_j)$, it is standard to choose points  $p_j\in M_j$, exhaust $(M_j, h_j)$ with closed Riemannian balls $\bar{B}_j(p_j,R)$ about the points $p_j$, and then study the Gromov-Hausdorff convergence of these balls. Specifically, we say that the associated metric spaces $(M_j,d_j,p_j)$ converge to $(M_\infty,d_\infty,p_\infty)$  in the pointed Gromov-Hausdorff sense,  
\be
(M_j,d_j,p_j)\ptGHto (M_\infty,d_\infty,p_\infty),
\ee
if for any $R>0$ we have
\be
d_{pt-GH}((\bar{B}_j(p_j,R),d_j),(\bar{B}_\infty(p_\infty,R),d_\infty)\to 0.
\ee
This works because closed balls in a Riemannian manifold, $\bar{B}(p,R)\subset (M,g)$,
endowed with the restricted Riemannian distance
are compact and exhaust the Riemannian manifold,
\be
\bigcup_{R\to \infty} \bar{B}(p,R)\supset M.
\ee
See Gromov \cite{Gromov-metric}.  

In our setting we cannot just use balls because we need the causal structure to be preserved.   So we need to exhaust the space-time with causally-null-compactifiable regions and use those regions rather than balls to define an intrinsic space-time convergence.

\begin{rmrk} \label{rmrk:exhausted-conv}
In Remark~\ref{rmrk:pt-exhaustion}
and Remark~\ref{rmrk:CS-exhaustion}
we suggested two possible
ways one might canonically exhaust 
a sequence of asymptotically flat space-times, $(N_j,g_j)$, by causally-null-compactifiable regions, $\Omega_j(R_i)$, so that
\be
N_{j,0,\tau_{max}}=\bigcup_{i=1}^\infty \Omega_j(R_i).
\ee
We could then define intrinsic space-time convergence (subject to these exhaustions) by requiring the appropriate intrinsic space-time   convergence of the regions in the exhaustion:
\be
\forall \, i\in {\mathbb N},
\,\,
d_{S-dist}((\Omega_j(R_i),g_j),
(\Omega_\infty(R_i),g_\infty))\to 0.
\ee
It would be interesting to explore the
properties of such a convergence for various
$d_{S-dist}$ we have proposed within.
\end{rmrk}

\subsection{{\bf Stability of Minkowski, Schwarzschild, and Kerr Space-Times}}
\label{sect:stability}

A key research theme 
in Mathematical General Relativity is the problem of stability for Minkowski, Schwarzschild, and Kerr  as solutions of the Einstein vacuum equations.  See, for example, the beautiful surveys of Dafermos  \cite{Dafermos-ICM} and
Giorgi \cite{Giorgi-BAMS} for an overview of this problem.  See work of
Christodoulou, Klainerman, Lindblad, Rodniansky, and
Bieri on the stability of
Minkowski space 
\cite{CK-Minkowski}
\cite{Lindblad-Rodniaski}\cite{Bieri-JDG}.   
See also the work 
on the stability of Schwarzschild 
by Dafermos, Holzegel, Klainerman, Rodnianski, Szeftel and Taylor \cite{DHR}\cite{DHRT-Sch}
\cite{KS-Global-Schwarzschild}
and on stability of Kerr by 
Giorgi, Klainerman, Szeftel,
\cite{KS-small}
\cite{GKS-22}.
This research area is very active and new results appear regularly.

Let us consider the question of stability of Minkowski space, which we may formulate
vaguely as follows: if initial data, $(M_j,h_j,k_j)$, converges to
Euclidean space, 
$({\mathbb E}^3, h_{{\mathbb E}^3},0)$,  in a certain sense, do
the exteriors of the maximal future developments with respect to the 
Einstein vacuum equations,
$(N'_{j,+}, g_j)$,
converge in some sense to future Minkowski space,
$({\mathbb R}^{1,3}_+, g_{Mink})$?  

The notion of convergence of the initial data in the aforementioned works is stronger than the notions of convergence considered in the current paper. For example, when considering the stability of Minkowski space it is assumed that $(M_j,h_j,k_j)$ and 
$({\mathbb E}^3, g_{{\mathbb E}^3},0)$ are diffeomorphic and the convergence is measured by comparing the metric tensors.  However, it
is important to verify stability even when the initial data for a sequence of space-times is not diffeomorphic to the initial data for the limit, perhaps due to the existence of small black holes or thin deep gravity wells.  See Remark~\ref{rmrk:PMT-Stab}.

\begin{rmrk}\label{rmrk:PMT-Stab} 
Recall that Schoen and Yau \cite{Schoen-Yau-positive-mass} proved that an asymptotically flat
time symmetric initial data set whose mass is  zero must be isometric to Euclidean space.
Lee and Sormani prove in \cite{LeeSormani1} that
sequences of uniformly asymptotically flat time symmetric 
spherically-symmetric initial data sets with masses converging to zero converge in the $\mathcal{F}$ sense to Euclidean space but have counter examples to smooth or even Gromov-Hausdorff convergence.
This geometric stability result 
or almost rigidity of the positive mass theorem, has subsequently been extended, in particular removing the assumption of spherical symmetry, to a variety of cases.  See for example related work by 
Alaee, Allen, Basso, Bryden, Cabrera Pacheco, Creuz, Del Nin, Dong, Elefterios, Huang, Huisken, Ilmanen, Jauregui, Khuri, Lee,
LeFloch,
McCormick, Perales, Song, to mention a few
\cite{Huisken-Ilmanen}
\cite{Sormani-Stavrov}
\cite{Bryden-stability}
\cite{Bryden-Khuri-Sormani}
\cite{HLS} 
\cite{Allen-IMCF-CAG}
\cite{LeFloch-Sormani-1}
\cite{LeeSormani1}
\cite{LeeSormani2} 
\cite{Allen-Perales}
\cite{Huang-Lee-Perales}
\cite{DelNin-Perales-Rigidity}
\cite{Basso-Creutz-Elefterios}
\cite{Dong-Song-Stability}
\cite{A-CP-Mc}
\cite{Jauregui-Lee-ADM-F}.
Note that the precise definition of uniform asymptotic flatness varies in these papers.
See also Bryden, Khuri, and Sormani's work \cite{Bryden-Khuri-Sormani} where intrinsic flat convergence of a sequence of initial data sets as above is obtained without assuming time symmetry. See also Graf-Sormani \cite{Graf-Sormani}.
The asymptotically hyperbolic case has also been studied in works by Cabrera Pacheco, Graf, Sakovich, Sormani and Perales \cite{Sak-Sor-AH}, \cite{Pacheco-graphs}, \cite{PGP}.
\end{rmrk}

We believe that similar techniques might be applied to study sequences of Lorentzian manifolds.
We make the following conjecture, which is deliberately vague with regards to the uniform 
asymptotic flatness, the meaning of exterior region, and the notion of convergence for the initial data sets: 

\begin{conj}\label{conj:Stab-Mink}
Suppose one has a sequence of uniformly asymptotically flat time-symmetric
initial data sets, converging to the Euclidean space in some weak sense:
\be
(M_j,h_j,\kappa_j=0)\to ({\mathbb E}^3, g_{\mathbb E^3},0).
\ee
Let 
$(N'_{j,+},g_j)$ be the  exterior region of the
maximal future development 
with respect to the Einstein vacuum equations 
of $(M_j,h_j,\kappa_j)$,
with the cosmological time function $\tau_j:N'_{j,+}\to [0,\infty)$
as in Remark~\ref{rmrk:exterior-region}.
Then for any $\tau_0>0$,
the cosmic strips, 
\be
N'_{j,0,\tau_0}=\tau_j^{-1}(0,\tau_0)\subset N'_{j,+},
\ee
converge to a cosmic strip
$(N^{Mink}_{0,\tau_0},g_{Mink})$
where
 $(N^{Mink}_+,g_{Mink})$ is the future  region of Minkowski space-time, in an intrinsic future developed  sense
using a Cederbaum-Sakovich STCMC exhaustion
by causally-null-compactifiable sub-space-times as in Remark~\ref{rmrk:CS-exhaustion}
and Remark~\ref{rmrk:exhausted-conv}.   We conjecture that, among the
 various intrinsic future developed distances between space-times 
 defined in Section~\ref{sect:defns-FD}, the intrinsic 
  $S-FD-FF$ or
  $S-FD-mm$ convergence may work better
  than the intrinsic $S-FD-HH$ convergence.   
\end{conj}

\begin{rmrk}\label{rmrk-SakSorSch}
Sakovich and Sormani have work in progress showing that
if $(N'_{j,+},g_j)$ is a sequence of future exterior Schwarzschild space-times as in Example~\ref{ex:ext-Sch} with mass converging to $0$,
\be
\mass_{ADM}(N'_{j,+},g_j)\to 0,
\ee
then $(N'_{j,+},g_j)$
converges to future Minkowski space with respect to various notions of space-time intrinsic flat convergence, using 
an appropriately defined exhaustion.  Perhaps one might explore metric measure convergence in this setting as well. Note that in this particular situation the exterior region is naturally defined by the event horizon and we can
find an exhaustion of this exterior region by causally null compactifiable sub space-times
as discussed in Section~\ref{sect:space-times-Schwarzschild} and Section~\ref{sect:conv-ex}. 
\end{rmrk}

\begin{rmrk}\label{rmrk:sph-sym}  To study more general spherically symmetric space-times would be more challenging because the cosmological time function has not yet been studied in those settings.  It would be worthwhile to study the cosmological time function in the spherically symmetric space-times studied by Christodoulou in \cite{Christodoulou-1999} satisfying Einstein vacuum equations. 
\end{rmrk}

\begin{rmrk}\label{rmrk:stable-generators}
It would be of interest to study the stability of cosmological time functions and their generators in the setting of 
Christodoulou-Klainerman's 
\cite{CK-Minkowski}
where the metric tensors have already been proven to be stable. 
\end{rmrk}

The problems suggested in Remark~\ref{rmrk:sph-sym} and Remark~\ref{rmrk:stable-generators}
should be quite manageable, while 
Conjecture~\ref{conj:Stab-Mink}
would be far more challenging. Let us make a few further remarks on the full conjecture.  

\begin{rmrk}\label{rmrk:limit-diff-matter-model}  
There has also been interesting work showing that sequences of manifolds satisfying Einstein vacuum equations can converge to a limit manifold satisfying a different matter model.
See the work of Huneau-Luk \cite{Huneau-Luk}
and Touati \cite{Touati-optics}. Stavrov Allen, Benko, Benjamin, and McDermott have studied 
convergence of Brill-Lindquist geometrostatic initial data towards dust initial data
\cite{Benko-Stavrov} and
\cite{Benjamin-McDermott-Stavrov}
using intrinsic flat convergence.
\end{rmrk}

\begin{remark}\label{rmrk:short-pulse}
Even when the initial data in the sequence of 
Conjecture~\ref{conj:Stab-Mink} 
are diffeomorphic to Euclidean space, it
is possible that their future developments will develop black holes and thus the 
cosmic strips will not be diffeomorphic to cosmic strips in Minkowski space. See for instance the work of Kehle and Unger \cite{Kehle-Unger-Vacuum} where examples of vacuum gravitational collapse to very slowly rotating Kerr black holes are constructed, using Christodoulou's short pulse method 
\cite{Christodoulou-formation-black} and a variant of the characteristic gluing of Aretakis, Chimek and Rodnianski \cite{Aretakis-Czimek-Rodnianski}. We believe that the notions of $S-FD-mm$ or $S-FD-FF$ convergence  developed in our paper can be useful for studying such gravitational collapse phenomena. As a  first step, it would be of interest to study the behavior of cosmological times in the examples of Kehle and Unger \cite{Kehle-Unger-Vacuum}. 
\end{remark}

\begin{rmrk}\label{rmrk:stab}
It is also natural to consider the stability of Schwarzschild and,
more generally of Kerr, as in work of 
Dafermos, Holzegel, Rodnianski, and Taylor \cite{DHR}\cite{DHRT-Sch}
and 
Giorgi, Klainerman, and Szeftel,
\cite{KS-small}
\cite{KS-Global-Schwarzschild}
\cite{GKS-22}. 
While causal structure and location of horizons has been explored in these papers, it would also be of interest to understand the behavior of the cosmological time in these settings.
Subsequently, one could apply these results to compare the null distance
in the exterior regions of black hole space-times with that of the exterior of Kerr or Schwarzschild space-times.
\end{rmrk}

\begin{rmrk}\label{rmrk:Eikonal}
Throughout this paper we have primarily studied the cosmological time function through the use of its generators.  Finding the generating geodesics and their Lorentzian lengths is just one method for computing the cosmological time function.   It is also possible that one might study the cosmological time function using the fact that it is a viscocity solution to the Eikonal equation.
See the papers of Cui-Jin \cite{Cui-Jin}
and Zhu-Wu-Cui \cite{Zhu-Wu-Cui}.  
\end{rmrk}

\subsection{{\bf Advancing Time Towards Infinity}}
\label{sect:Penrose}

Given reasonable initial data, due to gravitational
radiation, one expects that a general asymptotically flat solution of the 
Einstein vacuum equations will asymptotically  settle down into a stationary regime. A similar scenario is
expected to hold true in the presence of matter, as the matter fields shall be ``swallowed'' by the black hole in the process. This conjecture, called the {\bf Final State Conjecture},  if true, characterizes all possible asymptotic states of an isolated gravitational system: one expects it to evolve to a finite number of Kerr black holes moving away from each other.   See, for example, the surveys by Giorgi \cite{Giorgi-BAMS} and by Ionescu and Klainerman \cite{IonescuKlainerman} for more details.

If the system is an asymptotically flat galaxy with multiple black holes, one cannot expect convergence to this final state to be smooth convergence.   In fact, with deep gravity wells around stars falling into the central black hole, we would not even expect a convergence in any suitably defined Gromov-Hausdorff sense. See the left side of Figure~\ref{fig:final-state}. 

Here we restate the Final State Conjecture 
using our cosmic strips and pointed exhaustions by
causally null compactifiable sub-space-times defined in Sections~\ref{sect:space-times-strip}-\ref{sect:space-times-sub} and depicted in the center of
Figure~\ref{fig:final-state} as follows: 

\begin{conj}\label{conj:final-state}
Suppose $N'_+$ is the exterior
region outside of the event horizon in
an asymptotically flat future
developed space-time, $(N,g)$, with at
least one black hole satisfying Einstein's equations.
 Suppose $\tau: N'_+\to (0,\infty)$
is its cosmological time function and suppose $C:(0,\infty)\to N'$ is a causal curve
lying near the horizon boundary, $\partial N'$,
parameterized by cosmological time,
$
\tau(C(t))=t$ for any
$t\in (0,\infty).
$
Consider a sequence of well chosen
space-like surfaces, $M_j \subset N'$,
diverging to the future
as in Remark~\ref{rmrk:well-chosen}
where 
$
C(t_j)\subset \partial M_j \textrm{ and }t_j \to \infty.
$
Let $N'_{j,+}=J^+(M_j)\subset N'$ be the exterior future development of $M_j$, 
let $\tau_j:N'_{j,+}\to [0,\infty)$ be
the cosmological time of $N'_{j,+}$.
For any $\tau_0>0$,
the cosmic strips,
\be
N'_{j,0,\tau_0}=\tau_j^{-1}(0,\tau_0)\subset N'_{j,+}\subset N',
\ee
converge in some intrinsic space-time sense
(using a pointed exhaustion based at $C(t_j)$
by causally-null-compactifiable sub-space-times, $\Omega_j(R_i)\subset N'_{j,0,\tau_0}$
as in Remark~\ref{rmrk:pt-exhaustion}
and Remark~\ref{rmrk:exhausted-conv})
to a cosmic strip,
$(N'_{Kerr,0,\tau_0},g_{Kerr})$ in
 the exterior of future Kerr space-time,  $(N'_{Kerr,+},g_{Kerr})$ as in Remark~\ref{rmrk:Kerr-exterior}.  
 We conjecture that among the
 various intrinsic future developed distances between space-times 
 defined in Section~\ref{sect:defns-FD}, an intrinsic
  $S-FD-FF$ or
  $S-FD-mm$ convergence may work better
  than $S-FD-HH$.  
\end{conj}

\begin{figure}[h]
   \centering
   \includegraphics[height=1.4in]{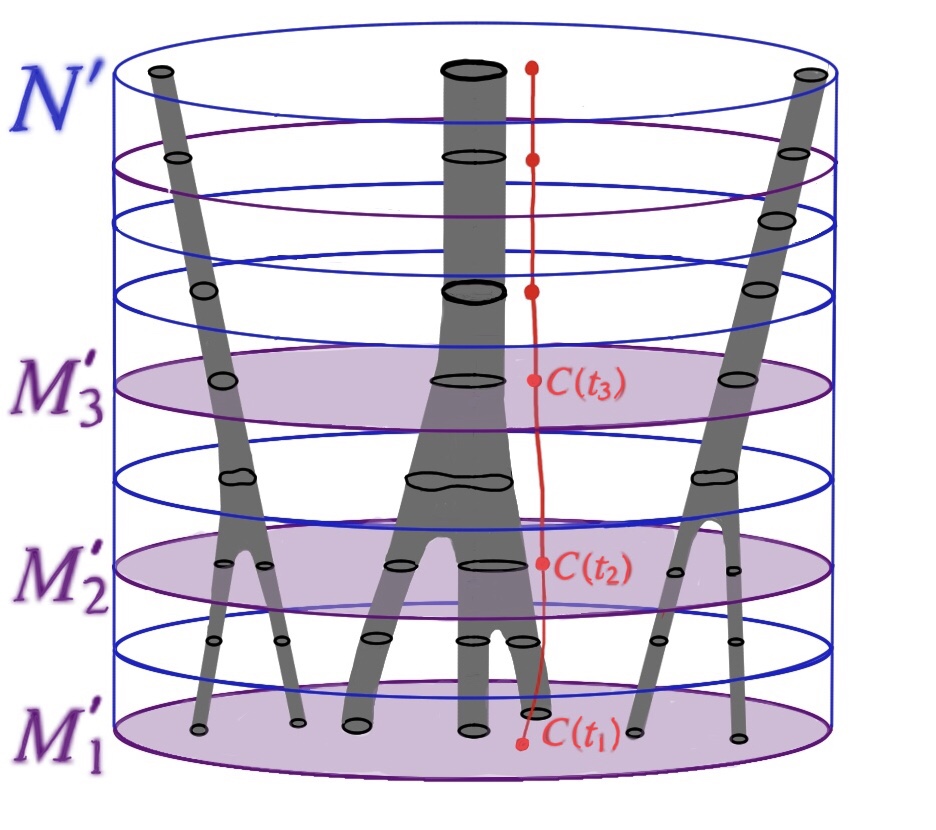}
   \includegraphics[height=1.4in]{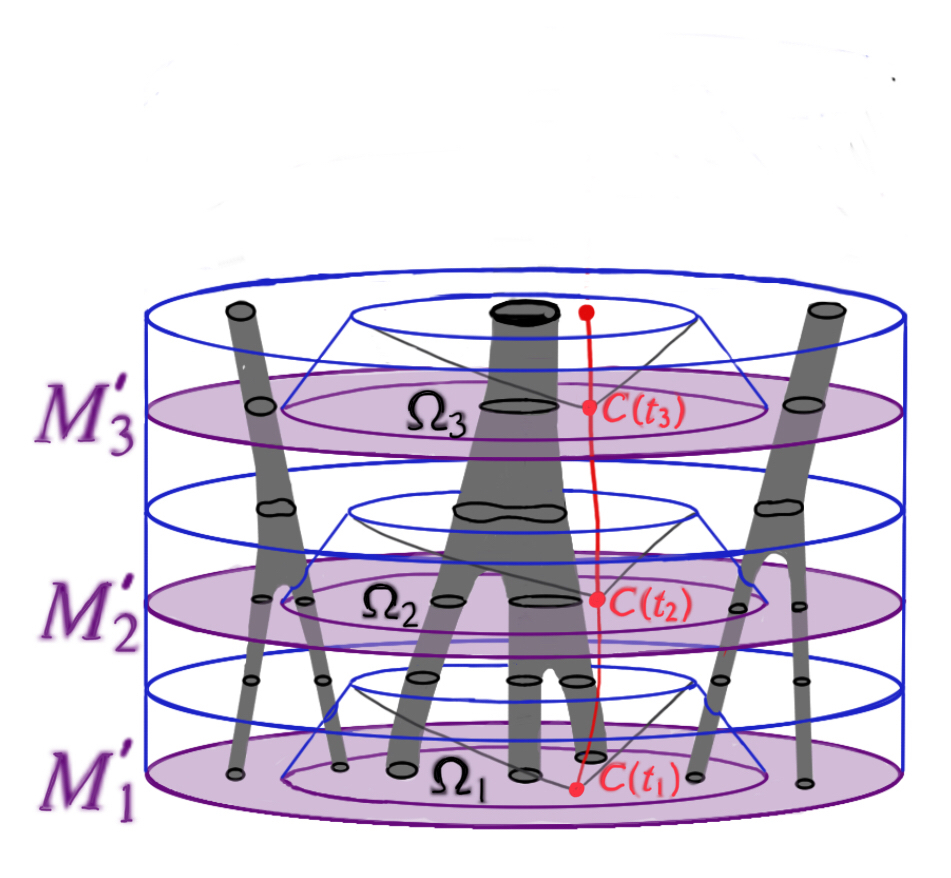}
  \,\, \includegraphics[height=1.4in]{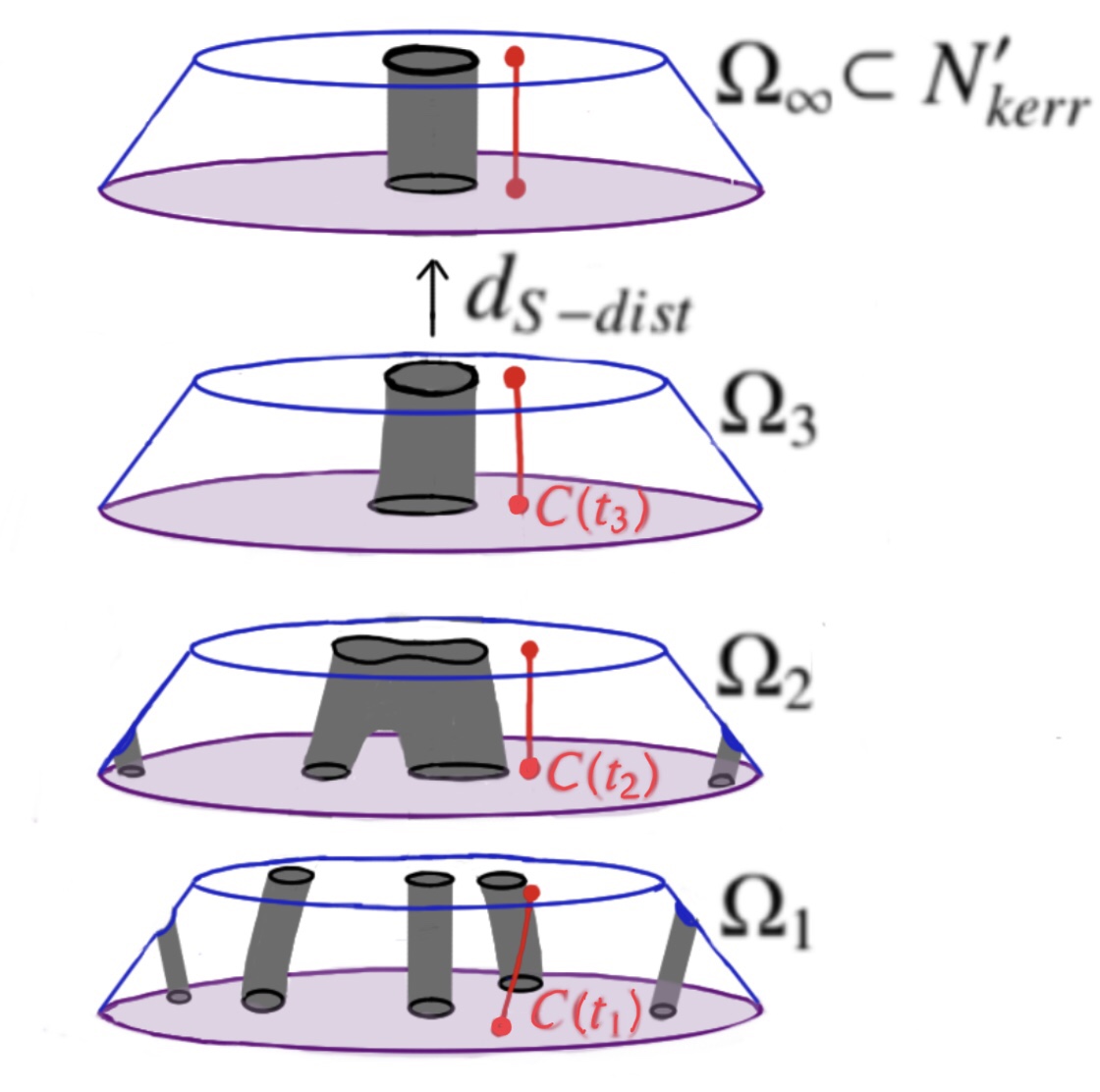}\,\,
\caption{On the left we see a space-time, $N'$,
   with black holes merging as time evolves forward towards the prediction of the Final State Conjecture (that there will be finitely many black holes moving away from one another) with a causal curve, $C$, following one of the black holes centering our view as in Conjecture~\ref{conj:final-state}.
   In the center, we see the 
the cosmic strips, $N'_{j,0,\tau_0}\subset N'$ and regions, $\Omega_j=\Omega_j(R_i)\subset N'_{j,0,\tau_0}$, from the pointed exhaustion centered on $C(t_j)$ for fixed $R_i$.  On the right, we see the 
$\Omega_j=\Omega_j(R_i)$ viewed as intrinsic 
 space-times converging upward to their limit
 $\Omega_\infty=\Omega_\infty(R_i)\subset N'_{Kerr,0,\tau_0}$
   as in Remark~\ref{rmrk:follow-one}.
   }\label{fig:final-state}
\end{figure}

\begin{rmrk}\label{rmrk:follow-one}
Note that in the above conjecture we only mention one black hole in the limit.  This is because we are taking a pointed convergence following a causal curve, $C$, along the horizon of just one black hole towards the final state.   In the Final State Conjecture the other black holes are moving away.  So for any fixed $R_i$, the other black holes are eventually outside of 
causally-null-compactifiable regions, $\Omega_j(R_i)\subset N_j'$, that are defined by the pointed exhaustion
in Remark~\ref{rmrk:pt-exhaustion}.  When taking a pointed limit as in Remark~\ref{rmrk:exhausted-conv},
we fix the $R_i$ and take $j\to \infty$
so that 
\be
\Omega_j(R_i)\to  \Omega(R_\infty)\subset N'_{Kerr,+}.
\ee
See the right side of Figure~\ref{fig:final-state}. If we follow a different black hole with the curve, $C$, we will see a different limit describing its destiny.
\end{rmrk}

\begin{rmrk}\label{rmrk-pted-choice}
For simplicity, we chose the sequence of points
$C(t_j)\subset M_j \subset N$ in the horizon boundary, $\partial N'$ in Conjecture`\ref{conj:final-state} to follow one black hole into the future.   However 
one might choose, $C(t_j)\subset M_j$, near a black hole or a gravity well but not on the horizon boundary as discussed by Lee-Sormani in \cite{LeeSormani1}.   Note if $C(t_j)$ are not near any black holes as $t_j\to \infty$ then 
$N_{j,0,\tau_0}=\tau_j^{-1}(0,\tau_0)\subset N'_{j,+}\subset N'$
will converge using a pointed exhaustion based at $C(t_j)$
to Minkowski (which is a special case of Kerr).
\end{rmrk}

\begin{rmrk}\label{rmrk:well-chosen}
It is unknown exactly what these ``well-chosen'' sequences of space-like asymptotically flat, $M'_j\subset N'$,
diverging to the future should be.   If the Final State Conjecture is correct, these surfaces are converging in the pointed sense to a maximal surface in Kerr.   
\end{rmrk}

If we add the hypothesis that we are starting with a time-symmetric asymptotically flat initial data set, then we are in a scenario considered by Penrose \cite{Penrose-1973}, see also the surveys by Mars \cite{Mars-overview-Penrose}\cite{Mars-CQG}, where it was suggested that the future development of such a space-time evolves towards a single black hole. See the left side of Figure~\ref{fig:Penrose-to-Kerr}.

\begin{figure}[b] 
   \centering
   \includegraphics[height=1.5in]{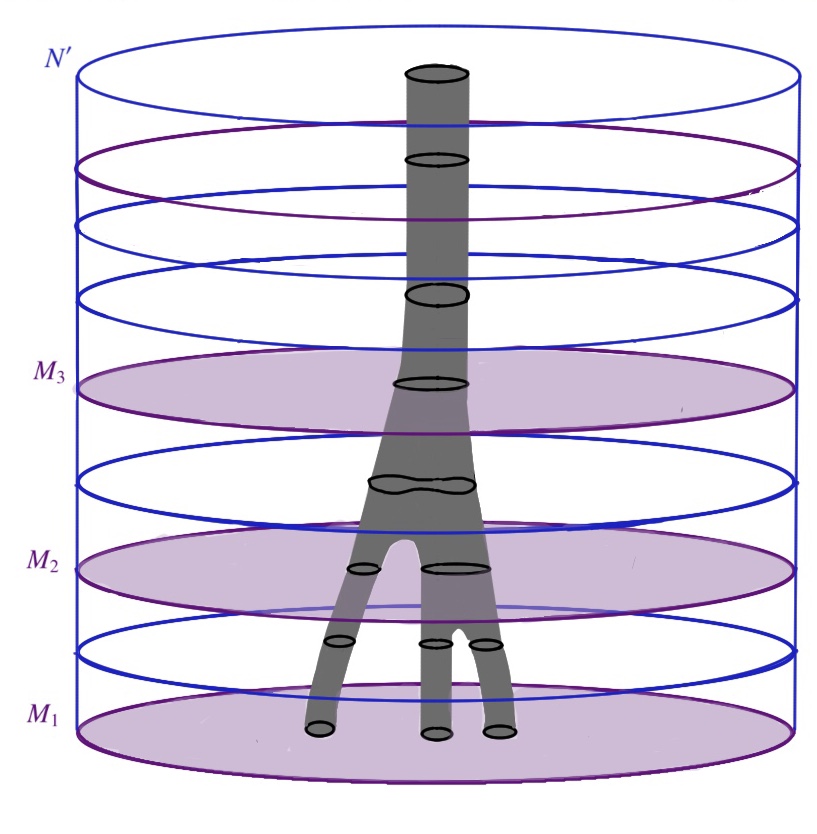}
   \includegraphics[height=1.5in]{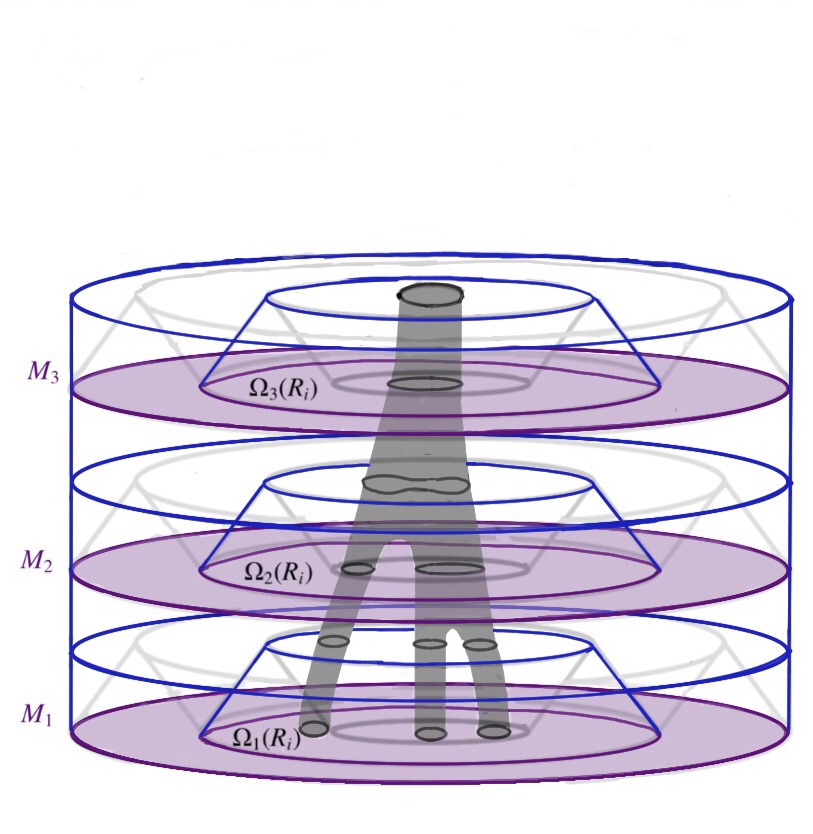}
   \includegraphics[height=1.5in]{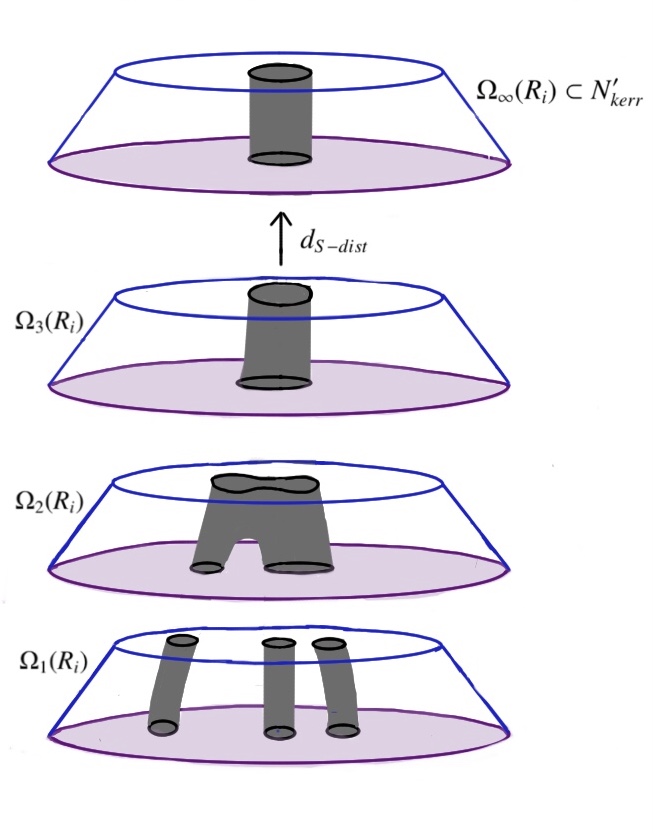}
   \caption{This depicts Conjecture~\ref{conj:Penrose-to-Kerr}.
   On the left we see a space-time, $N'$,
   with black holes merging into a single black hole.
In the center, we see the 
the cosmic strips, $N'_{j,0,\tau_0}\subset N'$ and STCMC regions, $\Omega_j=\Omega_j(R_i)\subset N'_{j,0,\tau_0}$, for fixed $R_i$.  On the right, we see the 
$\Omega_j=\Omega_j(R_i)$ viewed as intrinsic 
 space-times converging upward to their limit
 $\Omega_\infty=\Omega_\infty(R_i)\subset N'_{Kerr,0,\tau_0}\subset N'_{Kerr}$.
   }\label{fig:Penrose-to-Kerr}
\end{figure}

Since there is only one limiting black hole in this setting, we might use the Cederbaum-Sakovich STCMC exhaustions as in Remark~\ref{rmrk:CS-exhaustion} 
because they naturally center the evolving central black hole.  See Figure~\ref{fig:Penrose-to-Kerr}.
We can then restate this step of Penrose's argument
as the following conjecture:

\begin{conj}\label{conj:Penrose-to-Kerr}
Suppose $N'$ is
the exterior region (outside the event horizon) in
an asymptotically flat future
developed space-time, $(N,g)$,  satisfying Einstein's 
equations with a time-symmetric initial data set, $(M,h,0)$. 
Consider a sequence of well chosen
space-like surfaces, $M_j\subset N'$,
diverging to the future
as in Remark~\ref{rmrk:well-chosen}. 
and 
$
N'_{j,+}=J^+(M_j)\subset N
$
viewed as the exterior future development of $(M_j, h_j, \kappa_j)$.  Let 
$
\tau_j:N'_{j,+}\to [0,\infty)
$
the Cosmological time of $N'_{j,+}$.  
Then for any $\tau_0>0$,
the cosmic strips,
\be
N'_{j,0,\tau_0}=\tau_j^{-1}(0,\tau_0)\subset N'_{j,+}\subset N',
\ee
converge in an intrinsic future developed sense
(using a Cederbaum-Sakovich STCMC exhaustion
by causally-null-compactifiable sub-space-times, 
$\Omega_j(R_i)\subset N'_{j,0,\tau_0}$,
as in Remark~\ref{rmrk:CS-exhaustion}
and Remark~\ref{rmrk:exhausted-conv}) 
to the cosmic strip
\be
N'_{Kerr,0,\tau_0}=\tau_{g_{Kerr}}^{-1}(0,\tau_0)\subset N'_{Kerr,+}\subset N'_{Kerr},
\ee
where
$(N'_{Kerr},g_{Kerr})$ is the exterior of an asymptotically flat Kerr space-time.   We conjecture that among the
 various intrinsic future developed distances between space-times 
 defined in Section~\ref{sect:defns-FD}, the
  intrinsic flat and metric measure 
  notions,
  $d_{S-FD-FF}$ or
  $d_{S-FD-mm}$, may work better
  than $d_{S-FD-HH}$.  See Figure~\ref{fig:Penrose-to-Kerr}.
  \end{conj}

This conjecture is related to a key step in Penrose's original argument for proving the Penrose Inequality.
In his famous heuristic argument, Penrose considers 
the exterior, $N'$, of the future maximal development, $(N,g)$, of the initial state $(M,h,k=0)$ and concludes that
\be
\mass_{ADM}(M,h)\ge \frac{1}{2} \left(\frac{\area_h(\partial M')}{\omega_{n-1}}\right)^{\frac{n-2}{n-1}}
\ee  
where $\partial M'$ in his argument is the intersection
of $M$ with the event horizon, $\partial N'$.  He refers to the Final State Conjecture to claim that this space-time
evolves towards Kerr space-time as time goes to infinity
which can be rephrased as in Conjecture~\ref{conj:Penrose-to-Kerr}.

\begin{rmrk}\label{rmrk:Riem-Penrose}
Penrose's heuristic argument is not used in the mathematical proofs of the Penrose Inequality by Huisken-Ilmanen or Bray in 
\cite{Huisken-Ilmanen} \cite{Bray-Penrose}.   In their setting the surface, $\partial M$, is defined to be the outermost minimizing surface in the asymptotically flat Riemannian manifold, $M$, and there is no reference to its future development $N$ in their proofs. This is because one cannot define the event horizon of a black hole in terms of initial data $(M,h,k)$ for the Einstein equations.  However, heuristically, the area of the outermost minimizing surface, $\partial M$, is believed to be smaller than the area of $\partial M'$ considered by Penrose in his setting.
\end{rmrk}

It is well worth deriving a mathematical proof of Penrose's original argument.  We can at least restate each of his steps just as we have restated the Final State Conjecute above.

An important ingredient in the argument by Penrose is the Black Hole Area Law proven initially by Hawking (\cite{Hawking-Ellis}  and in more generality by Chrusciel-Delay-Galloway-Howard \cite{CDGH-area}.
 This law states that the total area of the  event horizon $\partial N'$ of a black hole space-time $(N',h)$ is a non-decreasing function of time.
 In particular:
 \be
 \area_{h}(\partial N' \cap \tau^{-1}(\tau_j))\le \area_{h}(\partial N' \cap \tau^{-1}(\tau_j))
 \textrm{ for }\tau_j\le \tau_{j+1}.
 \ee
 This is proven to be true at least when the level sets of $\tau$ are $C^2$ in \cite{CDGH-area}
 and might be true more generally.  
 
 In the argument of Penrose, he claims that the area increases to the 
 area of the horizon in the limit.  This means he
 believes there is at least semi-continuity of the area under convergence to that limit.
 So we restate this as the following conjecture:

\begin{conj}\label{conj:Penrose-Area}
Under the hypotheses of Conjecture~\ref{conj:Penrose-to-Kerr},
with 
\be
M'_j=\tau^{-1}(\tau_j)\subset N'
\ee
where $\tau_j \to \infty$,
and $h_j$ the induced metric from $g_j$,
we have the $\mathcal{VF}$ convergence
\be
(M'_j,h_j)\to (M'_\infty,h_\infty)
\ee
and
\be
\limsup_{j\to \infty} Area_{h_j}(\partial M'_j)
\le Area_{h_\infty}(\partial M'_\infty).
\ee
\end{conj}

\begin{rmrk}
Note that for intrinsic flat convergence, the areas of boundaries are usually lower semicontinuous, not
upper semicontinuous, so it would
be better to use the stronger properties of $\mathcal{VF}$ convergence developed in the work of 
Portegies, Jauregui-Lee,
and Jauregui-Lee-Perales in \cite{Portegies-evalues}  \cite{Jauregui-Lee-ADM-F}
\cite{Jauregui-Lee-Perales}
combined with the Slicing Theorem
of Ambrosio-Kirchheim \cite{AK} and
results in Portegies-Sormani \cite{Sormani-Portegies-prop}.
\end{rmrk}

To complete Penrose's arguement, one uses the above conjectures and then the fact that Penrose's Inequality holds on Kerr, so that
\be
\frac{1}{2} \left(\frac{\area_h(\partial M)}{\omega_{n-1}}\right)^{\frac{n-2}{n-1}}\le
\frac{1}{2} \left(\frac{\area_{Kerr}(\partial M)}{\omega_{n-1}}\right)^{\frac{n-2}{n-1}}
= \mass(M_{Kerr}) \le 
\mass(M,h).
\ee
The final inequality above is then justified by
the fact that 
$\mass(M_j,h_j,k_j)$ should be nonincreasing
\footnote{Justication by physicists as to why mass is nonincreasing can be found in the survey by Mars \cite{Mars-CQG}.}
starting from $\mass(M,h)$ to the
limit $\mass(M_{Kerr})$.
For this argument to hold one needs 
at least semicontinuity of the mass under convergence to the limit.  This can be phrased as follows: 

\begin{conj}\label{conj:Penrose-mass}
Under the hypotheses of 
Conjecture~\ref{conj:Penrose-Area},
the mass is lower semicontinuous:
\be \label{eq:Penrose-mass}
\liminf_{j\to \infty} \mass(M_j,h_j)
\ge \mass(M_\infty,h_\infty). 
\ee
\end{conj}

\begin{rmrk}
Note that 
when a sequence of asymptotically flat Riemannian manifolds $(M_j,h_j)$
converges to $(M_\infty,h_\infty)$ in the $\mathcal{VF}$ sense, 
Jauregui-Lee proved (\ref{eq:Penrose-mass})
using Huisken's isoperimetric mass
\cite{Jauregui-Lee-ADM-F}.   Their arguments should probably work for $mm$ convergences as well.
\end{rmrk}

 Returning to the Final State Conjecture,
 as we stated it in Conjecture~\ref{conj:final-state}, we have two final remarks:

\begin{rmrk}\label{rmrk:Final-State-Limits}
One of the main challenges towards proving the Final State Conjecture is to prove that there is a convergence (in any suitable sense) to a limit space-time that is a stationary solution of the Einstein equations. The advantage of using an intrinsic distance between space-times to describe this convergence is that one can approach this challenge using the theory of Cauchy sequences.  Under the hypotheses of Conjecture~\ref{conj:Penrose-to-Kerr}
, 
one could approach the problem by first proving that a sequence of cosmic strips is a Cauchy sequence with respect to the intrinsic timed $mm$ or $FF$ distance and then arguing that the limit, which is a timed-metric-space endowed with a measure or current structure, is stationary.
Part of the challenge in this problem would be to define what it means for a timed-metric-space to be stationary.
\end{rmrk}

\begin{rmrk}\label{rmrk:No-Hair}  
Once one has proven the limit space is a stationary space-time, then the Final State Conjecture is reduced to
proving that Kerr space-time is the unique stationary solution to the Einstein vacuum equations. 
This is called the No Hair Conjecture.
See, for example, work by Chrusciel
\cite{Chrusciel-no-hair}
and Alexakis-Ionescu-Klainerman
\cite{Alexakis-Ionescu-Klainerman-Rigidity}
and others surveyed by Giorgi in \cite{Giorgi-BAMS}. 
 To apply the ideas in Remark~\ref{rmrk:Final-State-Limits}, one could try to prove a No Hair Theorem for the limiting timed-metric-spaces using a weak definition of the Einstein Vacuum Equation.
 Such weak definitions have been explored using measures and optimal transport methods.
 See Remark~\ref{rmrk:measure-space-time}.
 Perales and Mondino are working to relate 
 timed metric measure spaces achieved as our intrinsic $mm$ limits to the metric measure space-times studied by Mondino-Suhr \cite{Mondino-Suhr} which satisfy an optimal transport formulation of the Einstein vacuum equations.  
\end{rmrk}

\subsection{{\bf Final Remarks}}
Throughout this paper we have remarked on further directions and stated conjectures.
There are
partial results on many of these problems that have already been completed (as described within the remarks) that we hope will appear soon.  We welcome anyone to work on the conjectures, but please let us know so we can coordinate efforts and form teams.  Ideally we can come together at workshops either in person or online.   For ease of communication, we chose names and notations for our notions that are distinct from existing ones.  We review all this notation, old and new, in the index below. 

\begin{rmrk}
After presenting the results in this paper, a few people suggested they might be interested in investigating intrinsic timed distances for causally-null-compactifiable space-times using other canonical time functions.   In many places within this paper we have directly used the fact that our time functions are regular cosmological time functions.  Not all our theorems hold for other time functions. So please proceed with caution.  If you do write a paper with other time functions and wish to distinguish between your work and ours, you might refer to our space-times as ``causally null compactifiable via the cosmological time function" 
and our intrinsic distances as ``intrinsic distances defined via null compactification with respect to cosmological time".  So for example, we have defined the ``intrinsic timed Hausdorff distance
defined via null compactification with respect to cosmological time". We ask that everyone please be careful not to cause confusion when proving things with other time functions. 
\end{rmrk}

\printindex

\newpage
\bibliographystyle{plain}
\bibliography{SF-bib}
\end{document}